\documentclass[11pt,DIV=12,a4paper,oneside,headings=small,abstract=on]{scrartcl}
\usepackage[T1]{fontenc}
\usepackage[utf8]{inputenc}
\usepackage[english]{babel}

\usepackage{amssymb,amsmath,amsfonts,amsthm,expdlist,enumerate,tikz}
\usepackage[noblocks]{authblk}

\newcommand{\ball}[2]{B (#1 , #2)}
\newcommand{\ballc}[1]{B (#1)}
\newcommand{\Elliptic}{\vartheta_{\mathrm{E}}}
\newcommand{\EllipticU}{\vartheta_{\mathrm{E}, +}}
\newcommand{\EllipticLower}{\vartheta_{\mathrm{E}, -}}
\newcommand{\Lipschitz}{\vartheta_{\mathrm{L}}}
\newcommand{\euler}{\mathrm{e}}
\newcommand{\RR}{\mathbb{R}}
\newcommand{\CC}{\mathbb{C}}
\newcommand{\NN}{\mathbb{N}}	
\newcommand{\ZZ}{\mathbb{Z}}
\newcommand{\EE}{\mathbb{E}}
\newcommand{\PP}{\mathbb{P}}
\newcommand{\T}{\mathrm{T}}
\newcommand{\cD}{\mathcal{D}}
\newcommand{\cJ}{\mathcal{J}}
\newcommand{\cL}{\mathcal{L}}
\newcommand{\cS}{\mathcal{S}}
\newcommand{\drm}{\mathrm{d}}
\newcommand{\ess}{\mathrm{ess}}
\newcommand{\D}{\mathfrak{D}}
\newcommand{\1}{\mathbf{1}}

\newcommand{\au}{\alpha_1}
\newcommand{\ao}{\alpha_2}
\newcommand{\bu}{\beta_1}
\newcommand{\bo}{\beta_2}

\newcommand{\sfUC}{\mathrm{sfUC}}
\newcommand{\Cac}{F}
\newcommand{\Op}{\mathcal{H}}
\newcommand{\diver}{\operatorname{div}}
\newcommand{\supp}{\operatorname{supp}}
\newcommand{\Ran}{\operatorname{Ran}}
\newcommand{\equi}{\cS}
\newcommand{\Dom}{\mathcal{D}}
\newcommand{\HEins}{H^1}
\newcommand{\HZwei}{H^2}
\newcommand{\HNullEins}{H_0^1}
\newcommand{\BIGOP}[1]{\mathop{\mathchoice%
{\raise-0.22em\hbox{\huge $#1$}}%
{\raise-0.05em\hbox{\Large $#1$}}{\hbox{\large $#1$}}{#1}}}
\newcommand{\BIGboxplus}{\mathop{\mathchoice%
{\raise-0.35em\hbox{\huge $\boxplus$}}%
{\raise-0.15em\hbox{\Large $\boxplus$}}{\hbox{\large $\boxplus$}}{\boxplus}}}
\newcommand{\bigtimes}{\BIGOP{\times}}
\renewcommand{\epsilon}{\varepsilon}
\newtheorem{theorem}{Theorem}[section]
\newtheorem{lemma}[theorem]{Lemma}
\newtheorem{corollary}[theorem]{Corollary}

\theoremstyle{definition}
\newtheorem{definition}[theorem]{Definition}
\theoremstyle{remark}
\newtheorem{remark}[theorem]{Remark}
\newtheorem{example}[theorem]{Example}

\usepackage[linktocpage, pdftex,colorlinks]{hyperref}
	\usepackage{color}
 	\definecolor{darkred}{rgb}{0.5,0,0}
 	\definecolor{darkgreen}{rgb}{0,0.5,0}
 	\definecolor{darkblue}{rgb}{0,0,0.5}
 	\hypersetup{colorlinks,linkcolor=darkblue,filecolor=darkgreen,urlcolor=darkred,citecolor=darkblue}
\hypersetup{
pdftitle={Sampling and equidistribution theorems for elliptic second order operators, lifting of eigenvalues, and applications},
pdfauthor={Ivan Veseli\'c, Martin Tautenhahn},
}
\begin{document}
%
%
%
%
%
%
\title{Sampling and equidistribution theorems for elliptic second order operators, lifting of eigenvalues, and applications}
\author[1]{Martin Tautenhahn}
\author[2]{Ivan Veseli\'c}
\affil[1]{Universit\"at Leipzig, Fakult\"at f\"ur Mathematik und Informatik, 04109 Leipzig, Germany}
\affil[2]{Technische Universit\"at Dortmund, Fakult\"at f\"ur Mathematik, 44121 Dortmund, Germany}
%
\date{\vspace{-2.9em}}
%
\maketitle
\begin{abstract}
We consider elliptic second order partial differential operators with Lipschitz continuous leading order coefficients on finite cubes and the whole Euclidean space.
We prove quantitative sampling and equidistribution theorems for eigenfunctions.
The estimates are scale-free, in the sense that for a sequence of growing cubes we obtain uniform estimates.
These results are applied to prove lifting of eigenvalues as well as the infimum of the essential spectrum, and an uncertainty relation (aka spectral  inequality) for short energy interval spectral projectors.
Several application including random operators are discussed.
In the proof we have to overcome several challenges posed by the variable coefficients of the leading term.
\\[1ex]
\end{abstract}

This paper is a correction of the publication \cite{TautenhahnV-20}.
We are grateful to Alexander Dicke, who has pointed out to us an error in one of the proofs of the original publication.
Here we present a corrected version. Only Sections~\ref{sec:three-annuli}
and \ref{sec:chaining} needed to be modified.
In Section~\ref{sec:three-annuli}, changes concern only the statement of properties of the constants $\Theta_1, \Theta_3,D_1, D_2, D_3$,
and the addition of the new Corollary \ref{thm:three_annuli_cubes}.
The main changes appear in Section~\ref{sec:chaining}, where several theorems and their proofs are modified.

In particular, all the results formulated in Sections 2 and 3 of \cite{TautenhahnV-20} are correct as stated there.

We have not updated the references and the discussion,
hence they reflects the state of the art at the time (late 2019) when the final version of \cite{TautenhahnV-20} was submitted to the journal and not
at the time when the (corrected) manuscript at hand was uploaded to arXiv.

We also thank Thomas Kalmes for helpful discussions.
%
%
%
%
%
%
\tableofcontents
%
%
%
%
%
%
\section{Introduction}
\emph{Scale-free unique continuation estimates} play an important role in
mathematical models of condensed matter and other structures with multiple length scales which are described by partial differential equations.
\par
They compare the $L^2$-norm of an eigenfunction
on the full domain to the $L^2$-norm on balls distributed evenly throughout the domain.
For this reason we call these scale-free unique continuation estimates
also \emph{sampling} or \emph{equidistribution} theorems, depending on whether the full domain is equal to $\RR^d$ or to a finite (but typically large) cube.
In the latter case, the bounds we present are independent of the size of the cube and are for this reason called \emph{scale-free}.
\par
They are quantitative geometric cousins of unique continuation principles,
which have been developed to study vanishing
orders of eigenfunctions \cite{DonnellyF-88},
absence of eigenvalues embedded in the continuous spectrum
\cite{JerisonK-85}, \cite{IonescuJ-03}, \cite{KochT-06},
limiting absorption principles \cite{Enblom-15},
solutions obeying the Sommerfeld radiation condition \cite{Zubeldia-14},
observability estimates \cite{LebeauR-95}, \cite{RousseauL-12},
inverse problems \cite{FursikovI-96}
and others.
\par
A primary field of applications of scale-free equidistribution theorems is the theory of random Schr{\"o}dinger operators. The importance of Carleman estimates in this field was first realized in \cite{BourgainK-05}.
They have been used in \cite{RojasMolinaV-13} to prove a new scale-free unique continuation principle
for random Schr{\"o}dinger operators, and conclude Anderson localization for Delone-Anderson Hamiltonians.
These results were strengthened and their applicability extended in \cite{Klein-13}.
The best possible scale-free equidistribution theorems valid for
the negative Laplacian plus a bounded potential accessible with Carleman estimates
were established in \cite{NakicTTV-18,NakicTTV-18-arxiv} and \cite{TaeuferT-17,Taeufer-18}.
They apply to functions in the range of a spectral projector
(or a sufficiently fast decaying function, respectively)
of a Schr{\"o}dinger operator associated with any compact energy interval.
Based on \cite{BorisovTV-17}, in \cite{TaeuferT-18} the results of
\cite{RojasMolinaV-13} and \cite{Klein-13} have been extended to the
physical situation where a bounded electromagnetic potential is present.
In the complementary situation of constant magnetic field,
and thus unbounded magnetic vector potential,
scale-free unique continuation estimates
have been established in \cite{CombesHKR-04} under a periodicity assumption.
These results crucially relied on explicit estimates on eigenfunctions of the Landau Hamiltonian
derived in \cite{RaikovW-02b}, and were later adapted for other problems in \cite{GerminetKS-07},
\cite{RojasMolina-12}, and \cite{TaeuferV-16}.
\par
In the context of control theory such estimates
sometimes bear the name of \emph{spectral inequalities}.
For domains with a multi-scale structure the sampling and equidistribution theorems
proved in \cite{EgidiV-16-arxiv,EgidiV-18}, \cite{NakicTTV-18}, and \cite{LebeauM} allow to derive
null-controllability of the heat equation with explicit estimates on the control cost,
see also \cite{WangWZZ,EgidiNSTTV,NakicTTV-control}.
\par
The purpose of the present paper is to generalize the above discussed
scale-free unique continuation estimates to elliptic second order operators.
This is clearly of interest in order to extend the control theory for the heat equation
to more general dissipative evolutions. In the context of Schr{\"o}dinger operators
one encounters such general elliptic operators as effective Hamiltonians
resulting from reduction procedures.
\par
Many methods developed for Schr{\"o}dinger operators can be
adapted to general elliptic second order operators.
However, in our situation the variable coefficient functions of the leading term pose challenges.
For this reason, we were in the prequel paper \cite{BorisovTV-17} only able to treat leading second order terms with
slowly varying coefficients.
In the proofs of the present paper we need additionally new versions of the interpolation inequality and the chaining argument which are adapted to the spectral and geometric situation.
This allows us to complete the argument for arbitrary Lipschitz coefficient functions.
\par
To illustrate the usefulness of our results we discuss several applications
in Section~\ref{sec:applications}.
In particular, we present a lifting estimate for discrete eigenvalues and the minimum of the essential spectrum under a semidefinite potential perturbation,
an uncertainty relation for spectral projectors on short energy intervals, and a coefficient-independent spectral inequality for low energies.
This is then applied to a homogenization scenario and to Wegner estimates for random Schr{\"o}dinger operators.
Some of these topics will be developed fully in a subsequent project.
\par
The paper is structured as follows.
In the following section we formulate our two main results: A sampling theorem valid
on $\RR^d$ and a scale-free equidistribution theorem for cubes.
The sketch of some applications follows in Section~\ref{sec:applications}.
In Section~\ref{sec:three-annuli} the first step of the proof is performed yielding a
\emph{three annuli inequality} tailored to our setting. The following Section~\ref{sec:intermezzo} is an intermezzo:
We give a short proof of our main result in the case of the pure Laplacian.
This enables us to discuss the difference between elliptic second order
operators with slowly and quickly varying coefficient functions.
The proof of the sampling and equidistribution theorems in the general case
are completed in Section~\ref{sec:chaining}. Some technical aspects are deferred to an appendix,
including the explicit estimation of constants and the construction of an extension.
%
%
%
%
%
%
\section{Notation and main results}
\label{sec:results}
Let $d \in \NN$ and consider an operator $\Op \colon C_{\mathrm{c}}^\infty (\RR^d) \to L^2 (\RR^d)$,
 \[
\Op u := - \diver (A \nabla u) + b^\T \nabla u + c u = -\sum_{i,j=1}^d \partial_i \left( a^{ij} \partial_j u \right) + \sum_{i=1}^d b^i \partial_i u + c u ,
\]
where  $A \colon \RR^d \to \RR^{d \times d}$
with $A = (a^{ij})_{i,j=1}^d$, $b : \RR^d \to \CC^d$, $c : \RR^d \to \CC$, and $\partial_i$ denotes the $i$-th weak derivative.
We assume that $a^{ij} \equiv a^{ji}$ for all $i,j \in \{1,\ldots , d\}$, and that there are constants $\Elliptic \geq 1$ and $\Lipschitz \geq 0$ such that for
all $x,y \in \RR^d$ and all $\xi \in \RR^d$ we have
\begin{equation} \label{eq:elliptic}
\Elliptic^{-1} \lvert \xi \rvert^2 \leq \xi^\T A (x) \xi \leq \Elliptic \lvert \xi \rvert^2
\quad\text{and}\quad
\lVert A(x) - A(y) \rVert_\infty \leq \Lipschitz \lvert x-y \rvert .
\end{equation}
Here we denote by $\lvert z \rvert$ the Euclidean norm of $z \in \CC^d$,
and by $\lVert M \rVert_\infty$ the row sum norm of a matrix $M \in \CC^{d\times d}$.
Moreover, we assume that $b \in L^\infty (\RR^d ; \CC^d)$ and $c \in L^\infty (\RR^d)$.
We denote the form associated to $\Op$ by $a_0$, i.e.\ $a_0 : C_{\mathrm{c}}^\infty (\RR^d) \times C_{\mathrm{c}}^\infty (\RR^d) \to \CC$, $a_0 (u,v) = \langle \Op u , v \rangle$.
The operator $\Op$ as well as the form $a_0$ are densely defined and sectorial, i.e.\ their numerical ranges are contained in a sector of the form
 \[
S_{\lambda_0,\theta} = \{\lambda \in \CC \colon \lvert \Im \lambda \rvert \leq \tan \theta (\Re \lambda - \lambda_0) \}
  = \{\lambda \in \CC \colon \operatorname{arg} (\lambda - \lambda_0) \leq \theta\}
  \]
for some $\lambda_0 \in \RR$, and $\theta \in [0,\pi / 2)$, see Section 11 in \cite{Schmuedgen-12}.
\par
Let $H : \Dom (H) \subset L^2 (\RR^d) \to L^2 (\RR^d)$ be the Friedrichs extension of $\Op$,
see e.g.\ page 325 ff.\ in \cite{Kato-80}.
More precisely, since $\Op$ is densely defined and sectorial,
the form $a_0$ is closable. Its closure is given by  $a : \HEins (\RR^d) \times \HEins (\RR^d) \to \CC$,
\[
 a (u,v) = \int_{\RR^d} \left( \sum_{i,j=1}^d a^{ij} \partial_j u \overline{\partial_i v} + \sum_{i=1}^d b^i \partial_i u \overline{v} + c u \overline{v} \right) \drm x .
\]
The form $a$ is densely defined, closed, and m-sectorial.
According to the first representation theorem (cf.~Theorem~2.1 in Chapter~VI of \cite{Kato-80})
there is a unique m-sectorial operator associated with the form $a$, which we denote by $H$.
We have that $C_{\mathrm c}^\infty (\RR^d)$ is an operator core for $H$, i.e.\ $C_{\mathrm c}^\infty (\RR^d)$ is dense in the domain of $H$ with respect to the graph norm
 \[
  \lVert f \rVert_H = \lVert f \rVert_{L^2 (\RR^d)} + \lVert Hf \rVert_{L^2 (\RR^d)},
  \quad
  f \in \Dom (H).
 \]
 This can be seen in the following way: Since the Friedrichs extension $H$ is an m-sectorial operator, its negative generates a $C_0$-semigroup, see e.g.\ Theorem~3.22 in \cite{ArendtCSVV-15}.
 (Beware that \cite{ArendtCSVV-15} has a slightly different terminology compared to \cite{Kato-80,Schmuedgen-12}.)
 Moreover, the first assertion of Theorem~2.3 in \cite{Eberle-99} establishes that the closure of $-\Op$ generates a $C_0$-semigroup as well. Note that \cite{Eberle-99} considers operators $\Op$ with $c = 0$. However, bounded perturbations do not affect the property to generate a $C_0$-semigroup.
 Since $H$ extends $\Op$, and since both $-H$ and the closure of $-\Op$ generate a $C_0$-semigroup, Theorem~1.2 in \cite{Eberle-99} implies that $C_{\mathrm{c}}^\infty (\RR^d)$ is an operator core for $H$.
\par
 Hence, for all $u \in \Dom (H)$ there is a sequence $(u_n)_{n \in \NN}$ in $C_{\mathrm{c}}^\infty (\RR^d)$ such that
 \begin{align} \label{eq:core}
  u_n   &\to u, \quad\text{and}\quad
  H u_n \to H u
  \quad\text{in} \ L^2 (\RR^d) .
 \end{align}
For $L,\rho > 0$ we denote by $\Lambda_L = (-L/2 , L/2)^d$ the open centered cube of side length $L$, and by $\ballc{\rho} = \{ y \in \RR^d \colon \lvert y \rvert < \rho \}$ the centered ball with radius $\rho$. If $x \in \RR^d$ we denote by $\Lambda_L (x) = \Lambda_L + x$ and $\ball{\rho}{x} = \ballc{\rho} + x$ its translates. For $\Omega \subset \mathbb{R}^d$ open and $\psi \in L^2 (\Omega)$ we denote by $\lVert \psi \rVert = \lVert \psi \rVert_\Omega$ the usual $L^2$-norm of $\psi$. If $\Gamma \subset \Omega$ we use the notation $\lVert \psi \rVert_\Gamma = \lVert \chi_\Gamma \psi \rVert_\Omega$.
\begin{definition}
 Let $G > 0$ and $\delta > 0$. We say that a sequence $Z = (z_j)_{j \in (G\ZZ)^d} \subset \RR^d$ is \emph{$(G,\delta)$-equidistributed}, if
 \[
  \forall j \in (G \ZZ)^d \colon \quad  \ball{\delta}{z_j} \subset \Lambda_G (j) .
\]
Corresponding to a $(G,\delta)$-equidistributed sequence $Z$, we define for $L > 0$ the sets
\[
\equi_{\delta,Z} = \bigcup_{j \in (G \ZZ)^d } \ball{\delta}{z_j} \subset \RR^d \quad \text{and} \quad
\equi_{\delta,Z}(L) = \bigcup_{j \in (G \ZZ)^d } \ball{\delta}{z_j} \cap \Lambda_L \subset \Lambda_L .
\]
Note that we suppress the dependence of $\equi_{\delta,Z}$ and $\equi_{\delta , Z} (L)$  on $G$.
\end{definition}
To point out the main technical advancement of the present paper we cite the following theorem from \cite{BorisovTV-17}.
\begin{theorem}[\cite{BorisovTV-17}]\label{thm:BTV}
Let $G > 0$ and assume
\begin{equation} \label{ass:sampling}
 \epsilon := 1 - 33 \mathrm{e} d (\sqrt{d} + 2) \Elliptic^{6} G \Lipschitz   > 0.
\end{equation}
Then for all measurable and bounded $V \colon \mathbb{R}^d \to \mathbb{R}$, all $\psi \in \HZwei (\mathbb{R}^d)$ and $\zeta \in L^2 (\RR^d)$ satisfying $\lvert H \psi \rvert \leq \lvert V\psi \rvert + \lvert \zeta \rvert$ almost everywhere on $\mathbb{R}^d$, all $\delta \in (0,G/2)$ and all $(G,\delta)$-equi\-distri\-buted sequences $Z$ we have
\[
 \lVert \psi \rVert_{\equi_{\delta , Z}}^2 + \delta^2 \lVert \zeta \rVert_{\RR^d}^2 \geq C_{\sfUC} \lVert \psi \rVert_{\mathbb{R}^d}^2 ,
\]
where
\begin{equation*}
 C_{\sfUC} = D_1 \left( \frac{\delta}{G D_2} \right)^{\frac{D_3}{\epsilon} \bigl( 1 +  G^{4/3}\lVert V \rVert_\infty^{2/3} + G^2 \lVert b \rVert_\infty^{2} + G^{4/3}\lVert c \rVert_\infty^{2/3} \bigr) -\ln \epsilon}
\end{equation*}
and $D_1$, $D_2$, and $D_3$ are positive constants depending only on $d$, $\Elliptic$, $\Lipschitz$, and $G$.
\end{theorem}
Note that $\HZwei (\mathbb{R}^d)\subset \Dom (H)$.
The drawback of Theorem~\ref{thm:BTV} is assumption~\eqref{ass:sampling},
which can be interpreted as a smallness assumption on the Lipschitz constant of $A$.
Hence, Theorem~\ref{thm:BTV} is valid for slowly varying second order coefficients only.
Our first main result Theorem~\ref{thm:sampling} gets rid of assumption~\eqref{ass:sampling}.
\begin{theorem}[Sampling Theorem]\label{thm:sampling}
Let $\delta_0=( 330  d \euler^2 \Elliptic^{11/2}(\Elliptic+1)^{5/2}(\Lipschitz+1))^{-1} $.
There is a positive constant $N$ depending  only on $d$, $\Elliptic$, and $\Lipschitz$,
such that for all measurable and bounded $V \colon \RR^d \to \RR$, all
$\psi \in \Dom (H)$ and $\zeta \in L^2 (\RR^d)$ satisfying $\lvert H \psi \rvert \leq \lvert V\psi \rvert + \lvert \zeta \rvert$
almost everywhere on $\RR^d$,
all $\delta \in (0,\delta_0)$,   and all $(1,\delta)$-equi\-distri\-buted sequences $Z$ we have
\[
 \lVert \psi \rVert_{\equi_{\delta,Z}}^2 + \delta^2 \lVert \zeta \rVert_{\RR^d}^2  \geq C_{\sfUC} \lVert \psi \rVert_{\RR^d}^2 ,
\]
where
\begin{equation*}
 C_{\sfUC} = \delta^{N(1+\lVert V \rVert_\infty^{2/3} + \lVert b \rVert_\infty^{2} + \lVert c \rVert_\infty^{2/3})} .
\end{equation*}
\end{theorem}
We refer to this result as a \emph{sampling theorem}, because $\lVert \psi \rVert_{\equi_{\delta,Z}}^2$ may be considered as a sample of the full norm
$\lVert \psi \rVert_{\RR^d}^2$.
It is a quantitative unique continuation estimate (in a specific geometric situation)
and may be considered as manifestations of the uncertainty relation.
\begin{remark}
Since $\delta \mapsto  \lVert \psi \rVert_{\equi_{\delta,Z}}$ is isotone, we have  for $\delta \geq \delta_0$ still the estimate
\[
 \lVert \psi \rVert_{\equi_{\delta,Z}}^2 + \delta^2 \lVert \zeta \rVert_{\RR^d}^2 \geq \lVert \psi \rVert_{\equi_{\delta_0,Z}}^2 + \delta_0^2 \lVert \zeta \rVert_{\RR^d}^2 \geq  \delta_0^{N(1+\lVert V \rVert_\infty^{2/3} + \lVert b \rVert_\infty^{2} + \lVert c \rVert_\infty^{2/3})}  \lVert \psi \rVert_{\RR^d}^2 ,
\]
but for large values of $\delta$ this estimate becomes trivial.
\end{remark}
By a scaling argument as in Appendix C in \cite{BorisovTV-17} we immediately obtain
\begin{corollary}[Scaled Sampling Theorem] \label{cor:samplingG}
Let $G > 0$, $\delta_0= G( 330  d \euler^2 \Elliptic^{11/2}(\Elliptic+1)^{5/2}(G\Lipschitz \allowbreak +1))^{-1} $,
and $N = N (d , \Elliptic , G \Lipschitz) > 0$ be as in Theorem~\ref{thm:sampling} with $\Lipschitz$ replaced by $G \Lipschitz$.
Then for all measurable and bounded $V \colon \RR^d \to \RR$, all
$\psi \in \Dom (H)$ and $\zeta \in L^2 (\RR^d)$ satisfying $\lvert H \psi \rvert \leq \lvert V\psi \rvert + \lvert \zeta \rvert$
almost everywhere on $\RR^d$,
all $\delta \in (0,\delta_0)$,  and all $(G,\delta)$-equi\-distri\-buted sequences $Z$ we have
\[
 \lVert \psi \rVert_{\equi_{\delta,Z}}^2 + G^2 \delta^2 \lVert \zeta \rVert_{\RR^d}^2  \geq C_{\sfUC} \lVert \psi \rVert_{\RR^d}^2 ,
\]
where
\begin{equation}\label{eq:CsfUC}
 C_{\sfUC} = \left( \frac{\delta}{G} \right)^{N(1+G^{4/3}\lVert V \rVert_\infty^{2/3} + G^2\lVert b \rVert_\infty^{2} + G^{4/3}\lVert c \rVert_\infty^{2/3})} .
\end{equation}
If $\psi $ satisfies even $H \psi =V\psi $ almost everywhere on $\RR^d$,
$C_{\sfUC}$ in \eqref{eq:CsfUC} can be replaced by
\begin{equation*} 
 C_{\sfUC}
 =
 \left(\frac{\delta}{G}\right)^{N(1 + G^{4/3}\lVert V-c\rVert_\infty^{2/3} + G^2\lVert b \rVert_\infty^{2})}
\end{equation*}
\end{corollary}
The last statement holds since $H \psi =V\psi $ is equivalent to $(H-V) \psi =0$.
\par
In \cite{BorisovTV-17},
again under the smallness condition \eqref{ass:sampling} on the Lipschitz constant of $A$, a variant of Theorem~\ref{thm:BTV} is proven for functions in $\HZwei (\Lambda_L)$.
In the finite box geometry, as well, we are able to overcome the smallness assumption \eqref{ass:sampling},
and treat arbitrary Lipschitz constants of the coefficients of $A$.
In order to state this second main result, some notation is in order.
\par
For $L > 0$ we introduce the differential operator $\Op_L : C_{\mathrm{c}}^\infty (\Lambda_L) \to L^2 (\RR^d)$,
  \begin{equation*}
  \Op_L u := - \diver (A_L \nabla u) + b_L^\T \nabla u + c_L u = -\sum_{i,j=1}^d \partial_i \left( a_L^{ij} \partial_j u \right) + \sum_{i=1}^d b_L^i \partial_i u + c_L u ,
  \end{equation*}
where $A_L : {\Lambda_L} \to \RR^{d \times d}$ with $A_L= (a_L^{ij})_{i,j=1}^d$, $b_L : {\Lambda_L} \to \CC^d$, and $c_L \colon {\Lambda_L} \to \CC^d$. We assume that $a_L^{ij} \equiv a_L^{ji}$ for all $i,j \in \{1,\ldots , d\}$, and that there are constants $\Elliptic \geq 1$ and $\Lipschitz \geq 0$ such that for
all $x,y \in {\Lambda_L}$ and all $\xi \in \RR^d$ we have
\begin{equation} \label{eq:elliptic2}
\Elliptic^{-1} \lvert \xi \rvert^2 \leq \xi^\T A_L (x) \xi \leq \Elliptic \lvert \xi \rvert^2
\quad\text{and}\quad
\lVert A_L(x) - A_L(y) \rVert_\infty \leq \Lipschitz \lvert x-y \rvert .
\end{equation}
Moreover, we assume that $b_L \in L^\infty (\Lambda_L ; \CC^d)$ and $c_L \in L^\infty ({\Lambda_L})$.
Let $H_L : \Dom (H_L) \subset L^2 (\Lambda_L) \to L^2 (\Lambda_L)$ be the Friedrichs extension of $\Op$, see e.g.\ page 325 ff.\ in \cite{Kato-80}. That is, we consider the form $a_L : \HNullEins (\Lambda_L) \times \HNullEins (\Lambda_L) \to \CC$ given by
\[
 a_L (u,v) = \int_{\Lambda_L} \left( \sum_{i,j=1}^d a_L^{ij} \partial_i u \overline{\partial_j v} + \sum_{i=1}^d b_L^i \partial_i u \overline{v} + c_L u \overline{v} \right) \drm x ,
\]
As before, the form $a_L$ is densely defined, closed, and sectorial, and $H_L$ is the unique m-sectorial operator associated with the form $a_L$.
\par
We want to derive equidistribution properties for functions $\psi_L \in \Dom (H_L)$ satisfying the
differential inequality $\lvert H_L \psi_L \rvert \leq \lvert V_L \psi_L \rvert$ almost everywhere on $\Lambda_L$ with $V_L \in L^\infty (\Lambda_L)$, that is,
we want to obtain a finite volume analogue of Theorem~\ref{thm:sampling}.
A particular feature of our estimate is that the constant will be
\emph{independent of the scale} $L$ of the cube $\Lambda_L$.
\par
Since the coefficients $a_L^{ij}$, $i,j\in\{1,\ldots d\}$ by assumption obey a Lipschitz condition on $\Lambda_L$, they are pointwise well defined,
and extend in a unique way to continuous functions  $a_L^{ij} : \overline{\Lambda_L} \to \RR$, $i,j\in\{1,\ldots d\}$, which will be denoted by the same symbol.
We shall also need the following auxiliary assumption for the coefficients $a_L^{ij}$, $i,j\in\{1,\ldots d\}$:
\begin{itemize}
  \item [(Dir)] For all $i,j\in \{1,\ldots, d\}$ with $i\neq j$, the coefficients $a_L^{ij}$ vanish on the sides of $\Lambda_L$.
\end{itemize}
\begin{theorem}[Equidistribution Theorem] \label{thm:equidistribution}
 Let $\delta_0=\bigl( 330  d \euler^2 \Elliptic^{11/2}(\Elliptic+1)^{5/2}(\Lipschitz+1) \bigr)^{-1} $, $L\in \NN$, Assumption (Dir) be satisfied,
 and $N = N (d, \Elliptic , \Lipschitz) > 0$ be as in Theorem~\ref{thm:sampling}.
Then for all measurable and bounded $V_L : \Lambda_L \to \RR$,
all $\psi_L \in \Dom (H_L)$ and $\zeta_L \in L^2 (\Lambda_L)$ satisfying $\lvert H_L \psi_L \rvert \leq \lvert V_L\psi_L \rvert + \lvert \zeta_L \rvert$ almost everywhere on $\Lambda_L$,
all $\delta \in (0,\delta_0)$, and all $(1,\delta)$-equi\-distri\-buted sequences $Z$ we have
\[
 \lVert \psi_L \rVert_{\equi_{\delta,Z} (L)}^2 + \delta^2 \lVert \zeta \rVert_{\Lambda_L}^2  \geq C_{\sfUC} \lVert \psi_L \rVert_{\Lambda_L}^2 ,
\]
where
\begin{equation*}
 C_{\sfUC} = \delta^{N(1+\lVert V_L \rVert_\infty^{2/3} + \lVert b_L \rVert_\infty^{2} + \lVert c_L \rVert_\infty^{2/3})} .
\end{equation*}
\end{theorem}
\begin{remark}
Theorem~\ref{thm:equidistribution} holds equally true if we replace (Dir) by the following weaker (but more technical) condition:
\begin{itemize}
 \item [(Dir')] $\forall k \in \{1,\ldots, d\}\ \forall \ i \in \{1,\ldots, d\} \setminus \{k\} \ \forall x \in \overline{\Lambda_L}\cap \overline{\Lambda_L(L e_k)}: \quad 0= a^{ik}(x)=a^{ki}(x)$,
\end{itemize}
where the last identity follows by the symmetry condition on the coefficients.
\end{remark}
Again, by a scaling argument as in Appendix C in \cite{BorisovTV-17} we immediately obtain
\begin{corollary}[Scaled Equidistribution Theorem]
\label{cor:scaledEquidistribution}
Let $G > 0$, $\delta_0= G( 330  d \euler^2 \Elliptic^{11/2}(\Elliptic+1)^{5/2}(G\Lipschitz \allowbreak +1))^{-1} $, $L \in G\NN$, Assumption (Dir) be satisfied,
and $N = N (d , \Elliptic , G \Lipschitz) > 0$ be as in Theorem~\ref{thm:sampling} with $\Lipschitz$ replaced by $G \Lipschitz$.
Then for all measurable and bounded $V_L : \Lambda_L \to \RR$,
all $\psi_L \in \Dom (H_L)$ and $\zeta_L \in L^2 (\RR^d)$ satisfying $\lvert H \psi_L \rvert \leq \lvert V_L \psi_L \rvert + \lvert \zeta_L \rvert$ almost everywhere on $\Lambda_L$,
all $\delta \in (0,\delta_0)$,  and all $(G,\delta)$-equi\-distri\-buted sequences $Z$ we have
\[
 \lVert \psi_L \rVert_{\equi_{\delta,Z} (L)}^2 + G^2 \delta^2 \lVert \zeta_L \rVert_{\Lambda_L}^2  \geq C_{\sfUC} \lVert \psi_L \rVert_{\Lambda_L}^2 ,
\]
where
\begin{equation}\label{eq:CsfUCfull}
 C_{\sfUC} = \left( \frac{\delta}{G} \right)^{N(1+G^{4/3}\lVert V_L \rVert_\infty^{2/3} + G^2\lVert b_L \rVert_\infty^{2} + G^{4/3}\lVert c_L \rVert_\infty^{2/3})} .
\end{equation}
If $\psi $ satisfies even $H \psi =V\psi $ almost everywhere on $\Lambda_L$,
$C_{\sfUC}$ in \eqref{eq:CsfUCfull} can be replaced by
\begin{equation*}
 C_{\sfUC}
 = \left( \frac{\delta}{G} \right)^{N(1  + G^{4/3}\lVert V_L-c_L\rVert_\infty^{2/3} + G^2\lVert b_L \rVert_\infty^{2})}
\end{equation*}
\end{corollary}
\section{Applications and Discussion}
\label{sec:applications}
In this section we present several applications
of our main theorems to self-adjoint operators.
Thereafter we discuss some limitations of our results and further research directions.
\subsection{Throughout this section we consider the following type of self-adjoint operators}
\label{sec:Self-adjoint}
To unify notation let us set $\Lambda_\infty:= \RR^d$, and fix $L \in \NN_{\infty}:=\NN\cup\{\infty\}$. We use the convention that $A_\infty = A$, $b_\infty = b$, $c_\infty = c$, $a_\infty = a$, and $H_\infty = H$.
We assume that
\begin{equation}\label{eq:sa}
 b_L = \mathrm{i} \tilde b_L
 \quad\text{and}\quad
 c_L = \tilde c_L + \mathrm{i} \diver \tilde b_L / 2
\end{equation}
for some real-valued $\tilde b_L \in L^\infty (\Lambda_L)$ with $\diver \tilde b_L \in L^\infty (\Lambda_L)$, and some real-valued $\tilde c_L \in L^\infty (\Lambda_L)$.
Note that \eqref{eq:sa} implies that the form $a_L$ is symmetric, and hence $H_L$ is a self-adjoint operator in $L^2 (\Lambda_L)$ with real spectrum.
If $L$ is finite, due to ellipticity, it has purely discrete spectrum.
Let $\delta_0 = (330 d \euler^2 \Elliptic^{11/2} (\Elliptic + 1)^{5/2} (\Lipschitz + 1))^{-1}$. For $\delta \in (0,\delta_0)$,  a $(1,\delta)$-equidistributed sequence $Z$, and $t \geq 0$ we define the self-adjoint operator
\[
H_L(t) = H_L + tW \colon \Dom (H_L) \to L^2 (\Lambda_L)
\quad\text{where}\quad
W =\mathbf{1}_{\equi_{\delta,Z} \cap \Lambda_L}.
\]
Note that we suppress the dependence of $H_L (t)$ on $\delta$ and $Z$. To simplify reading, we assume throughout Section~\ref{sec:applications} that the support of $W$ corresponds to a  $(1,\delta)$-equidistributed sequence $Z$. The general case of $(G,\delta)$-equidistributed sequences follows again by scaling.
\subsection{Uncertainty relation for short energy intervals and lower bounds on the lifting of spectra}
Theorems~\ref{thm:sampling} to \ref{cor:scaledEquidistribution} give quantitative
uncertainty relations only for eigenfunctions.
This is sufficient to estimate the lifting of isolated eigenvalues in Lemma~\ref{lemma:lifting} below.
For applications it is often required to have similar estimates for linear combinations of eigenfunctions, or more generally for $\psi\in \chi_{(-\infty,E]}(H_L)$ for arbitrary $E\in\RR$, see the discussion below. If $\Lambda_L=\Lambda_\infty=\RR^d$ this could include projectors on continuous spectrum. Currently we are only able to prove such an uncertainty principle for sufficiently short energy intervals.
The first result is an application of an idea from \cite{Klein-13}.
\begin{theorem}[Uncertainty relation for arbitrary positioned short intervals]
\label{thm:unserKlein}
Let $L\in \NN_\infty$,
Assumption (Dir) be satisfied if $L$ is finite, $\delta \in (0 , \delta_0)$, $Z$ be a $(1,\delta)$-equidistributed sequence, $E_0 \in \RR$, $N = N (d,  \Elliptic,\Lipschitz) > 0$ be as in Theorem~\ref{thm:sampling}, and
\[
 \kappa
 =
 \delta^{N (1+ \lvert E_0 \rvert^{2/3}+ \lVert c_L \rVert_\infty^{2/3}  + \lVert b_L \rVert_\infty^2)} .
\]
Then we have
\[
  \chi_I (H_L)  W  \chi_I (H_L)
  \geq  \frac{3\kappa}{4} \chi_I (H_L) ,
  \quad\text{where}\quad
  I = [E_0-\sqrt{\kappa} , E_0+\sqrt{\kappa}] .
  \]
\end{theorem}
\begin{proof}
We follow \cite[Proof of Theorem~1.1]{Klein-13}. Let $\psi \in \Ran \chi_I (H_L)$, and set $V \equiv E_0$ and $\zeta = (H_L - E_0) \psi$. Then the assumption $\lvert H_L \psi \rvert \leq \lvert V \psi \rvert + \lvert \zeta \rvert$ of Theorem~\ref{thm:sampling} (if $L = \infty$) or Theorem~\ref{thm:equidistribution} (if $L < \infty$) is satisfied by the triangle inequality.
Using $\lVert (H_L-E_0) \psi \rVert^2 \leq \kappa \lVert \psi \rVert^2$ we obtain the inequality
\begin{equation*}
 \kappa \lVert \psi \rVert_{\Lambda_L}^2
 \leq
 \lVert \psi \rVert_{\equi_{\delta,Z}\cap \Lambda_L}^2
 +
 \delta^2 \lVert (H_L- E_0) \psi \rVert_{\Lambda_L}^2
  \leq
 \lVert  \psi \rVert_{\equi_{\delta,Z} \cap \Lambda_L}^2
 +
 \delta^2 \kappa \lVert \psi \rVert_{\Lambda_L}^2 .
\end{equation*}
Since $\delta < \delta_0 < 1/2$ we find $(3/4)\kappa \lVert \psi \rVert_{\Lambda_L}^2 \leq \lVert \psi \rVert_{S_{\delta , Z} \cap \Lambda_L}^2$.
\end{proof}
In order to formulate lower bounds on the movement of eigenvalues and the infimum of the essential spectrum under the influence of the positive semi-definite potential $W$ we introduce some notation.
We set $\lambda_\infty (t) = \min \sigma_\ess (H_L (t))$.
We denote the eigenvalues of $H_L (t)$ below $\lambda_\infty (t) $ by $\lambda_k (t)$, $k \in \NN$, enumerated non-decreasingly and counting multiplicities.
If $\chi_{(-\infty , \lambda_\infty)}(H_L(t))$ has rank
$N \in \NN_0$, we set $\lambda_{k} (t) =  \lambda_\infty (t)$ for all $k \in \NN$ with $k > N$.
In the case where $\Lambda_L$ is a finite cube, this is an enumeration of the entire spectrum. If $\Lambda_L=\Lambda_\infty=\RR^d$, this may be only part of the spectrum, if any.
Note that we suppress the dependence of the eigenvalues on $L \in \NN_\infty$, $\delta \in (0,\delta_0)$ and the choice of the $(1,\delta)$-equidistributed sequence $Z$ in the notation.
\begin{lemma}[Lifting of eigenvalues and of $\min \sigma_{\mathrm{ess}}$]
\label{lemma:lifting}
Let $L\in \NN_\infty$, Assumption (Dir) be satisfied if $L$ is finite, $\delta \in (0,\delta_0)$, $Z$ be a $(1,\delta)$-equidistributed sequence,
$t > s\geq0$, $E\geq 0$, and $N = N (d , \Elliptic, \Lipschitz) > 0$ be as in Theorem~\ref{thm:sampling}.
\begin{enumerate}[(a)]
\item
  Then for all $k\in \NN$ such that $\lambda_k(t) \leq \lambda_1(0) +E$ and $\lambda_k(r) <\lambda_\infty(r)$ for all $r \in (s,t)$ we have
  \[
  \lambda_k(t)\geq \lambda_k(s) + (t-s)\delta^{N(1+  \max\{E,t\}^{2/3} +\lVert \lambda_1(0)-c_L\rVert_\infty^{2/3} + \lVert b_L \rVert_\infty^{2})} ,
  \]
\item
  If $k \in \{1,\infty\}$ and $\lambda_k(t) \leq \lambda_1(0) +E$ we have
 \[
  \lambda_k(t)\geq \lambda_k(s) + (t-s) \frac{3}{4} \delta^{N (1 + \lvert \lambda_k (0) \rvert^{2/3} + E^{2/3} + t^{2/3} + \lVert c_L \rVert_\infty^{2/3} + \lVert b_L \rVert_\infty^2)}.
 \]
\end{enumerate}
\end{lemma}
Let us emphasize that part (b) covers simultaneously $\min \sigma (H_L (t))$ and $\min \sigma_{\mathrm{ess}} (H_L (t))$.
To establish part (a) of Lemma~\ref{lemma:lifting} we will follow the first order perturbation arguments in
Section~4 of \cite{RojasMolinaV-13}. Part (b) of Lemma~\ref{lemma:lifting} is a consequence of Theorem~\ref{thm:unserKlein} and the following lemma.
\begin{lemma}
 \label{lem:bottom}
Let $A$ be self-adjoint and lower semi-bounded, and $B$ bounded and symmetric on a Hilbert space $\mathcal{H}$. Furthermore, let $\nu \in \RR$, $\lambda_1(A + B) =\min \sigma (A+B)$, $\lambda_\infty(A + B) =\min \sigma_{\mathrm{ess}}(A+B)$, and $k \in \{1,\infty\}$.
Assume that there is $\epsilon_0>0$ such that
 \[
 \forall x \in \Ran \chi_{[\lambda_k(A + B)-\epsilon_0,\lambda_k(A + B)+\epsilon_0]}(A + B) \colon \quad
 \langle x , B x \rangle \geq \nu  \lVert x \rVert^2 .
 \]
 Then we have
 \[
  \lambda_k(A + B) \geq \lambda_k(A) + \nu,
 \]
 where $\lambda_1(A):=\min \sigma (A)$ and $\lambda_\infty(A):=\min \sigma_{\mathrm{ess}}(A)$.
\end{lemma}
\begin{proof}
First we consider the case $k = 1$ and introduce the notation $J = [\lambda_1(A + B) - \epsilon_0 , \lambda_1(A + B) + \epsilon_0]$. By assumption we find
\begin{align*}
\lambda_1 (A+B) &= \inf_{\substack{x \in \Ran \chi_{J} (A+B) \\ \lVert x \rVert = 1}} \left( \langle x,Ax \rangle + \langle x,Bx\rangle \right)
 \geq \inf_{\substack{x \in \mathcal{H} \\ \lVert x \rVert = 1}} \langle x,Ax \rangle +  \nu = \lambda_1 (A) + \nu .
\end{align*}
For the case $k=\infty$ we adapt an argument of \cite{NakicTTV-18-arxiv}.
With the notation $I (\epsilon) = [\lambda_\infty(A + B) - \epsilon, \lambda_\infty(A + B) + \epsilon]$ we have using our assumption
 \begin{align*}
  \lambda_\infty(A + B)
  &=
  \inf_{0 < \epsilon \leq \epsilon_0}
  \sup_{\substack{x \in \Ran \chi_{I(\epsilon)}(A + B)\\ \lVert x \rVert = 1}}
  \left( \langle x, A x \rangle + \langle x, B x \rangle \right) \\[1ex]
  &\geq
  \inf_{0 < \epsilon \leq \epsilon_0}
  \sup_{\substack{x \in \Ran \chi_{I(\epsilon)}(A + B)\\ \lVert x \rVert = 1}}
  \langle x, A x \rangle + \nu .
 \end{align*}
 Since $\mathrm{rank} \, \chi_{I(\epsilon)}(A + B) =\infty$ for any $\epsilon > 0$, we have by the standard variational principle
  \begin{align*}
  \sup_{\substack{x \in \Ran \chi_{I(\epsilon)}(A + B)\\ \lVert x \rVert = 1}}
  \langle x, A x \rangle
  &=
  \sup_{\substack{\cL \subset
  \Ran \chi_{I(\epsilon)}(A + B)\\ \dim \cL= \infty}}
  \sup_{\substack{x \in \cL\\ \lVert x \rVert = 1}}
  \langle x, A x \rangle
   \\
  &\geq
  \inf_{\substack{\cL \subset
  \Ran \chi_{I(\epsilon)}(A + B)\\ \dim \cL= \infty}}
  \sup_{\substack{x \in \cL\\ \lVert x \rVert = 1}}
  \langle x, A x \rangle\\
  &\geq
  \inf_{\substack{\cL \subset
  \cD(A)\\ \dim \cL= \infty}}
  \sup_{\substack{x \in \cL\\ \lVert x \rVert = 1}}
  \langle x, A x \rangle
  \quad \quad   = \lambda_\infty(A).
  \qedhere
  \end{align*}
\end{proof}
We are now in position to prove Lemma~\ref{lemma:lifting}.
\begin{proof}[Proof of Lemma~\ref{lemma:lifting}]
First we prove (a), i.e.\ we treat the case where $k \in \NN$ and $\lambda_k (r) < \lambda_\infty (r)$ for all $r \in (s,t)$. In particular, $\lambda_k (r)$ is an eigenvalue of finite multiplicity for all $r \in (s,t)$.
For $r \in (s,t)$ we denote by $\psi_k (r) \in L^2 (\Lambda_L)$ a normalized eigenfunction of $H_L (r)$ corresponding to the eigenvalue $\lambda_k (r)$, i.e.\ $(H_L (r) -\lambda_k (r) )\psi_k (r)=0$.
We apply Corollary~\ref{cor:samplingG} if $L = \infty$ or Corollary~\ref{cor:scaledEquidistribution} if $L < \infty$ and obtain
\[
 \langle \psi_k (r) , W \psi_k (r) \rangle \geq \delta^{N (1 + \lVert \lambda_k (r) - c_L - tW \rVert_\infty^{2/3} + \lVert b_L \rVert_\infty^2)} .
\]
Since $\lambda_k (r) \leq \lambda_k (t) \leq \lambda_1 (0) + E$ and $0 \leq W \leq 1$, we have $\lvert \lambda_k (r) - c_L - t W \rvert \leq \lvert \lambda_1 (0) - c_L \rvert + \max\{E,t\}$ almost everywhere on $\Lambda_L$. Thus, we can further estimate
\begin{align*}
\langle \psi_k (r) , W \psi_k (r) \rangle
\geq
\delta^{N(1+  \max\{E,t\}^{2/3} +\lVert \lambda_1(0)-c_L\rVert_\infty^{2/3} + \lVert b_L \rVert_\infty^{2})} =: \kappa
\end{align*}
for all $r \in (s,t)$.
By first order perturbation theory or the Hellmann-Feynman-theorem we have for all $r \in (s,t)$
\[
 \frac{\drm}{\drm r} \lambda_k (r)
= \bigl\langle \psi_k (r), W   \psi_k (r) \bigr\rangle .
\]
This holds true also if the eigenvalue $\lambda_k (r)$ happens to be finitely degenerate at $r$, cf.\ for
instance \cite{IsmailZ-88} or \cite[\S4]{Veselic-08}.
Since $s <t$ we find
\begin{align*}
\lambda_k(t)
&=
\lambda_k(s) + \int_s^t \frac{\drm \lambda_k (r)}{\drm r} \drm r
\geq
\lambda_k (s) + \int_s^t   \kappa \drm r  = \lambda_k(s) + (t-s) \kappa .
\end{align*}
\par
In order to prove (b), let $A = H_L + sW$, $B = (t-s)W$, $\lambda_1 (t) = \min \sigma (A + B) = \min \sigma (H_L + tW)$, and $\lambda_\infty (t) = \min \sigma_{\mathrm{ess}} (A + B) = \min \sigma_{\mathrm{ess}} (H_L + tW)$.
 By Theorem~\ref{thm:unserKlein} we have for $k \in \{1,\infty\}$ and all $x \in \Ran \chi_{[\lambda_k (t) - \sqrt\kappa,\lambda_k (t) + \sqrt\kappa]}(A+B)$
 \[
  \langle x , B x \rangle \geq (t-s) \frac{3\kappa}{4} \lVert x \rVert^2
  \quad\text{where}\quad
  \kappa = \delta^{N (1+ \lvert \lambda_k (t) \rvert^{2/3} + \lVert c_L + tW \rVert_\infty^{2/3} + \lVert b_L \rVert_\infty^2)} .
 \]
 Distinguishing cases one sees that $\lvert \lambda_k (t) \rvert \leq \lvert \lambda_1 (0) \rvert + E$.
 The statement now follows from Lemma~\ref{lem:bottom}.
\end{proof}
\begin{remark}
The above proof of (a) did not actually use the fact that the considered eigenvalues are below the
essential spectrum. Indeed, the lemma holds also for discrete eigenvalues inside gaps of the essential spectrum.
However, in this case one has to introduce a consistent matching between
eigenvalues $\lambda(s)$ and $\lambda(t)$. This could be done for instance by analytic continuation in $t$, see for instance \cite{Veselic-K-08} or by choosing a reference point in the resolvent set, see e.g.\ \cite{NakicTTV-18-arxiv}.
\end{remark}
\subsection{An abstract uncertainty relation for low energy spectral projectors}
The following abstract uncertainty relation will enable us to eliminate in specific situations certain parameter dependencies which are present in Theorem~\ref{thm:unserKlein} and Lemma~\ref{lemma:lifting}. It is a variant of \cite[Theorem~1.1]{BoutetdeMonvelLS-11}. In fact, the proof is essentially the same as  \cite[Lemma~4.1]{Klein-13} which in turn is very similar to the proof in \cite{BoutetdeMonvelLS-11}.
\begin{lemma}\label{lemma:klein2}
 Let $X$ be a complex Hilbert space, $\mathfrak{h}_1$, $\mathfrak{h}_2$ and $\mathfrak{h}_3$ lower bounded, symmetric sesquilinear forms in $X$, $\mathfrak{h}_2$ non-negative, ${E_0}\in \RR$, $t>0$, $Y:=\{x\in \Dom(\mathfrak{h}_1) \cap\Dom(\mathfrak{h}_2) \cap\Dom(\mathfrak{h}_3) \colon \mathfrak{h}_1 (x,x) + \mathfrak{h}_2 (x,x) \leq {E_0} \langle x, x \rangle\}$, and
 $\gamma(t):=\inf\{(\mathfrak{h}_1 + t\mathfrak{h}_3)(x,x): x \in \Dom(\mathfrak{h}_1) \cap\Dom(\mathfrak{h}_3) \}$.
 Then we have
\[
\forall x \in Y \colon \quad
\mathfrak{h}_3 (x,x)
\geq
 \frac{\gamma(t) - {E_0}}{t} \langle x, x\rangle.
 \]
 In particular, let $T_1$ and $T_2$ be lower bounded, self-adjoint operators in $X$ such that $T_1 + T_2$ is self-adjoint, $T_2$ is non-negative, and $T_3$ is a bounded non-negative operator in $X$, ${E_0} \in \RR$, $I \subset (-\infty, {E_0}]$ measurable, $t > 0$, and $\gamma (t) = \min \sigma (T_1 + t T_3)$. Then we have
 \[
  \chi_I (T_1 + T_2) T_3 \chi_I (T_1 + T_2) \geq \frac{\gamma (t) - {E_0}}{t} \chi_I (T_1 + T_2) .
 \]
\end{lemma}
\begin{proof}
Since $\mathfrak{h}_1 + t \mathfrak{h}_3 \geq \gamma (t)$ and $\mathfrak{h}_2 (t) \geq 0$, we have for all $x\in Y$
\[
 t \mathfrak{h}_3 (x,x) \geq t \mathfrak{h}_3 (x,x) - {E_0} \langle x,x \rangle + \mathfrak{h}_1 (x,x) + \mathfrak{h}_2 (x,x) \geq (\gamma(t) -{E_0})\langle x,x \rangle .
\]
The first claim follows after dividing by $t>0$.
To conclude the second claim note that $\Dom (T_3) = X$,
$\Ran \chi_I(T_1 + T_2) \subset \Dom (T_1 + T_2) = \Dom (T_1) \cap \Dom (T_2)$ by definition,
and the domain of a self-adjoint operator $T_i$ is always a subset of the domain of the corresponding densely defined closed symmetric form $\mathfrak{h}_i$.
Thus $\Ran \chi_I(T_1 + T_2) \subset Y$. Finally,
$\min \sigma (T_1 + t T_3)$ equals the lower bound $\gamma(t)$ of the corresponding form.
\end{proof}
\begin{remark}
Note that the set $Y$ is not necessarily a linear space.
\par
Only $t>0$ such that ${E_0} < \gamma (t)$ give non-trivial bounds in the lemma.
If $J$ is any subset of $(0, \infty)$ the lemma implies
\[
\chi_I(T_1 + T_2) T_3 \chi_I(T_1 + T_2)
\geq
\kappa_J \chi_I(T_1 + T_2) \quad \text{ with }  \kappa_J:= \sup_{t\in J}\frac{\gamma (t) -{E_0}}{t}.
 \]
A natural choice would be $J=(0, \infty)$ giving
\[
\chi_I(T_1 + T_2) T_3 \chi_I(T_1 + T_2)
\geq
\sup_{t>0}\frac{\gamma(t)-{E_0}}{t}\chi_I(T_1 + T_2)
 \]
and another one $J:=J_{E_0}:=\{t>0: \gamma(t) >{E_0}\} $.
This formulation is chosen in \cite{BoutetdeMonvelLS-11} and \cite{Klein-13}.
We will instead directly insert an appropriate value of $t>0$ in our application.
\end{remark}
\subsection{Uniform uncertainty relations for spectral projectors of elliptic operators as in \S \ref{sec:Self-adjoint}}
Let us return to models as in \S \ref{sec:Self-adjoint}. We present two uncertainty relations which are valid with uniform constants for a whole family of operators. The first one is
\begin{theorem}[Uncertainty relation for low energy spectral projectors]
\label{thm:uncertaintyL}
Let $L\in \NN_\infty$, and $P$ be any non-negative operator in $L^2 (\Lambda_L)$
such that $H_L+P$ is still self-adjoint.
Then we have for all $t > 0$, ${E_0}\in \RR$, and measurable sets $I \subset (-\infty , {E_0}]$
  \[
  \chi_I (H_L + P) \, W \, \chi_I (H_L + P)
  \geq  \frac{\lambda_1(t)-{E_0}}{t}  \chi_I (H_L + P) .
  \]
In particular, let Assumption (Dir) be satisfied if $L < \infty$, $\delta \in (0 , \delta_0)$, $Z$ be a $(1,\delta)$-equidistributed sequence, $N = N (d,  \Elliptic,\Lipschitz) > 0$ be as in Theorem~\ref{thm:sampling}, and $I \subset (-\infty , \lambda_1 (0) + \kappa]$, where
\[
 \kappa = \frac{1}{4}\delta^{N (3 + \lvert \lambda_1 (0) \rvert^{2/3} + \lVert c_L \rVert^{2/3} + \lVert b_L \rVert^2)} .
\]
Then we have
 \[
  \chi_I (H_L + P)  W  \chi_I (H_L + P)
  \geq 2\kappa \chi_I (H_L + P) .
 \]
\end{theorem}
\begin{proof}
The first statement of Theorem~\ref{thm:uncertaintyL} is verbatim Lemma~\ref{lemma:klein2} with $T_1 = H_L$, $T_2 = P$, and $T_3 = W$.
For the second part we choose $t = 1$, $E= 1$,
${E_0} = \lambda_1 (0) + \kappa$,
and insert the lower bound on $\lambda_1(t)$ implied by Lemma~\ref{lemma:lifting} part (b). This way we obtain
\[
 \frac{\lambda_1(1)-{E_0}}{1} \geq \lambda_1 (1) - \lambda_1 (0) - \kappa \geq 2\kappa  ,
\]
and the second statement of Theorem~\ref{thm:uncertaintyL} follows.
\end{proof}
Next we formulate an uncertainty relation for low energy spectral projectors of the operator $H_L$ as in \S \ref{sec:Self-adjoint}, i.e.\ in the situation of Theorem~\ref{thm:uncertaintyL} with $P = 0$.
The gain compared to Theorem~\ref{thm:uncertaintyL} is that we eliminate Assumption (Dir), and that the constant in the lower bound is independent of the Lipschitz constant $\Lipschitz$.
We recall that the coefficient $A_L$ satisfies \eqref{eq:elliptic} respectively \eqref{eq:elliptic2}. Hence,
\begin{equation*}
\EllipticLower  := \inf_{x \in\Lambda_L}\sigma(A_L(x))
\quad\text{and}\quad
\EllipticU  := \sup_{x \in\Lambda_L}\sigma(A_L(x))
\end{equation*}
are both finite and positive.
\begin{theorem} \label{thm:uncertainty_Lipschitz}
 Let $L \in \NN_\infty$, $\delta \in (0 , \delta_0)$, and $Z$ be a $(1,\delta)$-equidistributed sequence.
Then,
\[
\forall x \in \{x \in \HNullEins (\Lambda_L) \colon a_L (x,x) \leq (\tilde \lambda_1 (0) + \kappa) \|x\|^2\} \colon \quad
 \langle x , W x \rangle \geq \kappa \|x\|^2,
\]
in particular
 \begin{equation*} 
 \chi_{I} (H_L) W \chi_{I} (H_L) \geq \kappa \chi_{I} (H_L).
\end{equation*}
Here
\begin{equation} \label{eq:without-lipschitz}
I = (-\infty, \tilde\lambda_1 (0) + \kappa]
,\qquad
  \kappa =
  \frac{1}{2}  \delta^{M \bigl(1 + \lvert \tilde \lambda_1 (0) / \EllipticLower \rvert^{2/3} + \EllipticLower^{-2/3} + \lVert c_L / \EllipticLower \rVert_\infty^{2/3} + \lVert b_L / \EllipticLower \rVert_\infty^2 \bigr)},
 \end{equation}
$\tilde\lambda_1(0)$ is the spectral minimum of the auxiliary quadratic form defined in
\eqref{eq:auxiliary_form}, and $M$ is a constant which depends only on the dimension.
\end{theorem}
\begin{proof}[Proof of Theorem~\ref{thm:uncertainty_Lipschitz}]
We note that $\HNullEins (\Lambda_\infty) = \HEins (\Lambda_\infty)$ and define the two forms $\tilde a_L : \HNullEins (\Lambda_L) \times \HNullEins (\Lambda_L) \to \CC$ and $p : \HNullEins (\Lambda_L) \times \HNullEins (\Lambda_L) \to \CC$
by
\begin{equation}\label{eq:auxiliary_form}
  \tilde a_L (u,v) =  \int_{\Lambda_L} \left( \EllipticLower \sum_{i=1}^d \partial_i u \overline{\partial_i v} + \sum_{i=1}^d b_L^i \partial_i u \overline{v} + c_L u \overline{v} \right) \drm x
\end{equation}
and
\[
 p (u,v) = a_L (u,v) - \tilde a_L (u,v) = \int_{\Lambda_L} \left( \sum_{i,j=1}^d a_L^{ij} \partial_i u \overline{\partial_j v} - \EllipticLower \sum_{i=1}^d \partial_i u \overline{\partial_i v} \right).
\]
The forms $\tilde a_L$ and $p$ are densely defined, closed, symmetric sesquilinear forms in $L^2 (\Lambda_L)$. Note that the form $p$ satisfies
\[
 p (u,u) \geq \EllipticLower \int_{\Lambda_L} \lvert \nabla u \rvert^2 \drm x - \EllipticLower \int_{\Lambda_L} \lvert \nabla u \rvert^2 \drm x = 0.
\]
hence is non-negative. Moreover, by definition we have $a_L = \tilde a_L + p$.
Since the form $\tilde a_L$ has constant second order coefficients, their Lipschitz constant is zero and Assumption (Dir) is certainly satisfied.
For $t \geq 0$ let $\tilde\lambda_1 (t) =
 \inf_{u \in \HNullEins (\Lambda_L)}
 \left(\tilde a_L (u,u) + t \langle u , W u \rangle \right) $.
If $\EllipticLower = 1$, then Lemma~\ref{lemma:lifting} part (b) (with $E=t$) implies for all $t \geq 0$
\[
 \tilde\lambda_1 (t) \geq \tilde\lambda_1 (0) + t \delta^{M \bigl(1 + \lvert \tilde\lambda_1 (0) \rvert + t^{2/3} + \lVert c_L \rVert_\infty^{2/3} + \lVert b_L \rVert_\infty^2 \bigr)}
\]
with a constant $M > 0$ which depends only on the dimension.
Now we consider general $\EllipticLower > 0$ and reduce it to the previous situation. Note that the second order coefficients of the form $\hat a_L = \EllipticLower^{-1} \tilde a_L$ have ellipticity constant one, the lower order coefficients are $b_L / \EllipticLower$ and $c_L / \EllipticLower$.
%
Since $\tilde \lambda_1 (0) = \EllipticLower \hat \lambda_1 (0)
:=\EllipticLower \inf_{u \in \HNullEins (\Lambda_L)} \hat a_L  (u,u)$
we obtain
\begin{align*}
 \tilde \lambda_1 (t)
 &=
 \EllipticLower \inf_{u \in \HNullEins (\Lambda_L)} \left( \EllipticLower^{-1} \tilde a_L (u,u) + (t / \EllipticLower) \langle u , Wu \rangle \right) \\
 &\geq
 \EllipticLower \left( \hat \lambda_1 (0) +  (t/\EllipticLower) \delta^{M \bigl(1 + \lvert \hat\lambda_1 (0) \rvert^{2/3} + (t/\EllipticLower)^{2/3} + \lVert c_L / \EllipticLower \rVert_\infty^{2/3} + \lVert b_L / \EllipticLower \rVert_\infty^2 \bigr)}  \right)  \\
 &= \tilde \lambda_1 (0) + t \delta^{M \bigl(1 + \lvert \tilde \lambda_1 (0) / \EllipticLower \rvert^{2/3} + (t/\EllipticLower)^{2/3} + \lVert c_L / \EllipticLower \rVert_\infty^{2/3} + \lVert b_L / \EllipticLower \rVert_\infty^2 \bigr)} .
\end{align*}
An application of Lemma~\ref{lemma:klein2} with $\mathfrak{h}_1 = \tilde a_L$, $\mathfrak{h}_2 = p$, $\mathfrak{h}_3 = \langle \cdot , W \cdot \rangle$, $E_0 = \tilde \lambda_1 (0) + \kappa$ with $\kappa$ as in the theorem, and $t = 1$ gives, using $\tilde a_L + p = a_L$,
\[
 \forall x \in \{x \in \HNullEins (\Lambda_L) \colon a_L (x,x) \leq (\tilde \lambda_1 (0) + \kappa) \langle x,x \rangle\} \colon \quad
 \langle x , W x \rangle \geq \kappa \langle x , x \rangle .
\]
This implies the statement of the theorem.
\end{proof}
\begin{remark}
The price to pay for eliminating in \eqref{eq:without-lipschitz} the dependence on the Lipschitz constant $\Lipschitz$ and the upper ellipticity constant $\EllipticU$ is the appearance of the implicit quantity
$\tilde\lambda_1(0)$ in the definition of $I$ and $\kappa$. However, for many interesting cases
one has $\tilde\lambda_1(0)=\lambda_1(0)$, e.g.\ if $b_L\equiv 0 \equiv c_L$.
If this is not the case there are the following rough bounds on
$\tilde \lambda_1 (0)$, which eliminate the auxiliary quantity from $\kappa$.
For the lower bound we use \eqref{eq:sa}, the product rule for the divergence $\diver (u \tilde b_L) = \nabla u^\T \tilde b_L + u \diver (\tilde b_L)$, and integration by parts to obtain
\begin{align*}
 \tilde a_L (u,u)
 &=
 \EllipticLower  \int_{\Lambda_L} \left( \lvert \nabla u \rvert^2 + \frac{\mathrm{i} \tilde b_L^\T \nabla u \overline u }{\EllipticLower} + \frac{\tilde c_L + \mathrm{i} \diver (\tilde b_L / 2)}{\EllipticLower} u \overline u \right) \drm x \\
 &=
 \EllipticLower  \int_{\Lambda_L} \left( \lvert \nabla u \rvert^2 + \frac{\mathrm{i} \tilde b_L^\T \nabla u \overline u}{2\EllipticLower}  + \frac{\tilde c_L \lvert u \rvert^2}{\EllipticLower} - \frac{\mathrm{i} u \tilde b_L^\T \overline{\nabla u}}{2\EllipticLower}  \right) \drm x  \\
 &=
 \EllipticLower  \int_{\Lambda_L} \left(
 \biggl( \nabla u - \frac{\mathrm{i} u \tilde b_L}{2 \EllipticLower} \biggr)^\T \overline{\biggl( \nabla u - \frac{\mathrm{i} u \tilde b_L}{2 \EllipticLower} \biggr)} +
 \frac{\tilde c_L \lvert u \rvert^2}{\EllipticLower}  - \frac{\lvert \tilde b_L \rvert^2 \lvert u \rvert^2}{4 \EllipticLower^2}  \right) \drm x  .
\end{align*}
Thus we find for all $u \in \HNullEins (\Lambda_L)$ the lower bound
\[
\tilde a_L (u,u)
\geq
\inf_{x \in \Lambda_L} \left(\tilde c_L (x)  - \frac{\lvert \tilde b_L (x) \rvert^2 }{4 \EllipticLower} \right)  \Vert u \Vert^2
\geq (-\lVert \tilde c_L \rVert_\infty - \lVert \tilde b_L \rVert^2 / (4 \EllipticLower)) \Vert u \Vert^2
\]
which could be improved by some magnetic Hardy inequality.
Since $\tilde a_L (u,u) \leq a_L (u,u)$ for all $u \in \HNullEins (\Lambda_L)$ we obtain in particular the two-sided estimate
\[
-\lVert \tilde c_L \rVert_\infty - \frac{\lVert \tilde b_L \rVert^2}{4 \EllipticLower} \leq \tilde \lambda_1 (0) \leq \lambda_1 (0) .
\]
\end{remark}
\begin{remark}
In the recent preprint \cite{StollmannS-18-arxiv} the authors prove an uncertainty relation for low energy spectral projectors for self-adjoint divergence type  operators as in Section~\ref{sec:Self-adjoint} (with $b = c = 0$).
The results in \cite{StollmannS-18-arxiv} do not require the coefficient functions to be continuous and
are in this respect more general than our Theorems~\ref{thm:uncertaintyL} and \ref{thm:uncertainty_Lipschitz}.
Let us briefly compare this with our results. Since Theorem~\ref{thm:uncertainty_Lipschitz} exhibits a constant $\kappa$ which is independent on the Lipschitz constant, it would be possible to deduce an uncertainty relation at small energies for non-continuous second order coefficients (by taking suitable limits), see e.g.\ \cite{DickeV}.
Additionally, we have also Theorem~\ref{thm:unserKlein} which does not require the considered energy interval to be close to the spectral minimum.
\end{remark}
From Theorem~\ref{thm:uncertainty_Lipschitz} we immediately obtain the following corollary which applies to the full space operator $H = H_\infty$.
\begin{corollary}
 Let $\delta \in (0 , \delta_0)$, $Z$ be a $(1,\delta)$-equidistributed sequence, $W = \mathbf{1}_{S_\delta , Z}$, $b = c =0$, $F : \RR^d \to \RR^d$ be Lipschitz continuous, $\EllipticLower , \EllipticU , \Lipschitz > 0$, $\tilde A : \RR^d \to \RR^{d \times d}$ satisfying for all $x, \in \RR^d$ and all $\xi \in \RR^d$
 \begin{equation*}
\EllipticLower \lvert \xi \rvert^2 \leq \xi^\T \tilde A (x) \xi \leq \EllipticU \lvert \xi \rvert^2
\quad\text{and}\quad
\lVert \tilde A (x) - \tilde A(y) \rVert_\infty \leq \Lipschitz \lvert x-y \rvert  ,
\end{equation*}
$A = \tilde A \circ F$, $I = (-\infty , \delta^{M (1 + \EllipticLower^{-2/3})} / 2 ]$, and $M = M (d) > 0$ be as in Theorem~\ref{thm:uncertainty_Lipschitz}.
Then we have
 \begin{equation} \label{eq:UCPhomogenization}
 \chi_{I} (H) W \chi_{I} (H) \geq \frac{1}{2} \delta^{M (1 + \EllipticLower^{-2/3})}  \chi_{I} (H) .
\end{equation}
\end{corollary}
The corollary and the following example are formulated with homogenization theory for partial differential operators in mind. It might be of interest to know that certain uncertainty relations remain stable throughout the homogenization limit and can be transferred to the homogenized operator (provided it exists).  In this context one may be interested in coefficient functions of the following type:
\begin{example}
Let $\tilde A(x)$ be a diagonal matrix with all entries on the diagonal equal to $2+ \cos (x)$. Then the lower ellipticity constant is $\EllipticLower=1$ and the upper one is $\EllipticU=\EllipticLower+2$.
Let $s_1, \dots , s_d>0$,
and $F: \RR^d \to  \RR^d$, $F(x)=(s_1x_1, \ldots, s_dx_d)$. Then $A (x) = (\tilde A\circ F)(x)$ is diagonal with
entries   $2+ \cos (s_1x_1), \ldots , 2+ \cos (s_d x_d)$.
The Lipschitz constant of $H_F$ diverges for $ s_i\to\infty$, but our bound \eqref{eq:UCPhomogenization} is not effected by that.
\end{example}
\subsection{Wegner estimate for elliptic second order operator plus random potential}
\label{ss:Wegner}
Let us introduce a class of random operators which are a sum of a deterministic part and a random potential.
The deterministic part is a self-adjoint partial differential operator of the type considered in \S \ref{sec:Self-adjoint}.
The random part consists of a random potential from a rather general class introduced in \cite{NakicTTV-18},
which includes alloy-type and random breather potentials as special cases, see the discussion below.
\par
Let $\mathcal{D} \subset \RR^d$ be a Delone set, i.e.\ there are $0< G_1<G_2$ such that for any $x \in \RR^d$
we have $\sharp\{ \mathcal{D} \cap (\Lambda_{G_1} + x) \} \leq 1$ and $\sharp \{ \mathcal{D} \cap (\Lambda_{G_2} + x) \} \geq 1$.
Here, $\sharp\{ \cdot \}$ stands for the cardinality. For $0 \leq \omega_{-} < \omega_{+} < 1$ we define the probability space $(\Omega , \mathcal{A} , \PP)$ with
\[
 \Omega = \bigtimes_{j \in \mathcal{D}} \RR , \quad
 \mathcal{A} = \bigotimes_{j \in \mathcal{D}} \mathcal{B} (\RR)
 \quad \text{and} \quad
 \PP = \bigotimes_{j \in \mathcal{D}} \mu ,
\]
where $\mathcal{B}(\RR)$ is the Borel $\sigma$-algebra and $\mu$ is a probability measure with $\mathrm{supp}\ \mu \subset [\omega_{-}, \omega_{+}]$ and a bounded density $\nu_{\mu}$.
\par
Furthermore, let $\{ u_t : t \in [0,1] \} \subset L^\infty(\RR^d)$ be functions such that there are $G_u \in \NN$, $u_{\max} \geq 0$, $\au, \bu> 0$ and $\ao, \bo \geq 0$ with
\begin{align}
 \forall  t \in [0,1] & \colon \supp u_t \subset \Lambda_{G_u}, \nonumber \\
 \forall  t \in [0,1] &\colon \lVert u_t \rVert_\infty \leq u_{\max} , \label{eq:condition_u} \\
 \forall  t \in [\omega_{-}, \omega_{+}],\ \delta \leq 1 - \omega_{+}:\
 \exists x_0 \in \Lambda_{G_u} &\colon u_{t + \delta} - u_t \geq \au  \delta^{\ao} \chi_{\ball{\bu  \delta^{\bo}}{x_0}} . \nonumber
\end{align}
For each $L \in \NN$ we define the family of Schr\"odinger operators $H_{\omega,L} : \Dom (H_L) \to L^2 (\Lambda_L)$, $\omega \in \Omega$, by
\[
 H_{\omega,L} := H_L + V_\omega \quad \text{where} \quad V_\omega(x) = \sum_{j \in \mathcal{D}} u_{\omega_j} (x-j) .
\]
Note that for all $\omega \in [0,1]^{\mathcal{D}}$ we have $\lVert V_\omega \rVert_\infty \leq K_u := u_{\max}  \lceil G_u/G_1 \rceil^d$.
Assumption~\eqref{eq:condition_u} includes many prominent models of linear and non-linear random Schr\"o\-ding\-er operators. We refer the reader to \cite{NakicTTV-18} for a more detailed discussion. Here, we give 'only' two prominent examples.
\begin{description}[\setleftmargin{0pt}\setlabelstyle{\bfseries}]
  \item[Standard random breather model:]Let $\mu$ be the uniform distribution on $[0, 1/4]$ and let $u_t(x) = \chi_{\ball{t}{0}}$, $j \in \ZZ^d$. Then $V_\omega = \sum_{j \in \ZZ^d} \chi_{\ball{\omega_j}{j}}$ is the characteristic function of a disjoint union of balls with random radii. For this model we have $G_u = u_{\mathrm{max}} = \alpha_1 = \beta_2 = 1$, $\alpha_2 = 0$, and $\beta_1 = 1/2$.
  \item[Alloy-type model]
  Let $0 \leq u \in L_0^\infty(\RR^d)$, $u \geq \alpha > 0$ on some open set and let $u_t(x) := t u(x)$. Then $V_\omega(x) = \sum_{j \in \ZZ^d} \omega_j u(x-j)$ is a sum of copies of $u$ at all lattice sites $j \in \ZZ^d$, multiplied with $\omega_j$. For this model we have $\alpha_1 = \alpha_2 = 1$, and $\beta_2 = 0$.
 \end{description}
\begin{theorem}[Wegner estimate]\label{thm:wegner}
For all $E_0 \in \RR$ there are positive constants
\begin{itemize}
 \item $C = C (d , \lVert b_L \rVert_\infty, \lVert \diver b_L \rVert_\infty , \lVert c_L \rVert_\infty , E_0, K_u , \Elliptic)$,
 \item $\kappa = \kappa(d, \omega_+ , \au , \ao , \bu , \bo , G_2 , G_u , K_u , E_0,\lVert b_L \rVert_\infty,\lVert c_L \rVert_\infty,\Elliptic,\Lipschitz)$, and
 \item $\epsilon_{\mathrm{max}} = \epsilon_{\mathrm{max}} (d, \omega_+ , \au , \ao , \bu , \bo , G_2 , G_u , K_u , E_0,\lVert b_L \rVert_\infty,\lVert c_L \rVert_\infty,\Elliptic,\Lipschitz)$,
\end{itemize}
such that for all $L \in (G_2 +G_u) \NN$ such that assumption (Dir) is satisfied, all $E \in \RR$ and
$\epsilon \leq \epsilon_{\max}$ with $[E - \epsilon, E + \epsilon] \subset (- \infty,E_0]$ we have
 \begin{equation*}
  \EE \left[ \mathrm{Tr} \left[ \chi_{[E - \epsilon, E + \epsilon]} (H_{\omega, L}) \right] \right] \leq C \lVert \nu_\mu \rVert_\infty
   (4\epsilon)^{1/\kappa}
   L^{2d}.
 \end{equation*}
\end{theorem}
\begin{remark}[Energy band and volume dependence in the Wegner estimate]
The Wegner estimate given in Theorem~\ref{thm:wegner} exhibits a H\"older continuity with respect to $2\epsilon$, the
length of the energy interval. For certain random potentials with linear dependence on the randomness,
e.g.\ alloy-type models with non-negative single site potentials, based on what is known for
Schr\"odinger operators one could expect that actually Lipschitz continuity holds.
For the  general model which we treat allowing a non-linear dependence on the random variables,
H\"older continuity is the best possible result.
\par
The Wegner estimate given in Theorem~\ref{thm:wegner} holds at all energies, but has a quadratic volume dependence.
The expected optimal dependence is linear in the volume of the cube. This can be improved in several ways.
Both of them have been worked out in the case of Schr\"odinger operators with electromagnetic potentials.
The extension to elliptic operators with variable second order coefficients as considered in this paper
would require a fair amount of technical work.
\begin{enumerate}[(a)]
 \item
The method of \cite{HundertmarkKNSV-06} allows one to replace the term $\epsilon^{1/\kappa} L^{2d}$ by
$\epsilon^{1/\kappa} |\ln \epsilon|^d L^d\leq \epsilon^{1/\tilde \kappa} L^d$.
This bound has the correct volume behavior, but a slightly worsened H\"older continuity.
However, in our situation, where we do not have an optimal estimate for the exponent $1/\kappa$, this is not relevant.
For Schr\"odinger operators this has been worked out in \cite{NakicTTV-18}.
\par
To extend this result to  general elliptic second order differential operators as treated in this paper
one would need to apply a generalized Feynman-Kac-Ito formula to obtain analogous spectral shift estimates as in
\cite{HundertmarkKNSV-06}.
\item
Alternatively, for energies near the bottom of the spectrum,
one can employ the idea of  \cite{BoutetdeMonvelLS-11} to establish an uncertainty principle for spectral projectors using a lifting estimate
for the ground state energy, cf.~Theorem \ref{thm:uncertaintyL} above.
Once this is available one can follow the strategy of \cite{CombesHK-07} in order to obtain
a Wegner estimate  where the term $\epsilon^{1/\kappa} L^{2d}$ above is replaced by $\epsilon^{1/\kappa} L^d$.
\par
For Schr\"odinger operators with random potential of generalized alloy or Delone-alloy-type
the last mentioned improvement has been implemented in a series of papers in increasing generality.
In the case where $H_L$ is a Schr\"odinger operator this has been implemented in \cite[\S 4.5]{RojasMolinaV-13}.
A variant suitable for high disorder alloy-type Schr\"odinger operators
 was established in \cite{Klein-13}, and in \cite{TaeuferT-18}
this  improvement was established for magnetic Schr\"odinger operators.
\par
To extend these results to the setting of the present paper
one would need to generalize the trace class Combes-Thomas estimates of \cite{CombesHK-07}
to general elliptic second order differential operators.
\item
Finally, the uncertainty principle for spectral projectors
as formulated in Theorem \ref{thm:uncertaintyL}  holds in the case of Schr\"odinger operators
not only for low energies, but actually for arbitrary energy intervals of the type $(-\infty,E]$.
This was proven in  \cite{NakicTTV-18}. It seems that this statement carries over to
general elliptic second order differential operators as treated in this paper.
We are pursuing this topic in a follow up project.
\par
This again can be used to obtain better Wegner estimates for the models considered in
this application section.
\end{enumerate}
\end{remark}
\begin{proof}
Fix $E_0 \in \RR$.
Note that $\lambda_i(H_{\omega,L}) \leq E_0$ implies that $\lambda_i(H_{\omega + \delta,L}) \leq E_0 + \lVert V_{\omega + \delta} - V_\omega \rVert_\infty \leq E_0 + 2 K_u$.
Now we apply a scaled version of Lemma~\ref{lemma:lifting}
and obtain for all $L \in (G_u+G_2) \NN$, all $\omega \in [\omega_{-}, \omega_{+}]^{\mathcal{D}}$, all $\delta \leq \min \{1 - \omega_+, ( 330  d \euler^2 \Elliptic^{11/2}(\Elliptic+1)^{5/2}((G_u + G_2)\Lipschitz+1) )^{-1} \} =: \delta_{\mathrm{max}}$ and all $i \in \NN$ with $\lambda_i(H_{\omega, L}) \leq E_0$ the inequality
\begin{equation*}
   \lambda_i(H_{\omega + \delta, L})
   \geq
   \lambda_i(H_{\omega, L}) + \au \delta^{\ao}
   \left(\frac{\delta}{G_u + G_2}\right)^{N(1+(E_0 + 2K_u)^{2/3} + \lVert b_L \rVert_\infty^{2} + \lVert c_L \rVert_\infty^{2/3})} .
\end{equation*}
In particular, there is $\kappa = \kappa(d, \omega_+ , \au , \ao , \bu , \bo , G_2 , G_u , K_u , E_0,\lVert b_L \rVert_\infty,\lVert c_L \rVert_\infty,\Elliptic,\Lipschitz) > 0$, such that for all $\delta \leq \delta_{\mathrm{max}}$
\begin{equation*} 
 \lambda_i(H_{\omega + \delta, L})
   \geq
   \lambda_i(H_{\omega, L}) + \delta^{\kappa}.
\end{equation*}
Let $\epsilon_{\mathrm{max}} = \delta_{\mathrm{max}}^{\kappa} / 4$, $0 < \epsilon \leq \epsilon_{\mathrm{max}}$, and choose $\delta \in (0,\delta_{\mathrm{max}}]$ such that
$4\epsilon=\delta^{\kappa}$, i.e.\ $\delta = (4\epsilon)^{1/\kappa}$.
With this notation we have
\begin{equation*}
 \lambda_i(H_{\omega + \delta,L}) \geq \lambda_i(H_{\omega,L}) + 4 \epsilon .
\end{equation*}
Now we follow literally the proof of Theorem~2.8 in \cite{NakicTTV-18} and obtain
\begin{equation*}
\EE \left[ \mathrm{Tr} \left[ \chi_{[E - \epsilon, E + \epsilon]} (H_{\omega, L}) \right] \right] \leq
\lVert \nu_\mu \rVert_{\infty}  \delta\sum_{n=1}^{\lvert \tilde\Lambda_L \rvert}
 \left[ \Theta_n(\omega_{+}+  \delta) - \Theta_n(\omega_{-}) \right],
\end{equation*}
where $\tilde \Lambda_L := \{ j \in \mathcal{D} : \exists t \in [0,1]: \supp u_t( \cdot - j) \cap \Lambda_L \neq \emptyset \}$ is the set of lattice sites which can influence the potential within $\Lambda_L$,
\begin{align*}
\Theta_n(t) &:= \mathrm{Tr} \left[ \rho \left( H_{{\tilde{\omega}^{(n, \delta)}(t)},L} - E - 2 \epsilon \right) \right] ,
\end{align*}
$\rho \in C^\infty(\mathbb{R})$, $-1 \leq \rho \leq 0$, is smooth, non-decreasing such that $\rho = -1$ on $(-\infty; -\epsilon]$, $\rho = 0$ on $[\epsilon; \infty)$, and $\lVert \rho' \rVert_\infty \leq 1/\epsilon$, and where
for given $\omega \in [\omega_{-}, \omega_{+}]^{\mathcal{D}}$,
$n \in \{ 1, \ldots, \lvert \tilde \Lambda_L \rvert \}$,
$\delta \in [0, \delta_{\max}]$ and
$t \in [\omega_{-},  \omega_{+}]$,
we define $\tilde{\omega}^{(n, \delta)}(t) \in [\omega_{-}, 1]^{\mathcal{D}}$ inductively via
\begin{align*}
\left( \tilde{\omega}^{(1, \delta)}(t) \right)_j &:=
        \begin{cases}
                 t & \mbox{ if } j = k(1),\\
                \omega_j & \mbox{else},
        \end{cases} \quad\text{and}\quad
\left( \tilde{\omega}^{(n, \delta)}(t) \right)_j :=
        \begin{cases}
                 t & \mbox{ if } j = k(n),\\
                \left( \tilde{\omega}^{(n-1, \delta)}(\omega_j + \delta) \right)_j & \mbox{else}.
        \end{cases}
\end{align*}
Here, $k : \{ 1, \ldots,  \sharp\tilde \Lambda_L \} \rightarrow \mathcal{D}$, $n \mapsto k(n)$, denotes an enumeration of the points in $\tilde \Lambda_L$.
Since $\rho \leq 0$ we have
\begin{align*}
  \Theta_n(\omega_{+}+  \delta) - \Theta_n(\omega_{-})
  &\leq   - \Theta_n(\omega_{-}) =- \mathrm{Tr} \left[ \rho \left( H_{{\tilde{\omega}^{(n, \delta)}(\omega_{-})},L} - E - 2 \epsilon \right) \right] .
\end{align*}
Since $-\rho \leq \chi_{(-\infty , \epsilon]}$ and by
a Weyl bound as in \cite[Lemma~5]{HundertmarkKNSV-06} we obtain
\begin{multline*}
\Theta_n(\omega_{+}+  \delta) - \Theta_n(\omega_{-})
    \leq
   \mathrm{Tr} \left[\chi_{(-\infty,E+3\epsilon] }\left( H_{{\tilde{\omega}^{(n, \delta)}(\omega_{-})},L}  \right) \right]
   \\
    \leq
   \lvert \Lambda_L \rvert \left(\frac{e}{2 \pi d \EllipticLower}\right)^{2/d}
   (E+3\epsilon +K_u + \lVert c\rVert_\infty + \lVert \diver  b\rVert_\infty
   + (\lVert b\rVert_\infty^2/4\EllipticLower) )^{d/2} .
\end{multline*}
The result follows since the number of terms in the $n$-sum is bounded by $\lvert \tilde\Lambda_L \rvert \leq (2L / G_1)^d$.
\end{proof}
\subsection{Outlook and further research goals}
To illustrate further the motivation for our results in Section \ref{sec:results},
we discuss possible extensions and resulting applications.
First we turn to the topic of
\paragraph{Control theory for the heat equation}
Let $L\in \NN_\infty$,  $\delta \in (0,1/2)$, $Z$ a $(1,\delta)$-equidistributed sequence, $W$ and $H_L$ be as in
Section~\ref{sec:Self-adjoint}.
Given $T>0$, we consider the inhomogeneous Cauchy problem
\begin{equation}
\label{eq:parabolic_control_system}
 \begin{cases}
  \partial_t u(t) + H_L u(t) &= W f(t),\ 0<t<T,\\
  u(0) &= u_0 \in L^2(\Lambda_L),
 \end{cases}
\end{equation}
where $u, f \in L^2([0,T], L^2(\Lambda_L))$.
The function $f$ is called~\emph{control function} and the  operator $W$ is
called~\emph{control operator}. In our case it is the multiplication operator with the characteristic function of the \emph{observability set} $\equi_{\delta,Z} \cap \Lambda_L$.
The~\emph{mild solution} to~\eqref{eq:parabolic_control_system} is given by the Duhamel formula
\begin{equation*}
 u(t)
 =
 \euler^{-t H_L} u_0 + \int_0^t \euler^{-(t-s) H_L} W f(s) \drm s.
\end{equation*}
One of the central questions in control theory is whether, given an input state $u_0$ and a time $T > 0$, it is
possible to find a control function $f$, such that $u(T) = 0$. If this is the case the system~\eqref{eq:parabolic_control_system} is called \emph{null-controllable in time $T$}.
\par
Such a result is implied by a sufficiently strong uncertainty relation, see for instance
\cite{TenenbaumT-11, RousseauL-12,BeauchardPS-18, EgidiNSTTV} and the references therein.
Specifically, we would need to have at our disposal an analog of Theorem \ref{thm:uncertainty_Lipschitz}
which holds for any semi-bounded energy interval of the form $(-\infty, E]$, $E \in \RR$.
While this is one of our future research goals, let us state a partial result which can be formulated with
the results established in this paper and which could serve as a first step in the proof of null-controllability of the system \eqref{eq:parabolic_control_system}.
It concerns an auxiliary control problem, which we formulate next.
\par
Fix some $E\in \RR$ and consider the system
\begin{equation} \label{eq:truncated_control_system}
 \begin{cases}
  \partial_t u(t) + H_L u(t) &= \chi_{(-\infty , E]} (H_L)W f(t),\ 0<t<T,\\
  u(0) &= \chi_{(-\infty , E]} (H_L)u_0, \quad u_0 \in  L^2(\Lambda_L),
 \end{cases}
 \end{equation}
Again we say that the system \eqref{eq:truncated_control_system} is null-controllable in time $T > 0$ if for all $u_0 \in L^2 (\Lambda_L)$
there exists a function $f \in L^2 ([0,T] , L^2 (\Lambda_L))$
such that the solution of \eqref{eq:truncated_control_system} satisfies $u (T) = 0$.
Moreover, we define the costs as
\[
\mathcal{C} (T, y_0) = \inf \{ \lVert f \rVert_{L^2 ([0,T] , L^2 (\Lambda_L))} \colon \text{the solution of \eqref{eq:truncated_control_system} satisfies} \ u (T) = 0 \}.
\]
\begin{lemma}
Let $L\in \NN_\infty$, Assumption (Dir) be satisfied if $L < \infty$, $\delta \in (0 , \delta_0)$, $Z$ be a $(1,\delta)$-equidistributed sequence, $N = N (d,  \Elliptic,\Lipschitz) > 0$ be as in Theorem~\ref{thm:sampling}, and $T>0$.
Assume further that $\lambda_1 (0) =0$ and let
 \[
  0 < E \leq \kappa :=
  \frac{1}{4} \delta^{N(3 + \lVert c_L \rVert_\infty^{2/3} + \lVert b_L \rVert_\infty^{2})} .
 \]
Then the system~\eqref{eq:truncated_control_system} is null-controllable at time $T > 0$ with costs
\[
 \mathcal{C} \leq \sqrt{\frac{2}{T}}  \delta^{-(N/2)(3 + \lVert c_L \rVert_\infty^{2/3} + \lVert b_L \rVert_\infty^{2})} \lVert u_0 \rVert_{\Lambda_L} .
\]
\end{lemma}
\begin{proof}
 Controllability of \eqref{eq:truncated_control_system} at time $T$ is equivalent to final state observability of the system
 \begin{equation*}
\begin{cases}
 \partial_t y(t) + H_L y(t)  = 0,\ 0<t<T,  \\
y(0) =y_0 \in \Ran \chi_{(-\infty , E]} (H_L).
 \end{cases}
 \end{equation*}
 This is a classical fact and can be inferred e.g.\ from  \cite{Coron-07}, \cite{TucsnakW-09}, \cite{RousseauL-12} or \cite{EgidiNSTTV}.
By the contractivity of the semigroup and the spectral theorem we have
\[
 T \lVert y(T) \rVert_{\Lambda_L}^2
 \leq \int_0^T \lVert \euler^{-s H_L} y(0) \rVert_{\Lambda_L}^2 \mathrm{d}s
 =    \int_0^T \lVert \chi_{(-\infty , E]} (H_L) \euler^{-s H_L}  y(0) \rVert_{\Lambda_L}^2 \mathrm{d}s .
\]
By Theorem~\ref{thm:uncertaintyL} we  obtain the  estimate
 \begin{align*}
  T \lVert y(T) \rVert_{\Lambda_L}^2
  &  \leq  \frac{1}{2\kappa}
  \int_0^T \lVert \chi_{(-\infty , E]} (H_L) \euler^{-s H_L} y(0) \rVert_{\equi_{\delta , Z} (L)}^2 \mathrm{d}s
   =     \frac{1}{2\kappa}
   \int_0^T \lVert \euler^{-s H_L} y(0) \rVert_{\equi_{\delta , Z} (L)}^2 \mathrm{d}s ,
 \end{align*}
which is the desired finite state observability.
It implies,  see e.g.\ \cite{RousseauL-12} or \cite{EgidiNSTTV}, that the control cost is estimated
by square root of the the observability constant $1/(2\kappa T)$.
\end{proof}
\paragraph{Wegner estimates for random divergence type operators with small support}
In Section \ref{ss:Wegner}
we have discussed Wegner estimates for elliptic second order operators with random potential. We envisage to apply our results to a related but more challenging goal, namely a Wegner estimate for \emph{random operators in divergence form}. These are elliptic second order operators  $- \operatorname{div} (A_\omega \nabla) $, where the second order term itself is random with a suitable matrix-valued random field $A_\omega$. They model propagation of classical waves in random media.
\par
Let us be a bit more specific about possible choices for the field $A_\omega$.
In \cite{FigotinK-97c}, operators of the form $- \operatorname{div} (\rho_\omega^{-1} \nabla)$
are studied.
There,
\[
\rho_\omega(x) = \rho_0(x) \biggl(1 + \epsilon \sum_{j \in \ZZ^d} \omega_j u(x-j) \biggr) ,
\]
where $\rho_0$ is a uniformly positive and bounded, periodic function,
$u$ a measurable, bounded, compactly supported and non-negative function, $\epsilon > 0$ a small disorder parameter, and the $\omega_j$, $j \in \ZZ^d$, are uniformly bounded independent and identically distributed random variables.
Note that the coefficient matrix is isotropic in the sense that it is a multiple of the identity.
This condition was dispensed with in~\cite{Stollmann-98}, which allowed the modeling of random anisotropic media.
There, the random perturbation consists of a sum of random rotations of random  positive diagonal matrices in every periodicity cell.
\par
We hope that our results enable us to study the case where the support of the function $u$ is small
and where small deviations of the position of the translates $u(\cdot -j)$ is allowed.
This will, however, be dealt with in a separate project.
%
%
%
%
%
%
%
%
\section{Three annuli inequality}
\label{sec:three-annuli}
In this section we deduce a three annuli inequality from the quantitative Carleman estimate of \cite{NakicRT}.
For our purpose the particular Carleman estimate from \cite{NakicRT} is crucial,
since it provides explicit upper and lower bounds of the weight function in terms of $\Elliptic$ and $\Lipschitz$.
A non-quantitative version of the Carleman estimate with the same weight function,
proven in \cite{EscauriazaV-03}, is not sufficient for our purpose.
\par
For $0< r_1 < R_1 \leq r_2 < R_2 \leq r_3 < R_3 < \infty$ we use for $i \in \{1,2,3\}$ the notation
\[
 Z_i := \ballc{R_i} \setminus \overline{\ballc{r_i}} \subset \RR^{d} .
\]
For $x \in \RR^d$ we denote by $Z_i (x) = Z_i + x$ the translated annuli.
\begin{theorem}[Three annuli inequality]\label{thm:three_annuli}
 Let $0 < r_1 < R_1 \leq r_2 < R_2 \leq r_3 < R_3$ and $\epsilon > 0$.
Then for all measurable and bounded $V \colon \RR^d \to \RR$ there are constants $\alpha^* \geq 1$ and $D_i > 0$, $i \in \{1,2,3\}$,
depending merely on $r_j$, $R_j$, $j \in \{1,2,3\}$, $\epsilon$, $d$, $\Elliptic$, $\Lipschitz$, $\lVert V \rVert_\infty$, $\lVert b \rVert_\infty$, and $\lVert c \rVert_\infty$,
such that for all
$\psi \in \Dom (H)$ and  $\zeta \in L^2 (\RR^d)$ satisfying $\lvert H \psi \rvert \leq \lvert V\psi \rvert + \lvert \zeta \rvert$
almost everywhere on $\ballc{R_3}$, and all $\alpha \geq \alpha^*$ we have
\[
 \alpha^3 \lVert \psi \rVert_{Z_2}^2
 \leq
 D_1 \left(\frac{R_2 \mu_1 \Elliptic}{r_1}\right)^{2\alpha} \lVert \psi \rVert_{Z_1}^2
 +
 D_2 \left(\frac{R_2 \mu_1 \Elliptic}{r_3}\right)^{2\alpha} \lVert \psi \rVert_{Z_3}^2
 +
 D_3 \left(\frac{R_2 \mu_1 \Elliptic}{r_1}\right)^{2\alpha} \lVert \zeta \rVert_{\ballc{R_3}}^2 ,
\]
where
\begin{equation}\label{eq:mu1}
 \mu_1 = \mu_1(R_3,\epsilon) =
        \begin{cases}
          \exp(\mu \sqrt{\Elliptic})    & \text{if $\mu \sqrt{\Elliptic} \leq 1$,}\\
          \euler \mu \sqrt{\Elliptic}   & \text{if $\mu \sqrt{\Elliptic} > 1$},
         \end{cases}
\end{equation}
with  
$\mu = 33 d R_3 \Elliptic^{11/2} \Lipschitz + \epsilon$.

\end{theorem}
\begin{lemma}\label{lemma:three_annuli}
 Let $\epsilon = 1$. Then
  \begin{align*}
   D_1 &\leq \frac{(R_1 - r_1)^{-2} R_2 \euler^{K (R_3 + 1)}   }{\min \{(R_1 - r_1)^2/16 , 1\}}
 \left(1 + \lVert V \rVert_\infty^2 + \lVert b \rVert_\infty^2 + \lVert c \rVert_\infty^2 \right), \\[1ex]
   D_2 &\leq \frac{(R_3-r_3)^{-2} R_2 \euler^{K (R_3 + 1)}}{\min\{(R_3-r_3)^2/16 , 1\}}
 \left( 1+\lVert V \rVert_\infty^2 + \lVert b \rVert_\infty^2 + \lVert c \rVert_\infty^2  \right),  \\[1ex]
   D_3 &\leq
   \left( (R_1 - r_1)^{-4} + (R_3 - r_3)^{-4} + 1  \right) R_1^2 R_2 \euler^{K (R_3 + 1)} , \ \text{and} \\[1ex]
   \alpha^* &\leq \euler^{K (R_3 + 1)} \left(1 + \lVert V \rVert_\infty^{2/3} + \lVert b \rVert_\infty^2 + \lVert c \rVert_\infty^{2/3} \right) ,
  \end{align*}
 where $K \geq 1$ is a constant depending only on $d$, $\Elliptic$ and $\Lipschitz$.
\end{lemma}
In order to prove Theorem~\ref{thm:three_annuli},
we start with a formulation of the the \emph{Cacciopoli inequality},
which may be found in \cite{RojasMolinaV-13} for
the pure Laplacian
and in \cite{BorisovTV-17} for second order elliptic differential operators.
\begin{lemma}[Cacciopoli inequality from \cite{BorisovTV-17}] \label{lemma:Cacciopoli}
Let $0 < \rho_1 < \rho_2$, $\kappa \in (0,\rho_1)$, $\omega = \ballc{\rho_2} \setminus \overline{\ballc{\rho_1}}$, $\omega^+ = \ballc{\rho_2 + \kappa} \setminus \overline{\ballc{\rho_1 - \kappa}}$, $V \colon \RR^d \to \RR$ bounded and measurable, $\xi \in L^2 (\RR^d)$, $u \in C_{\mathrm{c}}^\infty (\RR^d)$ satisfying $\lvert \Op u \rvert \leq \lvert V u \rvert + \lvert \xi \rvert$
almost everywhere on $\omega^+$.
Then there is an absolute constant $C' \geq 1$ such that
\begin{multline*}
  \int_{\omega} {\nabla u^\T A \nabla \overline{u}}
  \leq
  \Cac_{\kappa} \int_{\omega^+} \lvert u \rvert^2   + 2 \int_{\omega^+} \lvert \xi \rvert^2 ,
  \\
  \text{where } \quad
  \Cac_{\kappa} := \Cac_\kappa (V , b , c, \Elliptic):=
  1 + 2 \lVert V \rVert_\infty^2 +  2\lVert b \rVert_\infty^2 + 2 {\lVert c \rVert_\infty}+ \frac{8\Elliptic^2 C'}{\kappa^2} .
\end{multline*}
\end{lemma}
To formulate the Carleman estimate from \cite{NakicRT} we need some notation.
For $\mu,\rho > 0$ we introduce a function $w_{\rho , \mu} : \RR^d \to [0,\infty)$ by $w_{\rho , \mu} (x) := \varphi (\sigma (x / \rho))$, where $\sigma:\RR^d \to [0,\infty)$ and $\varphi : [0,\infty) \to [0,\infty)$ are given by
\begin{align*}
\sigma(x):= \left(x^\T A^{-1} (0) x \right)^{1/2}, \quad \text{and} \quad
\varphi(r):= r \exp \left( - \int_0^r\frac{1 - e^{-\mu t}}{t} \mathrm{d} t \right).
\end{align*}
Note that the function $w_{\rho , \mu}$ satisfies
\begin{equation}  \label{eq:weightbound2}
 \forall x \in \ballc{\rho} \colon \quad
 \frac{\Elliptic^{-1/2} \lvert x \rvert}{\rho \mu_1} \leq \frac{\sigma (x)}{\rho \mu_1} \leq
 w_{\rho , \mu}(x)
 \leq
 \frac{\sigma (x)}{\rho}
 \leq
 \frac{\sqrt{\Elliptic} \lvert x \rvert}{\rho}  ,
\end{equation}
where $\mu_1$ is as in \eqref{eq:mu1}.
\begin{theorem}[Carleman estimate from \cite{NakicRT}] \label{thm:carleman}
Let $\rho > 0$ and $\mu > 33 d \rho \Elliptic^{11/2} \Lipschitz $.
Then there are constants $\alpha_0 = \alpha_0 (d, \rho  , \Elliptic , \Lipschitz, \mu , \lVert b \rVert_\infty , \lVert c \rVert_\infty) > 0$
and $C = C (d , \Elliptic , \rho\Lipschitz, \mu) > 0$, such that for all $\alpha \geq \alpha_0$ and all
$u \in C_{\mathrm{c}}^\infty (\ballc{\rho} \setminus \{0\})$
we have
\begin{equation*}
 \int_{\RR^d} \left( \alpha \rho^2 w_{\rho , \mu}^{1-2\alpha} \nabla u^\T A \nabla u + \alpha^3 w_{\rho , \mu}^{-1-2\alpha} \lvert u \rvert^2  \right) \leq C \rho^4 \int_{\RR^d} w_{\rho , \mu}^{2-2\alpha} \lvert \Op u\rvert^2 .
\end{equation*}
\end{theorem}
\begin{remark}
Upper bounds for the constants $C$ and $\alpha_0$ are known explicitly, see \cite{NakicRT}. In the case where $b$ and $c$ are identically zero, the conclusion of Theorem~\ref{thm:carleman} holds with $C = \tilde C$ and $\alpha_0 = \tilde \alpha_0$ satisfying the upper bounds
\begin{align*}
 \tilde C &\leq 2d^2  \Elliptic^8 \euler^{4\mu\sqrt{\Elliptic}} \mu_1^4 \left( 3 \mu^2 + (9\rho \Lipschitz + 3)\mu + 1 \right) C_\mu^{-1} \\
 \intertext{and}
 \tilde \alpha_0 & \leq 11 d^4 \Elliptic^{33/2} \euler^{6\mu\sqrt{\Elliptic}}   \mu_1^6  (3\rho\Lipschitz + \mu + 1)^2 \left(1  + \mu (\mu + 1) C_\mu^{-1}   \right) ,
\end{align*}
where $C_\mu = \mu - 33 d \Elliptic^{11/2} \Lipschitz \rho$. In the general case where $b,c \in L^\infty (\ballc{\rho})$ the conclusion of the theorem holds with
\[
 C = 6 \tilde C
 \quad \text{and} \quad
 \alpha_0 =\max \left\{ \tilde\alpha_0, C \rho^2  \lVert  b  \rVert_{\infty}^2 \Elliptic^{3/2},
C^{1/3} \rho^{4/3} \lVert c \rVert_{\infty}^{2/3}\sqrt{\Elliptic}   \right\} .
\]
\end{remark}
\begin{proof}[Proof of Theorem~\ref{thm:three_annuli}]
 In order to use the Cacciopoli inequality we need slightly narrower auxiliary annuli.
 Thus we introduce $r_1' = r_1 + (R_1 - r_1)/4$,  $R_1' = R_1 - (R_1 - r_1)/4$, $r_3' = r_3 + (R_3 - r_3)/4$, $R_3' = R_3 - (R_3 - r_3)/4$, and the subsets
 \begin{equation*}
  Z_1' = \ballc{R_1'} \setminus \overline{\ballc{r_1'}} \subset Z_1
  \quad \text{and} \quad
  Z_3' = \ballc{R_3'} \setminus \overline{\ballc{r_3'}} \subset Z_3 .
 \end{equation*}
Furthermore, we choose a cutoff function $\eta \in C_{\mathrm c}^\infty (\RR^{d})$ with $0 \leq \eta \leq 1$, $\supp \eta \subset \ballc{R_3'} \setminus \overline{\ballc{r_1'}}$,
$\eta (x) = 1$ for all $x \in \ballc{r_3'} \setminus \overline{\ballc{R_1'}}$,
and
\begin{equation}\label{eq:Theta_eta}
\begin{split}
\max\{\lVert \Delta \eta \rVert_{\infty , Z_1'} , \lVert \lvert \nabla \eta \rvert \rVert_{\infty , Z_1'}\} \leq \frac{\tilde\Theta}{(R_1 - r_1)^2} =: & \Theta_1
\\
 \max\{\lVert \Delta \eta \rVert_{\infty , Z_3'} , \lVert \lvert \nabla \eta \rvert \rVert_{\infty , Z_3'}\} \leq \frac{\tilde\Theta}{(R_3 - r_3)^2} =: &\Theta_3 ,\\
\end{split}
\end{equation}
where $\tilde\Theta$ depends only on the
dimension.\footnote{Compared to the published version \cite{TautenhahnV-20} set $\tilde\Theta:=\max\{1,\tilde\Theta_1,\tilde\Theta_3\}.$ }
We set
\[
 \rho := R_3
 \quad\text{and}\quad
 \mu  :=  33 d \rho \Elliptic^{11/2} \Lipschitz + \epsilon
\]
and fix $\psi \in \Dom (H)$.
In order to apply the Cacciopoli and the Carleman inequality we will approximate the function
$\psi$ in the domain
by smoother functions.
By \eqref{eq:core} there is a sequence $(\psi_n)_{n \in \NN}$ in $C_{\mathrm{c}}^\infty (\RR^d)$ such that $\psi_n \to \psi$ and $H \psi_n \to H \psi$ in $L^2 (\RR^d)$. We apply the Carleman estimate from Theorem~\ref{thm:carleman} to the function $u = \eta \psi_n$ and obtain for all
$\alpha \geq \alpha_0 =
\alpha_0 (d, \rho  , \Elliptic , \Lipschitz, \mu , \lVert b \rVert_\infty , \lVert c \rVert_\infty)$ and all $n \in \NN$
\begin{equation} \label{eq:starting}
 \int_{\ballc{\rho}} \alpha^3 w^{-1-2\alpha} \lvert \eta \psi_n  \rvert^2  \leq \rho^4 C \int_{\ballc{\rho}} w^{2-2\alpha} \lvert \Op(\eta \psi_n) \rvert^2 =: I_1 ,
\end{equation}
where $C = C (d , \Elliptic , \rho \Lipschitz , \mu) > 0$ and $w = w_{\rho , \mu}$.
By
the Leibniz rule and $(a+b+c)^2 \leq 3(a^2+b^2+c^2)$ this yields that $I_1$ is bounded by
\begin{align}
 I_1 & =  \rho^4 C \int_{\ballc{\rho}} w^{2-2\alpha} \biggl\lvert ( \Op_c \eta ) \psi_n + (\Op \psi_n)\eta + 2 \sum_{i,j=1}^d a^{ij} (\partial_i \eta)(\partial_j \psi_n) \biggr\rvert^2  \nonumber \\
& \leq 3 \rho^4 C  \int_{\ballc{\rho}} w^{2-2\alpha} \biggl( \lvert \Op_c \eta  \rvert^2 \lvert \psi_n \rvert^2 + \lvert \Op \psi_n \rvert^2 \eta^2 + 4 \Bigl \lvert \sum_{i,j=1}^d a^{ij} (\partial_i \eta)(\partial_j \psi_n) \Bigr \rvert^2 \biggr)   \label{eq:Leibnitz} ,
\end{align}
where $\Op_c \eta = -\diver (A \nabla \eta) + b^\T \nabla \eta$. Since $a^{ij} = a^{ji}$ and $A (x) = (a^{ij} (x))_{i,j=1}^d$ is positive definite for all $x \in \ballc{\rho}$, we can apply Cauchy Schwarz and obtain
 \begin{align*}
\Bigl\lvert \sum_{i,j=1}^d a^{ij} (\partial_i \eta)(\partial_j \psi_n) \Bigr \rvert^2 &\leq \Bigl( \sum_{i,j=1}^d a^{ij} (\partial_i \eta)(\partial_j \eta) \Bigr) \Bigl( \sum_{i,j=1}^d a^{ij} (\partial_i \overline{\psi_n})(\partial_j \psi_n) \Bigr)
 \leq \Elliptic \lvert \nabla\eta\rvert^2(\nabla\psi_n^\T A \nabla\overline{\psi_n}) .
\end{align*}
Since $\Op_c \eta \not = 0$ only on $\supp \nabla \eta \subset Z_1' \cup Z_3'$ we have by using
the bounds \eqref{eq:Theta_eta} on the function $\eta$
\begin{multline*}
 I_2 :=
 \int_{\ballc{\rho}} w^{2-2\alpha} \biggl( \lvert \Op_c \eta  \rvert^2 \lvert \psi_n \rvert^2 + 4 \Bigl \lvert \sum_{i,j=1}^d a^{ij} (\partial_i \eta)(\partial_j \psi_n) \Bigr \rvert^2 \biggr) \\
 \leq \int_{Z_1'} w^{2-2\alpha}
       \left( (\Op_c \eta)^2 \lvert \psi_n \rvert^2 + 4 {\Elliptic} \Theta_1^2 {\nabla\psi_n^\T A \nabla\overline{\psi_n}} \right) \\
 +  \int_{Z_3'} w^{2-2\alpha} \Bigl( (\Op_c \eta)^2 \lvert \psi_n \rvert^2 + 4 {\Elliptic} \Theta_3^2 {\nabla\psi_n^\T A \nabla\overline{\psi_n}} \Bigr)   .
\end{multline*}
We use the bound on the weight function from Ineq.~\eqref{eq:weightbound2} and obtain
\begin{multline*}
 I_2 \leq \left( \frac{\rho \sqrt{\Elliptic} \mu_1}{r_1'} \right)^{2\alpha - 2} \int_{Z_1'}
       \left( (\Op_c \eta)^2 \lvert \psi_n \rvert^2 + 4 {\Elliptic} \Theta_1^2 {\nabla\psi_n^\T A \nabla\overline{\psi_n}} \right)
       \\
 + \left( \frac{\rho \sqrt{\Elliptic} \mu_1}{r_3'} \right)^{2\alpha - 2} \int_{Z_3'}  \Bigl( (\Op_c \eta)^2 \lvert \psi_n \rvert^2 + 4 {\Elliptic} \Theta_3^2 {\nabla\psi_n^\T A \nabla\overline{\psi_n}} \Bigr)   .
\end{multline*}
Now we use the pointwise estimate
\[
 \lvert \Op_c \eta \rvert^2
 \leq  3\Elliptic^2 \lvert \Delta \eta \rvert^2
 +     3 \Elliptic^2 (2d - 1)^2 \frac{\lvert \nabla \eta \rvert^2}{\lvert x \rvert^2}
 +    3 (\Lipschitz d^2 + \lVert b \rVert_\infty)^2 \lvert \nabla \eta \rvert^2 ,
\]
see \cite[Appendix~A]{BorisovTV-17}, and obtain again by using the properties of the function $\eta$ that $I_2$ is bounded from above by
\begin{align*}
 &\left(\frac{\rho \mu_1 \sqrt{\Elliptic}}{r_1'}\right)^{2\alpha - 2}
 \Theta_1^2
 \int_{Z_1'}
       \left[  \left(3 \Elliptic^2 + \frac{12 \Elliptic^2 d^2}{r_1'^2}  + 3  (\Lipschitz d^2 + \lVert b \rVert_\infty )^2 \right) \lvert \psi_n \rvert^2 + 4{\Elliptic}  {\nabla\psi_n^\T A \nabla\overline{\psi_n}} \right]
       \\
 + &\left(\frac{\rho \mu_1 \sqrt{\Elliptic}}{r_3'}\right)^{2\alpha - 2} \Theta_3^2
 \int_{Z_3'}  \left[  \left(3 \Elliptic^2 +  \frac{12 \Elliptic^2 d^2}{r_3'^2} + 3 (\Lipschitz d^2 + \lVert b \rVert_\infty )^2\right) \lvert \psi_n \rvert^2 + 4 {\Elliptic}  {\nabla\psi_n^\T A \nabla\overline{\psi_n}}  \right]  .
\end{align*}
Recall that by assumption we have $\lvert H \psi \rvert \leq \lvert V \psi \rvert + \lvert \zeta \rvert$ almost everywhere on $\ballc{R_3}$. Hence,
\[
\lvert \Op \psi_n \rvert = \lvert H \psi_n \rvert \leq \lvert H \psi \rvert + \lvert H (\psi - \psi_n) \rvert \leq \lvert V \psi \rvert + \lvert \zeta \rvert + \lvert H (\psi - \psi_n) \rvert.
\]
An application of Lemma~\ref{lemma:Cacciopoli} with
$\xi = \xi_n := \lvert \zeta \rvert +
\lvert H (\psi - \psi_n) \rvert +\vert V(\psi-\psi_n)\vert$,
and $\rho_1 = r_1'$, $\rho_2 = R_1'$ and $\kappa = (R_1 - r_1)/4$ (i.e.\ $\omega = Z_1'$ and $\omega_+ = Z_1$) for the first summand and $\rho_1 = r_3'$, $\rho_2 = R_3'$ and $\kappa = (R_3 - r_3) / 4$ (i.e.\ $\omega = Z_3'$ and $\omega_+ = Z_3$) for the second summand gives
\begin{multline} \label{eq:I2upper}
 I_2 \leq
       \left(\frac{\rho \mu_1 \sqrt{\Elliptic}}{r_1'}\right)^{2\alpha - 2} \Theta_1^2
       \left[ 3 \Elliptic^2 + \frac{12 \Elliptic^2 d^2}{r_1'^2} + 3(\Lipschitz d^2 + \lVert b \rVert_\infty )^2   + 4{\Elliptic}  \Cac_{(R_1 - r_1)/4} \right]       \lVert \psi_n \rVert_{Z_1}^2 \\
 +  \left(\frac{\rho \mu_1 \sqrt{\Elliptic}}{r_3'}\right)^{2\alpha - 2} \Theta_3^2 \left[ 3 \Elliptic^2 + \frac{12\Elliptic^2 d^2}{r_3'^2} + 3 (\Lipschitz d^2 + \lVert b \rVert_\infty )^2  + 4 {\Elliptic}  \Cac_{(R_3 - r_3) / 4}  \right] \lVert \psi_n \rVert_{Z_3}^2  \\
 +  \left(\frac{\rho \mu_1 \sqrt{\Elliptic}}{r_1'} \right)^{2\alpha - 2} 8 (\Theta_1^2 + \Theta_3^2 )
 {\Elliptic} \lVert \xi_n \rVert^2_{Z_1 \cup Z_3}
 \\ =: \bar D_1 \lVert \psi_n \rVert_{Z_1}^2 + \bar D_2\lVert \psi_n \rVert_{Z_3}^2 + \bar D_3 \lVert \xi_n \rVert^2_{Z_1 \cup Z_3} ,
\end{multline}
where $\Cac_{\kappa} = \Cac_\kappa (V , b , c , \Elliptic)$, $\kappa > 0$, is defined in Lemma~\ref{lemma:Cacciopoli}. From \eqref{eq:starting}, \eqref{eq:Leibnitz}, and \eqref{eq:I2upper}, we obtain by using $\Op = H$ on $C_{\mathrm{c}}^\infty$, $\psi_n \to \psi$ and $H \psi_n \to H \psi$ in $L^2 (\RR^d)$, and by taking the limit $n \to \infty$ that
\begin{equation} \label{eq:zwischenresultat}
\frac{\alpha^3}{3 \rho^4 C} \int_{\ballc{\rho}}  w^{-1-2\alpha} \lvert \eta \psi  \rvert^2
\leq \int_{\ballc{\rho}} w^{2-2\alpha} \lvert H \psi \rvert^2 \eta^2 +
\bar D_1 \lVert \psi \rVert_{Z_1}^2 + \bar D_2\lVert \psi \rVert_{Z_3}^2 + \bar D_3 \lVert \zeta \rVert^2_{Z_1 \cup Z_3} .
\end{equation}
The pointwise estimate $\lvert H \psi \rvert \leq \lvert V \psi \rvert + \lvert \zeta \rvert$, and $w \leq \sqrt{\Elliptic}$ on $\ballc{\rho}$ gives
\begin{equation} \label{eq:subsolution}
 \int_{\ballc{\rho}} w^{2-2\alpha} \lvert H \psi \rvert^2\eta^2  \leq   2 \lVert V \rVert_\infty^2 \Elliptic^{3/2} \int_{\ballc{\rho}} w^{-1-2\alpha}  \lvert \eta \psi \rvert^2  +
  2 \int_{\ballc{\rho}\setminus \ballc{r_1'}} w^{2-2\alpha}  \lvert \eta \zeta \rvert^2 .
\end{equation}
From Ineq.~\eqref{eq:zwischenresultat} \& \eqref{eq:subsolution},
and our bounds \eqref{eq:weightbound2} on the weight function we obtain for all $\alpha \geq \alpha_0$
\begin{equation*}
 \left[ \frac{\alpha^3}{3 \rho^4 C} - 2 \lVert V \rVert_\infty^2 \Elliptic^{3/2}  \right] \int_{\ballc{\rho}} w^{-1-2\alpha} \lvert \eta \psi \rvert^2
\leq  \bar D_1 \lVert \psi \rVert_{Z_1}^2 + \bar D_2\lVert \psi \rVert_{Z_3}^2 + \hat D_3 \lVert \zeta \rVert^2_{\ballc{R_3}} ,
\end{equation*}
where
\[
 \hat D_3 = \bar D_3+ 2 \left(\frac{\rho \mu_1 \sqrt{\Elliptic}}{r_1'} \right)^{2\alpha - 2} = \left(\frac{\rho \mu_1 \sqrt{\Elliptic}}{r_1'} \right)^{2\alpha - 2} \bigl(8 (\Theta_1^2 + \Theta_3^2)\Elliptic + 2 \bigr)
\]
We choose
\begin{equation*} 
 \alpha \geq \alpha^*:= \max\{\alpha_0,\alpha_1, 1\}, \ \text{ where } \
 \alpha_1 := \sqrt[3]{16 \rho^4 C \lVert V \rVert_\infty^2 \Elliptic^{3/2}}.
\end{equation*}
and $\alpha_0$ is as in Theorem \ref{thm:carleman}.
This ensures that
\[
\frac{5}{24} \frac{\alpha^3 }{\rho^4 C }  \int_{\ballc{\rho}}  w^{-1-2\alpha} \lvert \eta \psi \rvert^2    \leq \bar D_1 \lVert \psi \rVert_{Z_1}^2 + \bar D_2\lVert \psi \rVert_{Z_3}^2 + \hat D_3 \lVert \zeta \rVert^2_{\ballc{R_3}}  .
\]
Since $\eta \equiv 1$ on $Z_2$ and by our bound on the weight function we have
\begin{equation*}
  \frac{5}{24} \frac{\alpha^3}{\rho^4 C}  \left( \frac{\rho}{\sqrt{\Elliptic} R_2} \right)^{1+2\alpha} \lVert \psi \rVert_{Z_2}^2 \leq
  \bar D_1 \lVert \psi \rVert_{Z_1}^2 + \bar D_2\lVert \psi \rVert_{Z_3}^2 + \hat D_3 \lVert \zeta \rVert^2_{\ballc{R_3}}  .
\end{equation*}
Hence, we have shown the statement of the theorem with
\begin{align*}
 D_1 &= \frac{24}{5}  R_3 C \Elliptic^{-1/2} \mu_1^{-2} r_1'^2 R_2 \Theta_1^2
 \left[ 3 \Elliptic^2 + \frac{12 \Elliptic^2 d^2}{r_1'^2} + 3(\Lipschitz d^2 + \lVert b \rVert_\infty )^2   + 4{\Elliptic}  \Cac_{(R_1 - r_1)/4} \right] ,
 \\
 D_2 &= \frac{24}{5}  R_3 C \Elliptic^{-1/2} \mu_1^{-2} r_3'^2 R_2 \Theta_3^2
 \left[ 3 \Elliptic^2 + \frac{12\Elliptic^2 d^2}{r_3'^2} + 3 (\Lipschitz d^2 + \lVert b \rVert_\infty )^2  + 4 {\Elliptic}  \Cac_{(R_3 - r_3) / 4}  \right] , \\
 D_3 &= \frac{24}{5} R_3 C \Elliptic^{-1/2} \mu_1^{-2} r_1'^2 R_2 \bigl[8(\Theta_1^2 + \Theta_3^2)\Elliptic + 2 \bigr] . \qedhere
\end{align*}
\end{proof}

\newpage
As a corollary to the proof of Theorem \ref{thm:three_annuli} we obtain:

\begin{corollary}[Three annuli inequality on cubes] \label{thm:three_annuli_cubes}
Let $L \in \NN$, $R=3L$, $0 < r_1 < R_1 \leq r_2 < R_2 \leq r_3 < R_3:= (33 \euler d \Elliptic^{11/2} (\Lipschitz + 1))^{-1}$, $\epsilon > 0$,
$\mu = 33 d R_3 \Elliptic^{11/2} \Lipschitz + \epsilon $, and
\begin{equation}\label{eq:mu1-repeated}
 \mu_1 = \mu_1(R_3,\epsilon) =
        \begin{cases}
          \exp(\mu \sqrt{\Elliptic})    & \text{if $\mu \sqrt{\Elliptic} \leq 1$,}\\
          \euler \mu \sqrt{\Elliptic}   & \text{if $\mu \sqrt{\Elliptic} > 1$},
         \end{cases}
\end{equation}
Then, for all measurable and bounded $V \colon \Lambda_R \to \RR$ there are constants (the same as in Theorem \ref{thm:three_annuli})
$\alpha^* \geq 1$ and $D_i > 0$, $i \in \{1,2,3\}$,
depending merely on $r_j$, $R_j$, $j \in \{1,2,3\}$, $\epsilon$, $d$,
$\Elliptic$, $\Lipschitz$, $\lVert V \rVert_\infty$, $\lVert b \rVert_\infty$, and $\lVert c \rVert_\infty$,
such that for all   $x\in \Lambda_L$, for all
$\psi \in \Dom (H_R)$ and  $\zeta \in L^2 (\Lambda_R)$ satisfying $\lvert H_R \psi \rvert \leq \lvert V\psi \rvert + \lvert \zeta \rvert$
almost everywhere on $\ball{R_3}{x}$, and all $\alpha \geq \alpha^*$ we have
\[
 \alpha^3 \lVert \psi \rVert_{Z_2(x)}^2
 \leq
 D_1 \left(\frac{R_2 \mu_1 \Elliptic}{r_1}\right)^{2\alpha} \lVert \psi \rVert_{Z_1(x)}^2
 +
 D_2 \left(\frac{R_2 \mu_1 \Elliptic}{r_3}\right)^{2\alpha} \lVert \psi \rVert_{Z_3(x)}^2
 +
 D_3 \left(\frac{R_2 \mu_1 \Elliptic}{r_1}\right)^{2\alpha} \lVert \zeta \rVert_{\ball{R_3}{x}}^2.
\]
\end{corollary}
	
	%

\begin{proof}
Note that $R_3= (33 \euler d \Elliptic^{11/2} (\Lipschitz + 1))^{-1}< 1/89$.
Hence, $\ball{R_3}{x} \subset \Lambda_{3L/2}$ for all $x  \in \Lambda_{L}$.

For $\psi \in \Dom (H_R)$ and $\chi\in C_0^\infty(\Lambda_R)$, $\1_{\Lambda_{3L/2}}\leq \chi\leq \1_{\Lambda_{2L}}$ we claim that
\[
\chi \psi \in\Dom (H_R) 
\]
Indeed, since $\psi \in \Dom ( H_R ) \subset \Dom (a_R)$ there exists $w \in L^2 (\Lambda_{R})$ such that $ H_R \psi = w$. By the first representation theorem, in particular \cite[Theorem VI.2.1 part (i)]{Kato-80}, we have for all $v \in \Dom (a_R)$ that $a_R (\psi, v) = \langle w , v \rangle$. 
Since $\chi \psi \in \Dom (a_R)$ we obtain, using the product rule and integration by parts,
\[
  a_R (\chi \psi , v) = \langle \tilde w , v \rangle
\]
for all $v \in \Dom (a_R)$, where $\tilde w \in L^2 (\Lambda_{R})$ is given by
\[
\tilde w = \chi w + (b^\T \nabla \chi)  \psi - \nabla \chi^\T A \nabla  \psi - \diver (\psi A \nabla \chi)  .
\]
Again the first representation theorem, see \cite[Theorem VI.2.1 part (iii)]{Kato-80}, implies that $\chi \psi \in \Dom (H_R)$ (and $ H_R(\chi \psi) = \tilde w$). 
This proves the claim.
\smallskip

Let  $\widehat H_R$ be an extension of $H_R$ to $L^2(\RR^d)$ with coefficient functions of the type considered at the beginning of Section \ref{sec:results},
(i.e.~uniformly Lipschitz-continuous, uniformly elliptic, and symmetric  $A \colon \RR^d \to \RR^{d \times d}$, $b \in L^\infty (\RR^d ; \CC^d)$ and $c \in L^\infty (\RR^d)$),
coinciding on $\Lambda_R$ with those of $H_R$,  but  arbitrary on $\RR^d\setminus \Lambda_R$.
If we extend $\chi \psi$ by zero outside $\Lambda_{R}$ and consider it as an element of $L^2 (\RR^d)$ we find, using that our operators are local, $ \chi \psi \in\Dom (\widehat H_R)$.
\medskip

Since $C_0^\infty(\RR^d)$  is an operator core for $\hat H_R$
there exist a sequence $\psi_n \in C_0^\infty(\RR^d)$ with
\begin{itemize}
  \item $\psi_n \to \chi\psi \text{ in }  L^2(\RR^d)$
  \item $\hat H_R\psi_n \to \hat H_R(\chi\psi) \text{ in }  L^2(\RR^d)$
  \item $\supp \psi_n \subset \Lambda_{5/2L}$
\end{itemize}
Now the statement of the Corollary follows with the same arguments as in the proof
of Theorem \ref{thm:three_annuli}.
At the end of the proof one has to use that  $\widehat H_R(\chi \psi)$ equals $H_R( \psi)$ almost everywhere on $\ball{R_3}{x}$ since the operators are local.
\end{proof}
\bigskip
%
%
%
%
%
%

\section{Intermezzo: Short proof for the case of small Lipschitz constant coefficients}
\label{sec:intermezzo}
In this section we show how the three annuli inequality from Theorem~\ref{thm:three_annuli}
yields directly a proof of the sampling and equidistribution Theorems~\ref{thm:sampling}
\& \ref{thm:equidistribution}
in the special case where the second order part is the pure Laplacian.
This way we recover in particular Theorem 2.1 in \cite{RojasMolinaV-13}
and Theorem 1 in \cite{TautenhahnV-16}.
Let us note that our new proof is much shorter and simpler than the earlier ones from \cite{RojasMolinaV-13,TautenhahnV-16}.
\par
Thereafter we explain how this method extends to elliptic second order terms with sufficiently small Lipschitz constants, in particular to constant coefficients.
This way one can recover the results from \cite{BorisovTV-17} with a simplified proof compared to the original one in  \cite{BorisovTV-17}, which was based on the method of \cite{RojasMolinaV-13}.
We also discuss why this direct approach fails for second order terms with arbitrary Lipschitz coefficients.
\begin{theorem} \label{thm:Laplacian}
 Assume that $L\in \NN_\infty$,  and $a^{ij} (x) = \delta_{ij}$ for all $i,j \in \{1,\ldots , d\}$ and $x \in \Lambda_L$.
  There is a constant $N = N (d)$, such that for all measurable and bounded $V_L : \Lambda_L \to \RR$, all $\psi \in \Dom (H_L)$ satisfying $\lvert H_L \psi \rvert \leq \lvert V_L\psi \rvert$ almost everywhere on $\Lambda_L$, all $\delta \in (0,1/2)$ and all $(1,\delta)$-equi\-distri\-buted sequences $Z$ we have
\[
 \lVert \psi \rVert_{\equi_{\delta,Z} \cap \Lambda_L}^2 \geq C_{\sfUC} \lVert \psi \rVert_{\Lambda_L}^2,
\]
where
\begin{equation*} 
 C_{\sfUC} = \delta^{N \bigl( 1+  \lVert V_L \rVert_\infty^{2/3} + \lVert b_L \rVert_\infty^{2} + \lVert c_L \rVert_\infty^{2/3} \bigr)} .
\end{equation*}
\end{theorem}
\begin{proof} (I) Let us first consider the case $\Lambda_L=\RR^d$, i.e.\ $L=\infty$.
Since $A$ is by assumption the identity matrix, we have $\Elliptic = 1$ and $\Lipschitz = 0$.
We choose $\epsilon = 1$,
hence $\mu=1$ and $\mu_1 =\euler$. We also choose
\begin{align*}
  r_1 &= \delta / 2,
& r_2 &=1,
&r_3 &= 6 \mathrm{e} \sqrt{d}, \\
R_1 &=\delta, &
R_2 &=3 \sqrt{d}, &
R_3 &= 9 \mathrm{e} \sqrt{d} .
\end{align*}
We apply Theorem~\ref{thm:three_annuli} and Lemma~\ref{lemma:three_annuli} with these choices of the radii to translates of the sets $Z_i$, $i \in \{1,2,3\}$, and obtain for all $\alpha \geq \alpha^*$
\[
 \sum_{j\in\ZZ^d} \lVert \psi \rVert_{Z_2 + z_j}^2
\leq
D_1\left( \frac{\mathrm{e} R_2}{r_1} \right)^{2\alpha} \sum_{j\in\ZZ^d} \lVert \psi \rVert_{Z_1 + z_j}^2
+
D_2\left( \frac{\mathrm{e} R_2}{r_3} \right)^{2\alpha}\sum_{j\in\ZZ^d} \lVert \psi \rVert_{Z_3 + z_j}^2
\]
where $z_j$, $j \in \ZZ^d$, denote the elements of the $(1,\delta)$-equidistributed sequence $Z$. From Lemma~\ref{lemma:three_annuli} we infer that
\begin{align*}
   D_1 &\leq K \delta^{-4} \left(1 + \lVert V \rVert_\infty^2 + \lVert b \rVert_\infty^2 + \lVert c \rVert_\infty^2 \right), \quad
   D_2 \leq K \left( 1+\lVert V \rVert_\infty^2 + \lVert b \rVert_\infty^2 + \lVert c \rVert_\infty^2  \right),
  \end{align*}
and
\[
   \alpha^* \leq K \left(1 + \lVert V \rVert_\infty^{2/3} + \lVert b \rVert_\infty^2 + \lVert c \rVert_\infty^{2/3} \right) ,
\]
where $K$ is a constant depending only on the dimension. A covering argument gives
\begin{equation} \label{eq:covering_Laplace}
\sum_{j\in\ZZ^d} \lVert \psi \rVert_{Z_2 + z_j}^2 \geq \lVert \psi \rVert_{\RR^d}^2, \quad
\sum_{j\in\ZZ^d} \lVert \psi \rVert_{Z_1 + z_j}^2 \leq \lVert \psi \rVert_{\equi_{\delta,Z}}^2 , \quad \text{and} \quad
\sum_{j\in\ZZ^d} \lVert \psi \rVert_{Z_3 + z_j}^2 \leq K_d \lVert \psi \rVert_{\RR^d}^2 ,
\end{equation}
where $K_d=(18 \euler \sqrt{d}+1)^d$ is a combinatorial factor depending only on the dimension. Hence,
\[
 \lVert \psi \rVert_{\RR^d}^2
\leq
D_1 \left( \frac{\euler R_2}{r_1} \right)^{2\alpha} \lVert \psi \rVert_{\equi_{\delta,Z}}^2
+
D_2 \left( \frac{\euler R_2}{r_3} \right)^{2\alpha}K_d \lVert \psi \rVert_{\RR^d}^2 .
\]
This is equivalent to
\begin{equation}\label{eq:0}
 \left(1 - K_dD_2  \left(\frac{\euler R_2}{r_3} \right)^\alpha\right)\lVert \psi \rVert_{\RR^d}^2
\leq
D_1 \left( \frac{\euler R_2}{r_1} \right)^{2\alpha} \lVert \psi \rVert_{\equi_{\delta,Z}}^2 .
\end{equation}
Since $(\euler R_2 / r_3) = 1/2 < 1$ we can, additionally to $\alpha \geq \alpha^*$,  choose $\alpha$ sufficiently large, i.e.
\[
 \alpha \geq \frac{\ln (2 K_d D_2)}{\ln 4} =: \alpha^{**},
\]
such that the prefactor on the left hand side of Ineq.~\eqref{eq:0} is bounded from below by $1/2$. Hence, we obtain for $\alpha_0 := \max\{\alpha^* , \alpha^{**}\}$
\[
\lVert \psi \rVert_{\RR^d}^2
\leq
2D_1 \left( \frac{\euler R_2}{r_1} \right)^{2\alpha_0} \lVert \psi \rVert_{\equi_{\delta,Z}}^2 .
\]
We denote by $K_i$, $i\in \NN$ constants depending only on the dimension (which may change from line to line) and calculate for the final constant
\begin{align*}
 2D_1 \left( \frac{\euler R_2}{r_1} \right)^{2\alpha_0}
 & \leq  K_1 \delta^{-4} \left(1 + \lVert V \rVert_\infty^2 + \lVert b \rVert_\infty^2 + \lVert c \rVert_\infty^2 \right)
 \left( \frac{\delta}{6 \mathrm{e} \sqrt{d}} \right)^{-2(\alpha^* + \alpha^{**})} \\
 &\leq K_1 \left(1 + \lVert V \rVert_\infty^2 + \lVert b \rVert_\infty^2 + \lVert c \rVert_\infty^2 \right)
 \left( \frac{\delta}{6 \mathrm{e} \sqrt{d}} \right)^{-2(\alpha^* + \alpha^{**}) - 4} \\
 &\leq
 \left( \frac{\delta}{6 \mathrm{e} \sqrt{d}} \right)^{-2(\alpha^* + \alpha^{**}) - 4 - \ln \left(1 + \lVert V \rVert_\infty^2 + \lVert b \rVert_\infty^2 + \lVert c \rVert_\infty^2 \right) - \ln K_1} .
\end{align*}
For the exponent we have using $\ln (1+x) \leq 3x^{1/3}$ and $x^{2/3} \leq 1+x^2$
\[
 2(\alpha^* + \alpha^{**}) + 4 + \ln \left(1 + \lVert V \rVert_\infty^2 + \lVert b \rVert_\infty^2 + \lVert c \rVert_\infty^2 \right) + K_1
 \leq K_2 \left(1 + \lVert V \rVert_\infty^{2/3} + \lVert b \rVert_\infty^2 + \lVert c \rVert_\infty^{2/3} \right) .
\]
Hence, with $K_3=(1+2 \ln (6 \euler \sqrt{d}))K_2$ we have
\begin{align*}
 2D_1 \left( \frac{\euler R_2}{r_1} \right)^{2\alpha}
 &\leq
 \left( \frac{\delta}{6 \mathrm{e} \sqrt{d}} \right)^{-K_2 \left(1 + \lVert V \rVert_\infty^{2/3} + \lVert b \rVert_\infty^2 + \lVert c \rVert_\infty^{2/3} \right)}
 \leq \delta^{-K_3 \left(1 + \lVert V \rVert_\infty^{2/3} + \lVert b \rVert_\infty^2 + \lVert c \rVert_\infty^{2/3} \right)} .
\end{align*}
\par
(II) If $L\in \NN$, i.e.\ $\Lambda_L$ is a finite cube
the first step of the proof consists in extending the original problem on
$\Lambda_L$ to the whole of $\RR^d$ using the extension which is constructed in Appendix~\ref{sec:extensions}. To the resulting problem one can then apply the arguments
of part (I) of the proof.
This is analogous to the proof of Theorem \ref{thm:equidistribution}, to which we refer for details.
\end{proof}
\begin{remark}\label{rem:small-Lipschitz}
 Crucial for the proof of Theorem~\ref{thm:Laplacian} are
 \begin{enumerate}[(i)]
  \item the first covering bound in Ineq.~\eqref{eq:covering_Laplace}, and
  \item the fact that $ K_d D_2  (\mathrm{e} R_2 / r_3 )^{2\alpha} < 1$, in order that the left hand side of Ineq.~\eqref{eq:0} can be bounded from below.
 \end{enumerate}
Since $ K_d D_2  >1$ the only way to ensure (ii) is to guarantee that $\Elliptic R_2 \mu_1 / r_3 = \euler R_2 / r_3 < 1$,
and then choose $\alpha$ large enough.
\par
In the case of the pure Laplacian, (i) and (ii) are true due to the proper choice of $r_i$, $R_i$, and $\epsilon$
(actually:  $R_2=3 \sqrt{d}$ and $r_3=6 \mathrm{e}\sqrt{d}$, and $\Elliptic=1$ and $\mu=\epsilon=1$, so that $\mu_1=\mathrm{e}$).
If one attempts to apply this proof to variable second order coefficients, then it is in general not possible to verify (i) and (ii) simultaneously.
\par
On the one hand, in the general case one has to pick the radii $R_3 $, $R_2$ and $r_3$ as
functions of $d$, $\Elliptic$, and $\Lipschitz$, in order to satisfy (ii).
If all three radii  are proportional to  a sufficiently negative power of $\Elliptic(\Lipschitz+1)$,
then indeed (ii) can be achieved.%
\footnote{In fact, this choice will be what happens in the first step in the proof of Theorem \ref{thm:chaining}:
To ensure (ii) we choose the radii for instance as in Lemma~\ref{lemma:interpolation}.}
Note that this forces $R_2$ to be small (depending on $\Elliptic$ and $\Lipschitz$). However, once $R_2<\sqrt{d}$, the union of the annuli $Z_2+z_j$ will no longer cover all of $\RR^d$,
thus we cannot have (i). On the other hand, if one chooses the radii such that (i) holds, then $\Elliptic R_2 \mu_1 / r_3$ is smaller than one only if $\mu_1$ is sufficiently small.
The latter can be achieved by choosing $\Lipschitz$ and $\epsilon$ sufficiently small as a function of $\Elliptic$.
\par
This is why the above proof for the Laplacian
can be extended to second order terms with slowly varying coefficient functions
but \emph{not} for divergence form operators with \emph{arbitrary} coefficient functions
as considered in this paper.
\end{remark}
%
%
%
%
\section{Chaining argument and the proof of Theorems \ref{thm:sampling} and \ref{thm:equidistribution}}
\label{sec:chaining}
We discussed in Remark~\ref{rem:small-Lipschitz} in the last section,
why for arbitrary Lipschitz constants a sampling or equidistribution theorem does not directly follow from the three annuli inequality.
This is also the reason why the results in \cite{BorisovTV-17} were limited to slowly varying coefficient functions.

In this section we present a method how to overcome this limitation.
First we deduce an adapted \emph{interpolation inequality} from the three annuli inequality.
Then we apply a so-called \emph{chaining argument} similar to the one in \cite{Bakri-13},\footnote{Also, due to the inhomogeneity $\zeta$ our chaining argument needs a careful balancing of the terms involving $\psi$ and $\zeta$.}
in order to obtain a different covering bound replacing the one in Ineq.~\eqref{eq:covering_Laplace}.
In our situation the chaining is performed simultaneously in all periodicity cells.

Recall the conventions
$\NN_{\infty}:=\NN\cup\{\infty\}$, $\Lambda_\infty:= \RR^d$, and $H_\infty := H$.

\begin{theorem}[Interpolation inequality] \label{thm:interpolation}
Let $R \in \NN_{\infty}$,
$ \epsilon > 0$,  $0 < r_1 < R_1 \leq r_2 < R_2 \leq r_3 < R_3$ such that
\begin{equation}\label{ass:radii}
\mu_1 = \mu_1(R_3,\epsilon) < \frac{r_3}{R_2 \Elliptic}
\quad\text{and}\quad
\frac{(\mu_1 R_2 \Elliptic)^2}{r_1 r_3} \geq 1
\end{equation}
and $\Omega_*\subset \Lambda_R$ be open.

Then for all measurable and bounded $V \colon \Lambda_R \to \RR$,
all $\psi \in \Dom (H_R)$
and $\zeta \in L^2 (\Lambda_R)$ satisfying $\lvert H_R \psi \rvert \leq \lvert V\psi \rvert + \lvert \zeta \rvert$
almost everywhere on $\Omega_*$,
all $ \cJ \subset \ZZ^d$,
all sequences $(x_j)_{j \in  \cJ  } \subset \RR^d$ satisfying
\[
\forall j \in  \cJ   \colon \quad x_j \in \Lambda_{1+2a}(j) \text{ and }
 \ball{R_3}{x_j} \subset \Omega_*, \quad \text{ where } a= (R_2+3r_2)/4,
\]
all $t \geq 0$, {and all $\D_3 \geq D_3$} we have
\begin{multline}\label{eq:3annuli0}
\sum_{j \in  \cJ  } \lVert \psi \rVert_{Z_1 (x_j)}^2 +  \left( \frac{t}{2 \tilde D_2}+1 \right) \frac{\D_3 M}{\tilde D_1} \lVert \zeta \rVert_{ \Omega_*  }^2 \\
\geq
C_1(\gamma)^{-1/\gamma} \biggl( M \lVert \psi \rVert_{ \Omega_*  }^2 + \frac{t \D_3 M}{2}  \lVert \zeta \rVert^2_{ \Omega_*  } \biggr)^{1-1/\gamma}
\biggl( \sum_{j \in  \cJ  } \lVert \psi \rVert_{Z_2 (x_j)}^2 + t \D_3 M \lVert \zeta \rVert_{ \Omega_*  }^2 \biggr)^{1/\gamma} .
\end{multline}
Here,
\[
 C_1 (\gamma) = 2 \left(\frac{\tilde D_1}{\tilde D_2} \right)^\gamma \max \left\{ \tilde D_2  ,    \left( \frac{r_3}{r_1} \right)^{2\gamma \alpha^*} \right  \} > 0 ,
 \quad
 \gamma
 = \frac{\ln (r_3 / (R_2 \mu_1 \Elliptic))}{\ln (r_3 / r_1)} \in (0,1),
\]
$M = (2R_3 + 2a+ 1)^d$, $\tilde D_1 = \max \{ 1, D_1 \}$, $\tilde D_2 = \max \{ 1, D_2 \}$, and $D_1$, $D_2$, $D_3$, $\alpha^*$ and $\mu_1$ are as in Theorem~\ref{thm:three_annuli}.
\end{theorem}
A particular important case is $\cJ =\ZZ^d$ and $\Omega_* = \Lambda_R=\RR^d$.

Note that condition \eqref{ass:radii} is equivalent to $\sqrt{r_1 r_3} \leq \mu_1 R_2 \Elliptic \leq r_3$.
\begin{lemma}\label{lemma:interpolation}
 Let $\epsilon = 1$, and set
 \begin{align*}
 r_1&= R_1 / 2   & r_2&= R_2 / 5 & r_3 &= \frac{R_3}{\Elliptic+1} \\
 R_1& \leq r_2   & R_2&=\frac{R_3}{2 \euler (\Elliptic+1)^{5/2}} & R_3 &= (33 \euler d \Elliptic^{11/2} (\Lipschitz + 1))^{-1} .
\end{align*}
Then Assumption~\eqref{ass:radii} is satisfied,
\[
 C_1(\gamma)^{1/\gamma} \leq  R_1^{-K (1 + \lVert V \rVert_\infty^{2/3} + \lVert b \rVert_\infty^{2} + \lVert c \rVert_\infty^{2/3})},
 \quad\text{and}\quad
 \frac{D_3}{\tilde D_1} \leq K R_1^2 ,
\]
where $K\geq 1$ is constant only depending on $d,\Elliptic,\Lipschitz$.
\end{lemma}
\begin{proof}[Proof of Theorem~\ref{thm:interpolation}]
 Since  $R_2> r_1 $ and by condition \eqref{ass:radii}, we have that
 \[
  a_1 :=  \frac{R_2 \mu_1 \Elliptic}{r_1}\in (1,\infty)
  \quad \text{and} \quad
  a_3 := \frac{R_2 \mu_1 \Elliptic}{r_3} \in (0,1) ,
 \]
 where $\mu_1 \geq 1$ is as in \eqref{eq:mu1}.
Applying Theorem~\ref{thm:three_annuli}, respectively Corollary \ref{thm:three_annuli_cubes}, to the translated annuli $Z_i(x_j)$ of the sets $Z_i$, $i \in \{1,2,3\}$,
we obtain for all $\alpha \geq \alpha^* \geq 1$ and all $\D_3 \geq D_3$
 \begin{equation} \label{eq:three}
  \sum_{j \in \cJ} \lVert \psi \rVert_{Z_2 (x_j)}^2
  \leq
  a_1^{2\alpha} \left(  \tilde D_1 \sum_{j \in \cJ} \lVert \psi \rVert_{Z_1 (x_j)}^2 +  \D_3 \sum_{j \in \cJ} \lVert \zeta \rVert_{\ball{R_3}{x_j}}^2 \right)
  +
  a_3^{2\alpha} \tilde D_2 \sum_{j \in \cJ} \lVert \psi \rVert_{\ball{R_3}{x_j}}^2
  .
 \end{equation}
 By assumption on the sequence $(x_j)_{j \in \cJ}$ we have
\begin{equation} \label{eq:cover}
 \sum_{j \in \cJ} \lVert \zeta \rVert_{\ball{R_3}{x_j}}^2
 \leq
 M \lVert \zeta \rVert^2_{\Omega_*} ,
 \quad \text{where} \quad
M = (2R_3 + 2a+ 1)^d
\end{equation}
and the same inequality with $\zeta$ replaced by $\psi$.
Since $2xy \leq sx^2 + s^{-1} y^2$ and $a_1 a_3 \geq 1$ by Assumption \eqref{ass:radii}, we have
\begin{equation} \label{eq:mean}
 \lVert \zeta \rVert^2_{\Omega_*}  \leq a_1^{\alpha} a_3^{\alpha} \lVert \zeta \rVert^2_{\Omega_*}
 \leq \frac{s a_1^{2\alpha}}{2}  \lVert \zeta \rVert^2_{\Omega_*} + \frac{a_3^{2\alpha}}{2s}  \lVert \zeta \rVert^2_{\Omega_*} .
\end{equation}
From \eqref{eq:three}, \eqref{eq:cover} and \eqref{eq:mean} we conclude for all $t \geq 0$ and $s > 0$
\begin{align} \label{eq:summed}
 L:= \sum_{j \in \cJ} \lVert \psi \rVert_{Z_2 (x_j)}^2 + t  \D_3 M \lVert \zeta \rVert^2_{\Omega_*}
  \leq
  a_1^{2\alpha} A_1
  +
  a_3^{2\alpha} A_2
  ,
\end{align}
where
\[
 A_1 =   \tilde D_1 \sum_{j \in \cJ} \lVert \psi \rVert_{Z_1 (x_j)}^2 + \left( \D_3 M + \frac{t  \D_3 M s}{2} \right) \lVert \zeta \rVert^2_{\Omega_*}
 \quad\text{and}\quad
 A_2 = \tilde D_2 M \lVert \psi \rVert^2_{\Omega_*} + \frac{t \D_3 M}{2s} \lVert \zeta \rVert^2_{\Omega_*} .
\]
We choose
\begin{equation} \label{eq:alphahat}
 \hat\alpha := \frac{\ln A_2 - \ln A_1}{2 \ln (a_1) - 2 \ln (a_3)}
 \quad\text{and}\quad
 s = \frac{1}{\tilde D_2} ,
\end{equation}
and we distinguish two cases. If $\hat\alpha \geq \alpha^*$, we apply Ineq.~\eqref{eq:summed} with $\alpha = \hat \alpha$ and obtain
\[
 L
 \leq 2 A_1^{\gamma} A_2^{1-\gamma},
 \quad\text{where}\quad
 \gamma := \frac{-\ln (a_3)}{\ln (a_1/a_3)} \in (0,1) .
\]
Note that $\gamma > 0$ since $a_1 > 1$ and $0< a_3 < 1$, and $\gamma < 1$ since $r_1 < R_2 \mu_1 \Elliptic$. This proves the statement if $\hat\alpha \geq \alpha^*$. If $\hat\alpha < \alpha^*$ we conclude from Eq.~\eqref{eq:alphahat}
\begin{align*}
 A_2 < \left( \frac{a_1}{a_3} \right)^{2 \alpha^*} A_1.
\end{align*}
Thus, if $\hat \alpha < \alpha^*$ we find,
using \eqref{eq:summed} and \eqref{eq:cover} with $\zeta$ replaced by $\psi$,
\[
 L \leq \frac{2}{\tilde D_2} A_2^{1-\gamma+\gamma} < \frac{2}{\tilde D_2}\left( \frac{a_1}{a_3} \right)^{2 \alpha^* \gamma} A_1^\gamma A_2^{1-\gamma} .
\]
Combining the two cases we conclude Ineq.~\eqref{eq:3annuli0}.
\end{proof}
\begin{proof}[Proof of Lemma~\ref{lemma:interpolation}]
We remark that the radii $R_3$, $R_2$, $r_3$, and $r_2$ depend only on $d$, $\Elliptic$, and $\Lipschitz$. Therefore we only emphasize the dependence with respect to $R_1$. The first inequality of Assumption~\eqref{ass:radii} is satisfied since, using $r_3 =2 \euler (\Elliptic + 1)^{3/2} R_2$, $R_3 = (33\euler d \Elliptic^{11/2} (\Lipschitz + 1))^{-1}$, $\mu = 33 d R_3 \Elliptic^{11/2} \Lipschitz + 1 = \Lipschitz / (\euler (\Lipschitz+1))+1$, and $\mu_1 = \euler \sqrt{\Elliptic} \mu$, we find
\[
 \frac{r_3}{R_2  \Elliptic \mu_1}
 = \frac{2 (\Elliptic + 1)^{3/2}}{\Elliptic^{3/2} \left( \frac{\Lipschitz}{\euler (\Lipschitz + 1)} + 1 \right)}
 \geq \frac{2 (\Elliptic + 1)^{3/2}}{2 \Elliptic^{3/2} }
 > 1 .
\]
The second inequality of Assumption~\eqref{ass:radii} follows, using $\mu \geq 1$ and $R_1 \leq r_2$, from
\[
 \frac{(R_2 \mu_1 \Elliptic)^2}{r_1 r_3}
 \geq \frac{R_2^2 \euler^2 \Elliptic^3}{r_1 r_3}
 \geq \frac{\left( \frac{R_3}{2 \euler (\Elliptic+1)^{5/2}} \right)^2 \euler^2 \Elliptic^3}{\left(\frac{R_3}{20 \euler (\Elliptic+1)^{5/2}} \right) \frac{R_3}{\Elliptic+1}}
 =
 \frac{20\euler \Elliptic^3 (\Elliptic+1)}{4 (\Elliptic+1)^{5/2}} > 1  .
\]
We denote by $K \geq 1$ and $K_i > 0$, $i \in \NN$, constants only depending on the model parameters $d$, $\Elliptic$, and $\Lipschitz$. We allow them to change with each occurrence.
In order to estimate $C_1(\gamma)^{1/\gamma}$ we recall that
\[
 C_1 (\gamma)^{1/\gamma} =  \frac{2^{1/\gamma}  \tilde D_1}{\tilde D_2}  \max \left\{ \tilde D_2^{1/\gamma}  ,    \left( \frac{r_3}{r_1} \right)^{2 \alpha^*} \right  \} ,
\]
and we write
\[
\gamma=\frac{\ln(r_3 / (R_2\mu_1\Elliptic))}{\ln(r_3 / r_1)}=\frac{K_1}{\ln(K_2 / R_1)}
\]
with $K_2 > R_1$. Throughout the calculations, we will often use $R_1\leq 1/2$ and the estimate $\ln(1+x)\leq 3x^{1/3}$ for $x \geq 0$.
From Lemma~\ref{lemma:three_annuli} we have
$D_2\leq K(1 +  \lVert V \rVert_\infty^2 + \lVert  b \rVert_\infty^2 + \lVert c \rVert_\infty^2)$. Therefore, we obtain
\begin{align*}
D_2^{1/\gamma}
&\leq\left(\frac{K_2}{R_1}\right)^{K + K \ln(1 + \lVert V \rVert_\infty^2 + \lVert  b \rVert_\infty^2 + \lVert c \rVert_\infty^2)}\\
&\leq K_2^{K(1 + \lVert V \rVert_\infty^{2/3} + \lVert  b \rVert_\infty^{2/3} + \lVert c \rVert_\infty^{2/3})}R_1^{-K(1 + \lVert V \rVert_\infty^{2/3} + \lVert  b \rVert_\infty^{2/3} + \lVert c \rVert_\infty^{2/3})} .
\end{align*}
Using $R_1 \in (0,1/2)$ and $K_2 \leq R_1^{-K}$ (for some $K \geq 1$ independent on $R_1$) we find
\[
 \tilde D_2^{1/\gamma} =  \max\{1 , D_2^{1/\gamma}\} \leq R_1^{-K(1 + \lVert V \rVert_\infty^{2/3} + \lVert  b \rVert_\infty^{2/3} + \lVert c \rVert_\infty^{2/3})} .
\]
Using the definition of $D_1,D_2$ from the proof of Theorem~\ref{thm:three_annuli} and the estimates
\begin{align*}
F_{(R_1-r_1)/4} &\leq K R_1^{-2} (1+\lVert V \rVert_\infty^2 + \lVert b \rVert_\infty^2 + \lVert c \rVert_\infty),\quad \text{and} \\
F_{(R_3-r_3)/4} &\geq K(1 + \lVert V \rVert_\infty^2 + \lVert b \rVert_\infty^2 + \lVert c \rVert_\infty),
\end{align*}
we conclude
\begin{align*}
\frac{D_1}{D_2} \leq K R_1^{-4}\leq R_1^{-K}.
\end{align*}
From the definition of $D_2$ in the proof of Theorem~\ref{thm:three_annuli} we infer that $D_2 \geq K_2$. Hence, $1/D_2 \leq R_1^{-K}$. Hence we find
\[
 \frac{\tilde D_1}{\tilde D_2} = \frac{\max\{1 , D_1\}}{\max\{1, D_2\}}
 \leq \frac{D_1}{D_2} + \frac{1}{D_2} \leq R_1^{-K} .
\]

Moreover, using $\alpha^\ast\leq \euler^{K(R_3+1)}(1+ \lVert V \rVert_\infty^{2/3} + \lVert b \rVert_\infty^{2} + \lVert c \rVert_\infty^{2/3})$, we have
\begin{align*}
\left(\frac{r_3}{r_1}\right)^{2\alpha^\ast}
&\leq
 \left(\frac{K}{R_1}\right)^{K(1+ \lVert V \rVert_\infty^{2/3} + \lVert b \rVert_\infty^{2} + \lVert c \rVert_\infty^{2/3})}
\leq R_1^{-K(1+ \lVert V \rVert_\infty^{2/3} + \lVert b \rVert_\infty^{2} + \lVert c \rVert_\infty^{2/3})}.
\end{align*}
Since $2^{1/\gamma} \leq R_1^{-K}$, we conclude using $x^{2/3} \leq 1+x^2$ for $x \geq 0$
\begin{equation*}
C_1 (\gamma)^{1/\gamma}\leq 2^{1/\gamma} \frac{\tilde D_1}{\tilde D_2} \left( \tilde D_2^{1/\gamma} + \left(\frac{r_3}{r_1}\right)^{2\alpha^\ast} \right)
\leq R_1^{-K (1+ \lVert V \rVert_\infty^{2/3} + \lVert b \rVert_\infty^{2} + \lVert c \rVert_\infty^{2/3})}.
\end{equation*}
It remains to show the upper bound on $\D_3 / \tilde D_1$. By definition we have
\[
\frac{D_3}{\tilde D_1} \leq \frac{D_3}{D_1} = \frac{8 (1 + \Theta_3^2 / \Theta_1^2) \Elliptic + 2 \Theta_1^{-2}}{3 \Elliptic^2 + 12 \Elliptic^2 d^2 / r_1'^2 + 3(\Lipschitz d^2 + \lVert b \rVert_\infty )^2   + 4{\Elliptic}  \Cac_{(R_1 - r_1)/4}} .
\]
Since $\Theta_1 \geq K$, $\Theta_3 = K$, and $F_{(R_1 - r_1)/4} \geq K R_1^{-2}$ we find $D_3 / \tilde D_1 \leq K R_1^2$.
\end{proof}
%
%
%
%

\begin{theorem}[Chaining and covering] \label{thm:chaining}
Let $R \in \NN_{\infty}$ and $\Omega_-\subset \Omega_+$ be open subsets of $\Lambda_R$.
Let $\epsilon > 0$ and $0 < r_1 < R_1 \leq r_2 < R_2 \leq r_3 < R_3$ such that $R_1 \leq (R_2 - r_2) / 4$, and  \eqref{ass:radii} is satisfied.

Then for all measurable and bounded $V \colon \Lambda_R \to \RR$,
all $\psi \in \Dom (H_R)$ and $\zeta \in L^2 (\Lambda_R)$ satisfying $\lvert H_R \psi \rvert \leq \lvert V\psi \rvert + \lvert \zeta \rvert$ almost everywhere on $\Omega_+$,
all $(1,\delta)$-equidistributed sequences $Z=(z_j)_{j\in \ZZ^d}$,
all $ \cJ   \subset \ZZ^d$,
satisfying
\[
\Omega_- \subset \bigcup_{j\in \cJ} \Lambda_1(j) \text{  and }
\forall j \in  \cJ   : \Lambda_{1+2a + 2R_3}(j) \subset  \Omega_+, \text{ where } a= (R_2+3r_2)/4,
\]
and all $\D_3 \geq D_3$ we have
\begin{multline}\label{eq:C2-LambdaR}
 C_2 (\gamma) \left( \sum_{j \in \cJ} \lVert\psi\rVert_{Z_2 (z_j)}^2 +  2 \D_3 M \lVert \zeta \rVert_{\Omega_+}^2\right)
\\
\geq
\left( \lVert \psi \rVert_{\Omega_+}^2 +  {\D_3 } \lVert \zeta \rVert_{\Omega_+}^2 \right)^{1-\gamma^{-m+1}}
\biggl( \lVert \psi \rVert_{\Omega_-}^2 + \D_3 \lVert \zeta \rVert_{\Omega_+}^2 \biggr)^{\gamma^{-m+1}}.
\end{multline}
Here
\[
C_2 (\gamma)
=
C_1(\gamma)^{\gamma^{-1}+\ldots+\gamma^{-m+1}} M^{\gamma^{1-m} - 1} N^{\gamma^{1-m}} > 0,
\]
$N = \lceil 4 \sqrt{d}/ (R_2-r_2)\rceil^d$,  $m =2 \lfloor 2\sqrt{d}/(R_2 - r_2) \rfloor + 2$ and
$C_1(\gamma)$, $\gamma \in (0,1)$, and  $M$ are as in Theorem~\ref{thm:interpolation}.

If $\cJ=\ZZ^d$ and $\Omega_-=\Omega_+=\RR^d$ the bound above simplifies to
\begin{equation}\label{eq:C2-Rd}
 C_2 (\gamma) \left( \sum_{j\in\ZZ^d } \lVert \psi \rVert_{Z_2 (z_j)}^2
 +
  2 \D_3 M  \lVert \zeta \rVert_{\RR^d}^2 \right)
 \geq
  \lVert \psi \rVert_{\RR^d}^2 + \D_3 \lVert \zeta \rVert_{\RR^d}^2 .
\end{equation}
\end{theorem}

\begin{proof}
Let $\rho = (R_2 - r_2) / 4$ and define the sequence $(y_j)_{j \in \cJ}$ such that $y_j \in \overline{\Lambda_1 (j)}$ and
\[
  \lVert \psi \rVert_{\ball{\rho}{y_j}} = \sup_{x \in \Lambda_1 (j)} \lVert \psi \rVert_{\ball{\rho}{x}}.
\]
For each $j \in \cJ$ we choose a sequence of points $\tau_j = (z_j^{i})_{i=0}^m \subset \Lambda_{1+2a}(j)$,
with $z_j^{0} = z_j$, $z_j^{m} = y_j$, and $\lvert z_j^{i} - z_j^{i-1} \rvert \in [(R_2 + 3 r_2)/4 , (3R_2 + r_2) / 4]$ for $i \in \{1,\ldots , m\}$.
The proof of the existence of such sequences is postponed to Lemma~\ref{lemma:paths} below.
By construction of the sequences $\tau_j$ we have $\ball{\rho}{z_j^{i+1}} \subset Z_2 (z_j^{i})$.
Since $\tilde D_2 \geq 1$ we have $1/\tilde D_2 + 1 \leq 2$,
hence Theorem~\ref{thm:interpolation} with $t = 2$ and $\Omega_*=\Omega_+$ implies for all $\D_3 \geq D_3$
\begin{multline}\label{eq:3annuli}
\sum_{j \in \cJ} \lVert \psi \rVert_{Z_1 (x_j)}^2 + \frac{2 \D_3 M}{\tilde D_1} \lVert \zeta \rVert_{\Omega_+}^2 \\
\geq
C_1(\gamma)^{-1/\gamma} \biggl( M \lVert \psi \rVert_{\Omega_+}^2 + {\D_3 M}  \lVert \zeta \rVert^2_{\Omega_+} \biggr)^{1-1/\gamma}
\biggl( \sum_{j \in \cJ} \lVert \psi \rVert_{Z_2 (x_j)}^2 + 2 \D_3 M \lVert \zeta \rVert_{\Omega_+}^2 \biggr)^{1/\gamma} .
\end{multline}
for all sequences $(x_j)_{j \in \cJ}$ satisfying the assumption of Theorem~\ref{thm:interpolation}.
Since
\[
\bigcup_{z\in \Lambda_{1+2a}} \ball{R_3}{z} \subset \Lambda_{1+2a+2R_3}
\]
this holds true for the sequence $(z_j^i)_{j \in \cJ}$ for any $i \in \{0,1,\ldots , m-1\}$.
For $i \in \{0,1,\ldots , m-1\}$ we introduce the notation
\[
 A = M \lVert \psi \rVert_{\Omega_+}^2 +  {\D_3 M} \lVert \zeta \rVert_{\Omega_+}^2
 \quad\text{and}\quad
 B(i)  = \sum_{j \in \cJ} \lVert \psi \rVert_{Z_2 (z_j^i)}^2 +  2 \D_3 M \lVert \zeta \rVert_{\Omega_+}^2 .
\]
Since $\rho \geq R_1$ we find
\[
 B (0)
\geq
 \sum_{j \in \cJ} \lVert \psi \rVert_{Z_1 (z_j^1)}^2 + 2 \D_3 M \lVert \zeta \rVert_{\Omega_+}^2  .
\]
We apply Ineq.~\eqref{eq:3annuli} and obtain using $\tilde D_1 \geq 1$
\begin{align*}
  B (0)  &\geq C_1 (\gamma)^{-1/\gamma}  A^{1-1/\gamma}
B^{1/\gamma} (1) .
\end{align*}
After $m-1$ steps of this type we obtain
\begin{equation*}
 B(0) \geq C_1 (\gamma)^{-(\gamma^{-1} + \ldots + \gamma^{-m+1})}  A^{1-\gamma^{-m+1}}
B^{\gamma^{-m+1}} (m-1) .
\end{equation*}
Since $\ball{\rho}{z_j^{m}} \subset Z_2 (z_j^{m-1})$ we obtain
\begin{equation*}
 B(0)\geq C_1 (\gamma)^{-(\gamma^{-1} + \ldots + \gamma^{-m+1})}  A^{1-\gamma^{-m+1}}
\biggl( \sum_{j \in \cJ} \lVert \psi \rVert_{\ball{\rho}{z_j^m}}^2 + 2 \D_3 M \lVert \zeta \rVert_{\Omega_+}^2 \biggr)^{\gamma^{-m+1}} .
\end{equation*}
Since $\Lambda_{\rho / \sqrt{d}} \subset B_\rho$, for each $j \in \cJ$ we cover $\Lambda_1 (j)$ with $N = N (d,\rho) \leq \lceil \sqrt{d} / \rho \rceil^d$ balls of radius $\rho$. We denote the centers of these balls by $x_{jk}$, $k\in\{1,\ldots,N\}$.
Thus, for any $\Omega_- \subset \bigcup_{j  \in \cJ} \Lambda_1(j)$
\begin{equation*} 
\lVert \psi \rVert_{\Omega_-}^2
\leq
\sum_{j \in \cJ} \lVert \psi \rVert_{\Lambda_1 (j)}^2
\leq
\sum_{j \in \cJ} \sum_{k=1}^N \lVert \psi \rVert_{\ball{\rho}{ x_{jk}}}^2
\leq
\sum_{j \in \cJ} \sum_{k=1}^N \lVert \psi \rVert_{\ball{\rho}{y_j}}^2
=
N \sum_{j \in \cJ}  \lVert \psi \rVert_{\ball{\rho}{y_j}}^2 .
\end{equation*}
This implies together with $M,N \geq 1$
\begin{align}\label{eq:Omega+neqOmega-}
 B(0) \nonumber
 &\geq C_1 (\gamma)^{-(\gamma^{-1} + \ldots + \gamma^{-m+1})}  A^{1-\gamma^{-m+1}}
\biggl( \frac{1}{N} \lVert \psi \rVert_{\Omega_-}^2 + 2 \D_3 M \lVert \zeta \rVert_{\Omega_+}^2 \biggr)^{\gamma^{-m+1}}
\\
&\geq C_1 (\gamma)^{-(\gamma^{-1} + \ldots + \gamma^{-m+1})}  A^{1-\gamma^{-m+1}}
\biggl( \frac{1}{N}\biggr)^{\gamma^{-m+1}} \biggl( \lVert \psi \rVert_{\Omega_-}^2 +  \D_3  \lVert \zeta \rVert_{\Omega_+}^2 \biggr)^{\gamma^{-m+1}}
\end{align}
If $\cJ=\ZZ^d$ and $\Omega_-=\Omega_+=\RR^d$ we insert the definition of $A$ to obtain
\begin{align*}
 B(0) &\geq C_1 (\gamma)^{-(\gamma^{-1} + \ldots + \gamma^{-m+1})} M^{1-\gamma^{-m+1}} N^{-\gamma^{-m+1}} \left( \lVert \psi \rVert_{\RR^d}^2 + \D_3 \lVert \zeta \rVert_{\RR^d}^2 \right) .  \qedhere
\end{align*}
\end{proof}

\begin{lemma}[Existence of chain connection]\label{lemma:paths}
  Let $0<a<b<\infty$, $y,z\in\Lambda_1$, and
  $m = 2\lfloor \sqrt{d} / (b-a) \rfloor + 2$.
  Then there is a sequence $\tau = (z^{i})_{i=0}^m \subset \Lambda_{1+2a}$
  with $z^{0} = z$, $z^{m} = y$, and $\lvert z^{i} - z^{i-1} \rvert \in [a,b]$ for $i \in \{1,\ldots , m\}$.
\end{lemma}
In Theorem \ref{thm:chaining} we apply the Lemma with the choice $b= (3R_2+r_2)/4, a= (R_2+3r_2)/4$,
hence $\Lambda_{1+2a}=\Lambda_{1+(R_2+3r_2)/2}$.
\begin{proof}
  Starting from $z^k$, we observe that in two steps we can reach any point
  $z^{k+2}$
  inside the closed ball with radius $b-a$ and center $z^k$.
  This can be achieved by choosing $z^{k+1}$ such that $\lvert z^{k+1} - z^k \rvert =a$ and $\lvert z^{k+2} - z^{k+1} \rvert \in [a,b]$, where in the first step we move away from $z^{k+2}$,
  and in the second step we move back arriving at $z^{k+2}$, see Fig.~\ref{fig:1}.
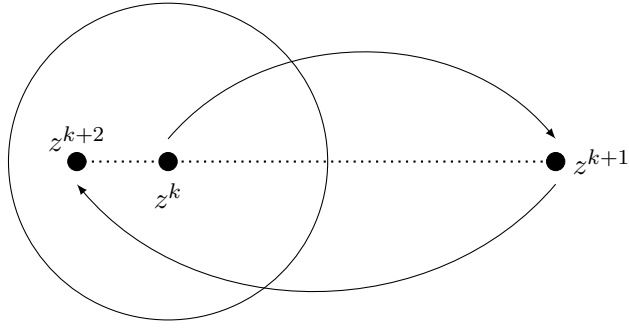
\begin{figure}[ht]\centering
  \begin{tikzpicture}
  \begin{scope}[xscale = 3, yscale = 3]
  \filldraw[fill=black] (2.3,3) circle (0.04cm);
  \draw (2.3,2.95) node[below] {$z^k$};
  \filldraw[fill=black] (4,3) circle (0.04cm);
  \draw (4.03,3) node[right] {$z^{k+1}$};
  \filldraw[fill=black] (1.9,3) circle (0.04cm);
  \draw (1.9,3) node[above] {$z^{k+2}$};
  \draw[dotted,thick] (1.9,3)--(4,3);
  \draw (2.3,3) circle (0.7cm);
  \draw[-latex] (2.3,3.1) to [out=50,in=130] (4,3.1);
  \draw[-latex] (4,2.9) to [out=-130, in=-50] (1.9,2.9);
  \end{scope}
  \end{tikzpicture}
  \caption{Within two steps we can reach any point in $\overline{\ball{b-a}{z^k}}$\label{fig:1}}
\end{figure}
  Now let $\mu := 2 \lfloor \lvert y-z \rvert / (b-a) \rfloor$
  and note that $\mu = 0$ if $y \in \ball{b-a}{z}$, and $\mu \leq 2 \lfloor \sqrt{d} / (b-a) \rfloor = m-2$.
Moving along the line connecting $z$ and $y$, in the first $\mu$ steps we approach $y$ in the sense that for all $k \in 2 \NN_0$ with $0 \leq k \leq \mu - 2$ we have
\[
 \lvert z^{k} - z^{k+2} \rvert = b-a \quad \text{and} \quad
 \lvert y - z^{k+2} \rvert = \lvert y - z^{k} \rvert - (b-a) .
\]
This can be achieved by choosing $z^{k+1}$ and $z^{k+2}$ such that $\lvert z^{k+1} - z^k \rvert = a$ and $\lvert z^{k+2} - z^{k+1} \rvert = b$ where the first step is done moving away from $y$ and the second one is done moving closer towards $z^k$, see again Fig.~\ref{fig:1}.
Then we repeat this double step exactly $\mu / 2$ times, see Fig.~\ref{fig:2}.
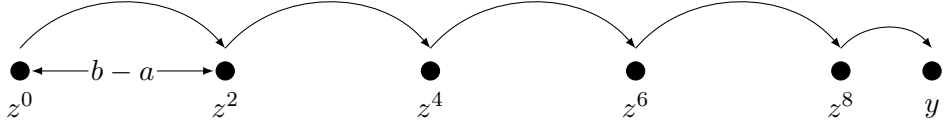
\begin{figure}[ht]\centering
  \begin{tikzpicture}[xscale = 3, yscale = 3]
  \foreach \x/\y in {0/0,0.9/2,1.8/4,2.7/6}{
	\filldraw[fill=black] (\x,0) circle (0.04cm);
	\draw (\x,-0.05) node[below] {$z^{\y}$};
	\draw[-latex] (\x,0.1) to [out=50,in=130] (\x+0.9,0.1);
  }
  \filldraw[fill=black] (3.6,0) circle (0.04cm);
  \draw (3.6,-0.05) node[below] {$z^{8}$};
  \filldraw[fill=black] (4,0) circle (0.04cm);
  \draw (4,-0.08) node[below] {$y$};
  \draw[-latex] (3.6,0.1) to [out=50,in=130] (4,0.1);
  \draw (0.45,0) node {$b-a$};
  \draw[-latex] (0.6,0)--(0.85,0);
  \draw[-latex] (0.3,0)--(0.05,0);
  \end{tikzpicture}
\caption{Illustration of a sequence $\tau$ with $\mu = 8$\label{fig:2}}
\end{figure}
By construction we now have $y \in \ball{b-a}{z^{\mu}}$, see Fig.~\ref{fig:2}.
Hence, after $\mu + 2$ steps we reach $y$, see Fig.~\ref{fig:1}.
The remaining $m - \mu - 2$ steps we just go back and forth such that $z^k = y$ for $k \in 2 \NN$ with $\mu + 2 \leq k \leq m$.
By construction we have $\tau \subset \Lambda_{1+2a}$.
\end{proof}
Now we are in position to prove our two main theorems. The first one concerns functions on the whole of $\RR^d$.
\begin{proof}[Proof of Theorem~\ref{thm:sampling}]
 We choose $\epsilon = 1$,
\begin{equation}
\begin{aligned} \label{eq:radii-delta}
R_3 &= (33 \euler d \Elliptic^{11/2} (\Lipschitz + 1))^{-1} & \quad
R_2&=\frac{R_3}{2 \euler (\Elliptic+1)^{5/2}} &\quad
R_1& = R_1^\delta = \delta
\\
r_3 &= \frac{R_3}{\Elliptic+1} &\quad
r_2&= R_2 / 5 &\quad
r_1&= r_1^\delta =  \delta  / 2 .
 \end{aligned}
\end{equation}
We have now introduced a superscript $\delta$ in $r_1^\delta$ and $R_1^\delta$ making the dependence on this parameter explicit. (Note that the other radii do not depend on the parameter $\delta$.) Consequently, $\tilde D_1 = \tilde D_1^\delta$, $D_3 = D_3^\delta$ and $C_1 (\gamma) = C_1^\delta (\gamma)$ since they all depend on $R_1 = R_1^\delta$.
Since $\delta \in (0,\delta_0]$ we have $R_1^\delta= \delta \leq \delta_0= r_2$ and Lemma~\ref{lemma:interpolation} applies.
Hence, Assumption~\eqref{ass:radii} is satisfied. Now we apply Theorem~\ref{thm:interpolation} with $\cJ =\ZZ^d$, $\Omega_* = \Lambda_R=\RR^d$,
\begin{equation}\label{eq:gamma1}
 \gamma_1 = \gamma_1^\delta = \frac{\ln (r_3 / (R_2 \mu_1 \Elliptic))}{\ln (r_3 / r_1^\delta)} = \frac{\ln (r_3 / (R_2 \mu_1 \Elliptic))}{\ln (2r_3 / \delta )} \in (0,1)
\end{equation}
and $t = 2$, and obtain using again $1/ \tilde D_2 + 1 \leq 2$
\begin{align}\label{eq:apply-chain}
\lVert \psi & \rVert_{\equi_{\delta,Z}}^2 + \frac{2 D_3^\delta M}{\tilde D_1^\delta} \lVert \zeta \rVert_{\RR^d}^2\nonumber
\geq
\sum_{j \in \ZZ^d} \lVert \psi \rVert_{Z_1 (z_j)}^2 +  \frac{2 D_3^\delta M}{\tilde D_1^\delta} \lVert \zeta \rVert_{\RR^d}^2 \nonumber \\
&\geq
C_1^\delta (\gamma_1^\delta)^{-1/\gamma_1^\delta} (M \lVert \psi \rVert_{\RR^d}^2 +D_3^\delta M \lVert \zeta \rVert_{\RR^d}^2)^{1-1/\gamma_1^\delta}
\biggl( \sum_{j \in \ZZ^d} \lVert \psi \rVert_{Z_2 (z_j)}^2 + 2 D_3^\delta M \lVert \zeta \rVert_{\RR^d}^2 \biggr)^{1/\gamma_1^\delta} .
\end{align}
Next we apply Theorem~\ref{thm:chaining}, but with all radii independent of $\delta$, namely:
\begin{equation}
\begin{aligned} \label{eq:radii-2}
R_3 &= (33 \euler d \Elliptic^{11/2} (\Lipschitz + 1))^{-1} & \quad
R_2&=\frac{R_3}{2 \euler (\Elliptic+1)^{5/2}} &\quad
R_1 & =  r_2
\\
r_3 &= \frac{R_3}{\Elliptic+1} &\quad
r_2&= R_2 / 5 &\quad
r_1&=  r_2 / 2 .
 \end{aligned}
\end{equation}
Then  Assumption~\eqref{ass:radii} is still satisfied and additionally $R_1 = (R_2 - r_2) / 4$ holds.
Moreover, calculating the derivative shows that the map $\delta \mapsto D_3^\delta$ in strictly decreasing on the interval $(0,\delta_*]$ with
\[
\delta_*:= \left[\frac{128 \, \Elliptic \, \tilde\Theta}{2+8\, \Elliptic \, \Theta_3}\right]^{1/4}.
\]
Due to $R_3= 10 \, \euler \, (\Elliptic+1)^{5/2} \cdot \delta_0$ we see that  $\delta_0< \delta_*$.
Having in mind $\delta_0 = r_2$ this gives
\[
D_3^\delta \geq D_3  := D_3^{r_2} \text{ for all } \delta \in(0,\delta_0).
\]
The last inequality and Theorem~\ref{thm:chaining} with $\D_3 = D_3^\delta$ imply
\begin{equation*}
\sum_{j\in\ZZ^d} \lVert \psi \rVert_{Z_2 (z_j)}^2
 +
  2 D_3^\delta M \lVert \zeta \rVert_{\RR^d}^2
 \geq
  C_2 (\gamma_2)^{-1} \left(\lVert \psi \rVert_{\RR^d}^2 + D_3^\delta \lVert \zeta \rVert_{\RR^d}^2 \right),
\end{equation*}
where
\begin{equation} \label{eq:gamma2}
C_2 (\gamma_2)
=
C_1(\gamma_2)^{\gamma_2^{-1}+\ldots+\gamma_2^{-m+1}} M^{\gamma_2^{1-m} - 1} N^{\gamma_2^{1-m}},
 \quad
 \gamma_2 = \frac{\ln (r_3 / (R_2 \mu_1 \Elliptic))}{\ln (2r_3 / r_2\})} \in (0,1).
\end{equation}
Inserting this into \eqref{eq:apply-chain} we obtain
\begin{align*}
 \lVert \psi \rVert_{\equi_{\delta,Z}}^2 &+ \frac{2 D_3^\delta M}{\tilde D_1^\delta}  \lVert \zeta \rVert_{\RR^d}^2 \\
 &\geq C_1^\delta(\gamma_1^\delta)^{-1/\gamma_1^\delta} C_2 (\gamma_2)^{-1/\gamma_1^\delta} M^{1-1/\gamma_1^\delta}
  (\lVert \psi \rVert_{\RR^d}^2 + D_3^\delta \lVert \zeta \rVert_{\RR^d}^2)^{1-1/\gamma_1^\delta}
   \left(\lVert \psi \rVert_{\RR^d}^2 + D_3^\delta \lVert \zeta \rVert_{\RR^d}^2 \right)^{1/\gamma_1^\delta} \\
 &\geq C_1^\delta(\gamma_1^\delta)^{-1/\gamma_1^\delta} C_2 (\gamma_2)^{-1/\gamma_1^\delta} M^{1-1/\gamma_1^\delta}
  (\lVert \psi \rVert_{\RR^d}^2 + D_3^\delta \lVert \zeta \rVert_{\RR^d}^2)
   .
\end{align*}
In order to estimate the constant $C_1^\delta (\gamma_1^\delta)^{-1/\gamma_1^\delta} C_2 (\gamma_2)^{-1/\gamma_1^\delta} M^{1-1/\gamma_1^\delta} $ from below,
we will estimate its inverse from above and denote by $K_i \geq 1$, $i \in \NN_0$, constants depending only on $d$, $\Elliptic$ and $\Lipschitz$.
Note that $R_{\mathrm{max}} = \sqrt{d} + (R_2 + 3r_2)/4$, $M = (2R_3 + 2a+ 1)^d$, $m = 2\lfloor 2\sqrt{d} (R_2 - r_2) \rfloor + 2$, $N = [4/(R_2 - r_2)]^d$,
\[
 \frac{1}{\gamma_2} = \frac {\ln (2r_3 / r_2\})}{\ln (r_3 / (R_2 \mu_1 \Elliptic))},
 \quad
 \mu_1 =
        \begin{cases}
          \exp( \sqrt{\Elliptic} \mu )& \text{if $\sqrt{\Elliptic} \mu \leq 1$,}\\
            \euler \sqrt{\Elliptic} \mu & \text{if $\sqrt{\Elliptic} \mu > 1$},
         \end{cases}
\]
$\mu = 33 d R_3 \Elliptic^{11/2} \Lipschitz + 1$, and $1 / r_2
=330  \euler^2 d (\Elliptic+1)^{5/2} \Elliptic^{11/2} (\Lipschitz + 1)$
are greater than or equal to one and functions of $d$, $\Elliptic$ and $\Lipschitz$ only.
By Lemma~\ref{lemma:interpolation} we have
\[
 C_1^\delta (\gamma_1^\delta)^{1/{\gamma_1^\delta}} \leq {\delta}^{-K_1 (1+\lVert V \rVert_\infty^{2/3} + \lVert b \rVert_\infty^2 + \lVert c \rVert_\infty^{2/3})} .
\]
We write $M^{1/\gamma_1^\delta-1} =  M^{-1} M^{1/\gamma_1^\delta}$.
Since $\delta\leq \delta_0 < \frac{1}{2}$, we have $M^{-1} \leq 1<2 <\delta^{-1}$.
Furthermore, we have
\begin{equation}\label{eq:nur-oben}
    M^{1/\gamma_1^\delta}
 =  M^{\frac{\ln (2r_3 / \delta )}{\ln (r_3 / (R_2 \mu_1 \Elliptic))}}
 = \left( \frac{2 r_3}{\delta} \right)^{\frac{M}{\ln (r_3 / (R_2 \mu_1 \Elliptic))}}.
\end{equation}
The last term (observing $\delta\leq \delta_0=r_2 \leq 2 r_3$) is bounded by
\begin{equation}\label{eq:nur-oben2}
 \left( \frac{2 r_3}{\delta} \right)^{K_5} \leq  \left( \frac{1}{\delta} \right)^{K_5}
\end{equation}
with $K_5\geq 1$, since $2 r_3 \leq 1$. Collecting terms we obtain
$M^{1/\gamma_1^\delta-1}\leq \delta^{-1}\,  \delta^{-K_5} =:\delta^{-K_6}$.
We apply once more Lemma~\ref{lemma:interpolation},
this time with $R_1 = r_2$
\[
C_1 (\gamma_2)
\leq r_2^{-\gamma_2 \cdot K_1(1 + \lVert V \rVert_\infty^{2/3} + \lVert b \rVert_\infty^{2} + \lVert c \rVert_\infty^{2/3})\cdot K_0}
\leq K_8^{(1 + \lVert V \rVert_\infty^{2/3} + \lVert b \rVert_\infty^{2} + \lVert c \rVert_\infty^{2/3})}
\]
with $K_8 \geq 1$, since $r_2<1$ and $\gamma_2 \cdot K_1 >0$.
Since $M,N,m,1/\gamma_2$ depend only on $d$, $\Elliptic$ and $\Lipschitz$ and are at least one,
we see that
\begin{align*}
C_2 (\gamma_2)^{1/\gamma_1^\delta}
 & =\left(C_1 (\gamma_2)^{(\gamma_2^{-1} + \ldots + \gamma_2^{-m+1})} M^{\gamma_2^{1-m} - 1} N^{\gamma_2^{1-m}}\right)^{1/\gamma_1^\delta} \\
 & \leq
 \left( K_8^{(1 + \lVert V \rVert_\infty^{2/3} + \lVert b \rVert_\infty^{2} + \lVert c \rVert_\infty^{2/3}) \cdot K_9} K_{10}\right)^{1/\gamma_1^\delta}
=
\left(K_{11}^{1/\gamma_1^\delta} \right)^{(1 + \lVert V \rVert_\infty^{2/3} + \lVert b \rVert_\infty^{2} + \lVert c \rVert_\infty^{2/3})} .
\end{align*}
Arguing as in \eqref{eq:nur-oben} and \eqref{eq:nur-oben2}, now with $M$ replaced by $K_{11}$, we conclude
\begin{align*}
 C_2 (\gamma_2)^{1/\gamma_1^\delta} &
\leq
\left({\delta}^{-K_{12}} \right)^{(1 + \lVert V \rVert_\infty^{2/3} + \lVert b \rVert_\infty^{2} + \lVert c \rVert_\infty^{2/3})} .
\end{align*}
Putting everything together we obtain
\[
C_1 (\gamma_1)^{-1/\gamma_1} C_2 (\gamma_2)^{-1/\gamma_1} M^{1-1/\gamma_1}
\geq {\delta}^{K_{13}(1+\lVert V \rVert_\infty^{2/3} + \lVert b \rVert_\infty^{2} + \lVert c \rVert_\infty^{2/3})}.
\]
By Lemma~\ref{lemma:interpolation} we have $2 D_3^\delta M / \tilde D_1^\delta \leq K_{14} \delta^2$. This shows the statement of the theorem.
\end{proof}
\bigskip

After completing the proof for the $\RR^d$-case, we turn to functions defined on finite boxes $\Lambda_L$.

\begin{proof}[Proof of Theorem~\ref{thm:equidistribution}]
Since $L\in \NN$ and $R_3 < 1/89  \ll  1$ we have  $ L + R_3 \leq  3L/2 $, hence
$\ball{R_3}{x} \subset \Lambda_{3L}$ for all $x \in \Lambda_L$.
\medskip

As explained in Appendix~\ref{sec:extensions}, we extend the functions $\psi_L$ and $\zeta_L$, the coefficients $A_L$, $b_L$, and $c_L$, as well as the potential $V_L$ to $\Lambda_{9L}$, such that the properties given below in Lemmata \ref{l:extension} 
and \ref{lemma:extensions} are satisfied. 
We denote these extensions by $\hat\psi_{L}$, $\hat \zeta_{L}$, $\hat A_{L}$, $\hat b_{L}$, $\hat c_{L}$, and $\hat V_{L}$. Moreover, we denote by $\hat a_{L} : \HNullEins (\Lambda_{9L}) \times \HNullEins (\Lambda_{9L}) \to \CC$ the densely defined, closed, and sectorial form
\[
 \hat a_L (u,v)  = \int_{\Lambda_{9L}} \left( \sum_{i,j=1}^d \hat a_L^{ij} \partial_i u \overline{\partial_j v} + \sum_{i=1}^d \hat b_L^i \partial_i u \overline{v} + \hat c_L u \overline{v} \right) \drm x ,
\]
and by $\hat H_L$ the m-sectorial operator associated with the form $\hat a_L$. We denote the domain of $\hat H_L$ by $\Dom (\hat H_L)$.
Note that $\hat H_L$ is the Friedrichs extension of the differential operator $\hat \Op_L : C_{\mathrm{c}}^\infty (\Lambda_{9L}) \to L^2 (\Lambda_{9L})$,
 \[
\hat \Op_L u := - \diver (\hat A_L \nabla u) + \hat b_L^\T \nabla u + \hat c_L u .
\]
Then we have $\hat \psi_L \in \Dom (\hat H_L)$, $\lvert \hat H_L \hat \psi_L \rvert \leq \lvert \hat V_L \hat \psi_L \rvert + \lvert \hat \zeta_L \rvert$ almost everywhere on $\Lambda_{9L}$, $\hat A_L$ satisfies the ellipticity and Lipschitz condition \eqref{eq:elliptic2} on $\Lambda_{9L}$, $\lVert \hat b_L \rVert_\infty = \lVert b_L \rVert_\infty$, and $\lVert \hat c_L \rVert_\infty = \lVert c_L \rVert_\infty$, see Lemma~\ref{lemma:extensions}.
Note that $\hat \psi_L \mid_{\Lambda_{3L}}\colon \Lambda_{3L} \to \CC$ has all these nice properties as well.
\medskip

As in the proof of Theorem~\ref{thm:sampling} we want to apply Theorem~\ref{thm:interpolation}  
and make the choice $\epsilon = 1$,
\begin{equation}
\begin{aligned} \label{eq:radii-delta-L}
R_3 &= (33 \euler d \Elliptic^{11/2} (\Lipschitz + 1))^{-1} & \quad
R_2&=\frac{R_3}{2 \euler (\Elliptic+1)^{5/2}} &\quad
R_1& = R_1^\delta = \delta
\\
r_3 &= \frac{R_3}{\Elliptic+1} &\quad
r_2&= R_2 / 5 &\quad
r_1&= r_1^\delta =  \delta  / 2 .
 \end{aligned}
\end{equation}
\begin{equation}\label{eq:gamma1-L}
 \gamma_1 = \gamma_1^\delta = \frac{\ln (r_3 / (R_2 \mu_1 \Elliptic))}{\ln (r_3 / r_1^\delta)} = \frac{\ln (r_3 / (R_2 \mu_1 \Elliptic))}{\ln (2r_3 / \delta )} \in (0,1)
\end{equation}
and additionally $t = 2$, $\Omega_*= \Lambda_{3L}$,
$\Lambda_R=\Lambda_{9L}$,
$\cJ= \ZZ^d\cap\Lambda_L$.
Since $D_3^\delta \geq D_3$ and $1/ \tilde D_2 + 1 \leq 2$ Theorem~\ref{thm:interpolation}  gives
with the same $C_1(\gamma)$, $M$, $\tilde D_2$, and $\tilde D_1$ as there
\begin{align}\label{eq:apply-chain-L}
\lVert \hat\psi_L  & \rVert_{\equi_{\delta,Z}(L)}^2 + \frac{2 D_3^\delta M}{\tilde D_1^\delta} \lVert \hat\zeta_L \rVert_{\Lambda_{3L}}^2
\geq
\sum_{j \in \ZZ^d \cap\Lambda_L} \lVert \hat\psi_L  \rVert_{Z_1 (z_j)}^2 +  \frac{2 D_3^\delta M}{\tilde D_1^\delta} \lVert \hat \zeta_L \rVert_{\Lambda_{3L}}^2 \\
&\geq  \nonumber
C_1^\delta (\gamma_1^\delta)^{-1/\gamma_1^\delta} (M \lVert \hat\psi_L  \rVert_{\Lambda_{3L}}^2 +D_3^\delta M \lVert \hat \zeta_L \rVert_{\Lambda_{3L}}^2)^{1-1/\gamma_1^\delta}
\biggl( \sum_{j \in \ZZ^d\cap\Lambda_L} \lVert \hat\psi_L  \rVert_{Z_2 (z_j)}^2 + 2 D_3^\delta M \lVert \hat \zeta_L \rVert_{\Lambda_{3L}}^2 \biggr)^{1/\gamma_1^\delta} .
\end{align}
(At this stage it becomes apparent why we need the extensions to a larger cube:
the annuli $Z_i(x_j)$, $i\in\{2,3\}$, around $x_j$ for $j\in \ZZ^d\cap \Lambda_L$ might extend beyond the cube $\Lambda_L$ depending on the choice of the radii and their centers $x_j$.)

Next we apply Theorem~\ref{thm:chaining}, with  $\cJ= \ZZ^d \cap \Lambda_L$,
$\Omega_-=\Lambda_{L},\Omega_+=\Lambda_{9L}$, $\D_3 = D_3^\delta$,
\begin{equation}
\begin{aligned} \label{eq:radii-2-L}
R_3 &= (33 \euler d \Elliptic^{11/2} (\Lipschitz + 1))^{-1} & \quad
R_2&=\frac{R_3}{2 \euler (\Elliptic+1)^{5/2}} &\quad
R_1 & =  r_2
\\
r_3 &= \frac{R_3}{\Elliptic+1} &\quad
r_2&= R_2 / 5 &\quad
r_1&=  r_2 / 2 .
 \end{aligned}
\end{equation}
and $
 \gamma_2 = \frac{\ln (r_3 / (R_2 \mu_1 \Elliptic))}{\ln (2r_3 / r_2\})} \in (0,1).$
 This yields
\begin{multline}
\label{eq:C2-LambdaR-L}
 C_2 (\gamma_2) \left( \sum_{j \in \ZZ^d\cap\Lambda_L} \lVert\hat \psi_L\rVert_{Z_2 (z_j)}^2
 + 2 D_3^\delta M \lVert \hat \zeta_L \rVert_{\Lambda_{9L}}^2\right)
\\
\geq
\left( \lVert \hat \psi_L \rVert_{\Lambda_{9L}}^2 +  {D_3^\delta } \lVert \hat \zeta_L \rVert_{\Lambda_{9L}}^2 \right)^{1-\gamma_2^{-m+1}}
\biggl( \lVert \hat \psi_L \rVert_{\Lambda_{L}}^2 + D_3^\delta \lVert \hat \zeta_L \rVert_{\Lambda_{9L}}^2 \biggr)^{\gamma_2^{-m+1}}
\end{multline}
where
\begin{equation}
\label{eq:gamma2-L}
C_2 (\gamma_2)
=
C_1(\gamma_2)^{\gamma_2^{-1}+\ldots+\gamma_2^{-m+1}} M^{\gamma_2^{1-m} - 1} N^{\gamma_2^{1-m}},
 \quad
 \gamma_2 = \frac{\ln (r_3 / (R_2 \mu_1 \Elliptic))}{\ln (2r_3 / r_2\})} \in (0,1).
\end{equation}
Due to the  reflection extension of
$\psi, \zeta\colon \Lambda_{L}\to \CC$  we have
\[
\|\psi\|_{\Lambda_{L}}
=
\frac{1}{9^d}\|\hat \psi_L\|_{\Lambda_{9L}}
\quad \text{ and } \quad
\lVert \zeta_L \rVert_{\Lambda_{L}}
=
\frac{1}{3^d}\|\hat \zeta_L\|_{\Lambda_{3L}} \quad \text{ etc. }
\]
and consequently the right hand side of \eqref{eq:C2-LambdaR-L} can be estimated in the following way 
\begin{align*}
\left( \lVert \hat \psi_L \rVert_{\Lambda_{9L}}^2 +  {D_3^\delta } \lVert \hat\zeta_L \rVert_{\Lambda_{9L}}^2 \right)^{1-\gamma_2^{-m+1}}
\biggl( \frac{1}{9^d} \lVert \hat \psi_L \rVert_{\Lambda_{9L}}^2 + D_3^\delta \lVert \hat\zeta_L \rVert_{\Lambda_{9L}}^2 \biggr)^{\gamma_2^{-m+1}}
\\
\geq
9^{-d\gamma_2^{-m+1}}
\left( \lVert \hat \psi_L \rVert_{\Lambda_{9L}}^2 +  {D_3^\delta } \lVert \hat\zeta_L\rVert_{\Lambda_{9L}}^2 \right)
\end{align*}
Inserting this into \eqref{eq:apply-chain-L} gives
\begin{multline*}
\lVert \psi_L   \rVert_{\equi_{\delta,Z}(L)}^2 +
{3^d}
\frac{2 D_3^\delta M}{\tilde D_1^\delta} \lVert \hat\zeta_L \rVert_{\Lambda_{L}}^2\nonumber
=
\lVert \hat\psi_L   \rVert_{\equi_{\delta,Z}(L)}^2 + \frac{2 D_3^\delta M}{\tilde D_1^\delta} \lVert \hat\zeta_L \rVert_{\Lambda_{3L}}^2
\geq
\nonumber \\
C_1^\delta (\gamma_1^\delta)^{-1/\gamma_1^\delta} (M \lVert \hat\psi_L  \rVert_{\Lambda_{3L}}^2 +D_3^\delta M \lVert \hat \zeta_L \rVert_{\Lambda_{3L}}^2)^{1-1/\gamma_1^\delta}
\biggl(C_2 (\gamma_2)^{-1}
9^{-d\gamma_2^{-m+1}}
\left( \lVert \hat \psi_L \rVert_{\Lambda_{9L}}^2 +  {D_3^\delta } \lVert \hat\zeta_L\rVert_{\Lambda_{9L}}^2 \right)\biggr)^{1/\gamma_1^\delta}
 \\
=
C_1^\delta (\gamma_1^\delta)^{-1/\gamma_1^\delta}
M^{1-1/\gamma_1^\delta}
C_2 (\gamma_2)^{-1/\gamma_1^\delta} 
\biggl(9^{-d\gamma_2^{-m+1}} 9^d \biggr)^{1/\gamma_1^\delta}
( \lVert \hat\psi_L  \rVert_{\Lambda_{L}}^2 +D_3^\delta \lVert \hat \zeta_L \rVert_{\Lambda_{L}}^2)
\end{multline*}

Note that on  $ \Lambda_L\supset \equi_{\delta,Z}(L)$ the extension  $\hat\psi_L$ coincides with $\psi_L$.
The stated bounds on the constants are already given in the proof of Theorem~\ref{thm:sampling}, 
except for the new factor $ (9^{-d\gamma_2^{-m+1}} 9^d)^{1/\gamma_1^\delta}$. 
Since $m$ and $ \gamma_2$ depend only on
$d$, $\Elliptic$, $\Lipschitz$ this factor is of the form $ K^{1/\gamma_1^\delta}$.
Hence as in
\eqref{eq:nur-oben} and \eqref{eq:nur-oben2} we obtain
\[
(9^{-d\gamma_2^{-m+1}} 9^d)^{1/\gamma_1^\delta} \leq \delta^{-\tilde K}
\]
with $\tilde K$ depending only on  $d$, $\Elliptic$, $\Lipschitz$.
\end{proof}
\appendix
%
%
%
%
%
%
%
%
%

\section{Extension of the differential inequality}
\label{sec:extensions}
In this appendix we complement the proof of Theorem~\ref{thm:equidistribution} and explain how to extend $\psi_L$, $V_L$, and the coefficients of the operator $H_L$. We start with a slightly simpler example.
\par
Let $\Omega_- = (-1,0) \times (0,1)^{d-1}$ and consider the form $a_- : \HNullEins (\Omega_-) \times \HNullEins (\Omega_-) \to \CC$ given by
\[
 a_- (u,v) = \int_{\Omega_-} \left( \sum_{i,j=1}^d a_-^{ij} \partial_i u \overline{\partial_j v} + \sum_{i=1}^d b_-^i \partial_i u \overline{v} + c_- u \overline{v} \right) \drm x ,
\]
where $A_- : {\Omega_-} \to \RR^{d \times d}$ with $A_-= (a_-^{ij})_{i,j=1}^d$, $b_- : {\Omega_-} \to \CC^d$, and $c_- \colon {\Omega_-} \to \CC$. We assume that $a_-^{ij} \equiv a_-^{ji}$ for all $i,j \in \{1,\ldots , d\}$, and that there are constants $\Elliptic \geq 1$ and $\Lipschitz \geq 0$ such that for
all $x,y \in {\Omega_-}$ and all $\xi \in \RR^d$ we have
\begin{equation*} 
\Elliptic^{-1} \lvert \xi \rvert^2 \leq \xi^\T A_- (x) \xi \leq \Elliptic \lvert \xi \rvert^2
\quad\text{and}\quad
\lVert A_-(x) - A_-(y) \rVert_\infty \leq \Lipschitz \lvert x-y \rvert .
\end{equation*}
Moreover, we assume that $b_-,c_- \in L^\infty ({\Omega_-})$.
The form $a_-$ is densely defined, closed, and sectorial. We denote $H_-$ the m-sectorial operator associated with the form $a_-$, and its domain by $\Dom (H_-)$. Note that $H_-$ is the Friedrichs extension of the operator $\Op_- : C_{\mathrm{c}}^\infty (\Omega_-) \to L^2 (\Omega_-)$,
 \[
\Op_- u := - \diver (A_- \nabla u) + b_-^\T \nabla u + c_- u = -\sum_{i,j=1}^d \partial_i \left( a_-^{ij} \partial_j u \right) + \sum_{i=1}^d b_-^i \partial_i u + c_- u .
\]
Fix $\psi_- \in \Dom (H_-)$, $V_- \in L^\infty (\Omega_-)$ real-valued and $\zeta_- \in L^2 (\Omega_-)$ such that $\lvert H_- \psi_- \rvert \leq \lvert V_- \psi_- \rvert + \lvert \zeta_- \rvert$ almost everywhere on $\Omega_-$.
\par
Let $\Omega = \operatorname{int} (\overline\Omega_- \cup \overline\Omega_+)$, where $\Omega_+ = (0,1) \times (0,1)^{d-1}$. We now explain how to extend the functions $\psi_-$ and $\zeta_-$, and the coefficients of the operator to the set $\Omega$.
Since the coefficients $a_-^{ij}$, $i,j\in\{1,\ldots d\}$ obey a Lipschitz condition on $\Omega_-$
by assumption, they are pointwise well defined,
  and extend in a unique way to continuous functions  $a_-^{ij} : \tilde\Omega_- = (-1,0]\times (0,1)^{d-1} \to \RR$, $i,j\in\{1,\ldots d\}$, which will be denoted by the same symbol.
We assume that
\begin{enumerate}[(Dir'')]
  \item the coefficients $a_-^{1k} = a_-^{k1}$ vanish on $\tilde\Omega_- \setminus \Omega_-$ for all $k \in \{2,\ldots , d\}$.
\end{enumerate}
We first extend $b_-, c_-, V_-,\zeta_-,\psi_-$ to  $\tilde\Omega_-$ by setting their value on the interface $\tilde\Omega_- \setminus \Omega_-$  equal to zero.
 We extend the function $\psi_-$ from $\Omega_-$ to $\Omega$ by antisymmetric reflection with respect to the boundary $\tilde\Omega_- \setminus \Omega_-$, and denote the extended function by $\psi_\Omega \in L^2 (\Omega)$. By antisymmetric reflection we mean that $\psi_\Omega = \psi_- $ on $\tilde\Omega_-$, and
\[
 \psi_\Omega (x) = - \psi_- (x - 2 x_1 e_1)
\]
for $x \in \Omega_+$. The extended coefficient functions are defined by
\[
  a_\Omega^{ij} (x) = a_-^{ij} (x), \quad b_\Omega^i (x) = b_-^i (x), \quad c_\Omega (x) = c_- (x), \quad V_\Omega (x) = V_- (x) ,
\]
if $x \in \tilde\Omega_-$, and extended by the rule
\begin{align*}
&a_\Omega^{kk}(x)= a_-^{kk} (x-2 x_1 e_1 ) & \text{for $k \in \{1,\ldots , d\}$},
\\
&a_\Omega^{kj} (x) = a_\Omega^{jk}(x) = a_-^{kj} (x -2 x_1 e_1) & \text{for $k,j \in \{2,\ldots , d\}$ with $k \not = j$}, \\
&a_\Omega^{1k} (x) = a_\Omega^{k1} (x) = -a_-^{1k} (x - 2 x_1 e_1 )  & \text{for $k \in \{2,\ldots , d\}$}, \\
& b_\Omega^i(x)=  b_-^i ( x -2 x_1 e_1 ) & \text{for $i \in \{2,\ldots , d\}$}, \\
& b_\Omega^1(x)= -b_-^1 ( x -2 x_1 e_1 ) & \\
& c_\Omega (x)=c_-(x -2 x_1 e_1), & \\
& V_\Omega (x)=V_-(x -2 x_1 e_1), & \\
& \zeta_\Omega (x)=\zeta_-(x -2 x_1 e_1), &
\end{align*}
if $x \in \Omega_+$.
We use the notation $A_\Omega = (a_\Omega^{ij})_{i,j=1}^d$ and $b_\Omega = (b_\Omega^i)_{i=1}^d$.
We denote by $a_\Omega : \HNullEins (\Omega) \times \HNullEins (\Omega) \to \CC$ the form given by
\[
 a_\Omega (u,v) = \int_{\Omega} \left( \sum_{i,j=1}^d a_\Omega^{ij} \partial_i u \overline{\partial_j v} + \sum_{i=1}^d b_\Omega^i \partial_i u \overline{v} + c_\Omega u \overline{v} \right) \drm x ,
\]
and by $H_\Omega$ the associated m-sectorial operator with domain $\Dom (H_\Omega)$. Note that $H_\Omega$ is the Friedrichs extension of the operator
 $\Op_\Omega : C_{\mathrm{c}}^\infty (\Omega) \to L^2 (\Omega)$,
 \[
\Op_\Omega u := - \diver (A_\Omega \nabla u) + b_\Omega^\T \nabla u + c_\Omega u = -\sum_{i,j=1}^d \partial_i \left( a_\Omega^{ij} \partial_j u \right) + \sum_{i=1}^d b_\Omega^i \partial_i u + c_\Omega u .
\]
\begin{lemma}\label{l:extension}
\hspace{2em}
\begin{enumerate}[(i)]
 \item Let Assumption (Dir'') be satisfied. Then for all $x,y \in {\Omega}$ and all $\xi \in \RR^d$ we have
\begin{equation*}
\Elliptic^{-1} \lvert \xi \rvert^2 \leq \xi^\T A_\Omega (x) \xi \leq \Elliptic \lvert \xi \rvert^2
\quad\text{and}\quad
\lVert A_\Omega (x) - A_\Omega (y) \rVert_\infty \leq \Lipschitz \lvert x-y \rvert .
\end{equation*}
\item We have $\psi_\Omega \in \Dom (H_\Omega)$ and
 \begin{equation*}
  (H_\Omega \psi_\Omega)(x) = \begin{cases}
                (H_- \psi_-)(x) & \text{for $x \in \Omega_-$}, \\
                -(H_- \psi_-)(x-2x_1 e_1) & \text{for $x \in \Omega_+$} .
               \end{cases}
 \end{equation*}
 Hence we have $\lvert H_\Omega \psi_\Omega \rvert \leq \lvert V_\Omega \psi_\Omega \rvert + \lvert \zeta_\Omega \rvert$ almost everywhere on $\Omega$.
\end{enumerate}
\end{lemma}
\begin{proof}
Recall that by assumption we have for all $x_0,y_0 \in \tilde\Omega_- = (-1,0]\times (0,1)^{d-1}$
 \begin{equation*}
 \lVert A_\Omega (x_0) - A_\Omega (y_0) \rVert_\infty =\lVert A_- (x_0) - A_- (y_0) \rVert_\infty  \leq \Lipschitz \lvert x_0-y_0 \rvert.
 \end{equation*}
Moreover, we have for all $x_0 \in \tilde\Omega_-$ and $\xi \in \RR^d$ that
\begin{equation*}
 \Elliptic^{-1} \lvert \xi \rvert^2 \leq
 \xi^\T A_- (x_0) \xi  = \xi^\T A_\Omega (x_0) \xi \leq \Elliptic \lvert \xi \rvert^2 .
\end{equation*}
By the definition of the extensions and assumption (Dir'') we have for all $x,z\in \tilde\Omega_+ = [0,1)\times (0,1)^{d-1}$
 \begin{equation*}
\lVert A_\Omega (x) - A_\Omega (z) \rVert_\infty  \leq \Lipschitz \lvert x-z \rvert.
 \end{equation*}
Let now $x \in \Omega_-$, $y \in \Omega_+$, $T := \{x + s (y-x) \colon s \in [0,1]\}$, and choose $z \in \tilde\Omega_- \cap \tilde\Omega_+ \cap T$. Then we have
\begin{align*}
 \lVert A_\Omega (x) - A_\Omega (y) \rVert_\infty &\leq \lVert A_\Omega (x) - A_\Omega (z) \rVert_\infty + \lVert A_\Omega (z) - A_\Omega (y) \rVert_\infty \\
 &\leq \Lipschitz \left(  \lvert x - z \rvert + \lvert z - y \rvert \right) = \Lipschitz \lvert x-y \rvert .
\end{align*}
This shows the Lipschitz continuity in part (i) of the lemma.
Let now $x\in \tilde\Omega_+$. Then there exists a point $x_0\in\tilde\Omega_-$ such that $A_\Omega(x)= \tilde{A}_-(x_0)$, where $\tilde{A}_-(x_0)$ is the matrix obtained from $A_-(x_0)$ by multiplying the 1st column and the 1st row by minus one.
This corresponds to conjugation with a diagonal unitary matrix.
Consequently, the eigenvalues of $\tilde{A}_-(x_0)$ and $A_\Omega(x)$ coincide, which implies the validity of the ellipticity condition from part (i).
\par
 We now turn to the proof of part (ii). First we show that $\psi_\Omega \in \HNullEins (\Omega)$. To this end let $(\psi_n)_{n \in \NN}$ be a sequence in $C_{\mathrm{c}}^\infty (\Omega_-)$ such that $\psi_n \to \psi_-$ in $\HEins (\Omega_-)$. We denote by $\hat\psi_n$ the function on $\Omega$ obtained by antisymmetric reflection of $\psi_n$ with respect to $\{0\} \times (0,1)^{d-1}$.
 Then we have for all $m,l \in \NN$
 \[
  \lVert \hat\psi_m - \hat\psi_l \rVert_{\HEins (\Omega)}
  =
  2\lVert \psi_m - \psi_l \rVert_{\HEins (\Omega_-)} .
 \]
 Hence, $(\hat\psi_n)_{n \in \NN}$ is a Cauchy sequence in $\HEins (\Omega)$ of compactly supported functions. We denote its limit by $\overline \psi \in \HNullEins (\Omega)$. Since we have
 \[
  \lVert \overline \psi - \psi_\Omega \rVert_{L^2 (\Omega)}
  \leq
  \lVert \overline \psi - \hat\psi_n \rVert_{L^2 (\Omega)} +
  \lVert \psi_\Omega - \hat\psi_n \rVert_{L^2 (\Omega)}
  =
  \lVert \overline \psi - \hat\psi_n \rVert_{L^2 (\Omega)} +
  2\lVert \psi_- - \psi_n \rVert_{L^2 (\Omega_-)},
 \]
 we find $\overline \psi = \psi_\Omega$, and hence $\psi_\Omega \in \HNullEins (\Omega)$.
 \par
By definition of the extensions we have for $x \in \Omega_+$
 \[
  (\partial_i \psi_\Omega)(x) = \begin{cases}
                           (\partial_1 \psi_-) (x-2 x_1 e_1) & \text{if $i = 1$}, \\
                            - (\partial_i \psi_-) (x-2 x_1 e_1) & \text{if $i \in \{2,\ldots , d\}$}.
                           \end{cases}
 \]
Choose now a test function $\phi \in C_{\mathrm{c}}^\infty (\Omega)$, and define
for $x \in \Omega_-$ the function $\phi_- (x) = \phi (x - 2x_1 e_1)$. Then $\partial_1 \phi_- (x) = -(\partial_1 \phi)(x - 2x_1 e_1)$ and $\partial_j \phi_- (x) = (\partial_j \phi)(x - 2x_1 e_1)$ for $j \in \{2,\ldots , d\}$.
We obtain by substitution
 \begin{align*}
  \int_{\Omega}  \nabla \psi_\Omega^\T A_\Omega \nabla \overline\phi
  &= \int_{\Omega_-} \nabla \psi_\Omega^\T A_\Omega \nabla \overline\phi
  + \sum_{i,j=1}^d \int_{\Omega_+} (\partial_i\psi_\Omega) a_\Omega^{ij} (\partial_j \overline\phi) \\
  &= \int_{\Omega_-} \nabla \psi_-^\T A_- \nabla \overline\phi
  - \sum_{i,j=1}^d \int_{\Omega_-} (\partial_i\psi_-) a_-^{ij} (\partial_j \overline\phi_-)
 \end{align*}
and
\begin{align*}
 \int_\Omega b_\Omega^\T \nabla \psi_\Omega \overline\phi
 &=
 \int_{\Omega_-} b_\Omega^\T \nabla \psi_\Omega \overline\phi
 +
 \sum_{i=1}^d \int_{\Omega_+} b_\Omega^i (\partial_i \psi_\Omega) \overline\phi
 =
 \int_{\Omega_-} b_\Omega^\T \nabla \psi_\Omega \overline\phi
 -
 \sum_{i=1}^d \int_{\Omega_-} b_-^i (\partial_i \psi_-) \overline\phi_- .
\end{align*}
Hence, we obtain
\begin{align*}
 a_\Omega (\psi_\Omega , \phi) &= \int_{\Omega} \left( \nabla \psi_\Omega^\T A_\Omega \nabla \overline\phi + b_\Omega^\T \nabla \psi_\Omega \overline\phi + c_\Omega \psi_\Omega \overline\phi \right) \drm x  \\
 &= \int_{\Omega_-} \nabla \psi_-^\T A_- \nabla (\overline\phi - \overline\phi_-)
 +
 \int_{\Omega_-} b_-^\T \nabla \psi_- (\overline\phi - \overline\phi_-)
 +
 \int_{\Omega_-} c_- \psi_- (\overline\phi - \overline\phi_-) .
\end{align*}
For $x \in \Omega_-$ we use the notation $\tilde\phi_- (x) = \phi (x) - \phi_- (x)$. Note that $\tilde\phi_- \in C^\infty (\Omega_-)$ and $\tilde\phi_- (x) = 0$ for $x \in \partial\Omega_-$. Hence $\tilde\phi_- \in \HNullEins (\Omega_-)$, see e.g.\ \cite[Theorems~A6.6 and A6.10]{Alt-06}.
We obtain by the first representation theorem for quadratic forms and substitution
\[
 a_\Omega (\psi_\Omega , \phi)
 =
 a_- (\psi_- , \tilde\phi_-)
 =
\int_{\Omega_-} (H_- \psi_-) \overline{\tilde \phi}_-
 = \int_{\Omega_-} H_- \psi_- \overline\phi - \int_{\Omega_+} (H_- \psi_-) (x-2x_1 e_1) \overline\phi (x) .
\]
We have shown that
\[
 a_\Omega (\psi_\Omega , \phi) = \langle \tilde\psi  , \phi \rangle ,
 \quad\text{where}\quad
 \tilde \psi (x) = \begin{cases}
                (H_- \psi_-)(x) & \text{for $x \in \Omega_-$}, \\
                -(H_- \psi_-)(x-2x_1 e_1) & \text{for $x \in \Omega_+$} .
               \end{cases}
\]
Hence, by the first representation theorem we find $\psi_\Omega \in \Dom (H_\Omega)$ and $H_\Omega \psi_\Omega = \tilde \psi$.
\end{proof}

%
Now we explain how to obtain the extensions needed in the proof of Theorem~\ref{thm:equidistribution}.
by applying Lemma~\ref{lemma:extensions} iteratively.
Fix $\psi_L \in \Dom (H_L)$, $V_L \in L^\infty (\Lambda_L)$ real-valued and $\zeta_L \in L^2 (\Lambda_L)$, such that $\lvert H_L \psi_L \rvert \leq \lvert V_L \psi_L \rvert + \lvert \zeta_L \rvert$ almost everywhere on $\Lambda_L$.
We recall that the coefficients $A_L$ are extended continuously to $\overline{\Lambda_L}$.
For $b_L$, $c_L$, $\psi_L$, $\zeta_L$, and $V_L$, we define them on $\overline{\Lambda_L}$ by setting their value at the boundary to zero.
The proof of Theorem~\ref{thm:equidistribution} requires extensions of $\psi_L$, $A_L$, $b_L$, $c_L$, $V_L$ and $\zeta_L$ to $\Lambda_{RL}$
satisfying the properties spelled out below.
We will denote the extensions by $\hat\psi_L$, $\hat A_L$, $\hat b_L$, $\hat c_L$, $\hat V_L$, and $\hat \zeta_L$.

\par
We now proceed iteratively to define them on $\Lambda_{RL}$.
Recall that $R$ is a sufficiently large integer power of three.
In a first step we extend $\psi_L : \Lambda_L \to \CC$ to the set
 $\{x \in \Lambda_{3L} \colon  x_i \in (-L/2 , L/2), \ i \in \{2,\ldots , d\}\}$ by requiring $\hat\psi_L=\psi_L$ on $\Lambda_L$, and
\[
\hat \psi_L (x \pm L e_1) = -\psi_L (x -2 x_1 e_1)
\]
for almost all $x \in \Lambda_L$. Now we iteratively extend $\psi_L$ in the remaining $d-1$ directions using the same procedure and obtain a function $\hat \psi_L : \Lambda_{3L} \to \CC$. Iterating this procedure we obtain a function $\hat\psi_L : \Lambda_{RL} \to \CC$. Let us note that this is equivalent to require for the extension $\hat\psi_L$ that $\hat\psi (x) = \psi_L (x)$ for almost all $x \in \Lambda_L$, and
\[
 \hat\psi (x\pm L{e}_k ) = - \hat\psi (x+2(\gamma_k - x_k)e_k)
\]
for all $\gamma\in (L\ZZ)^d \cap \Lambda_{RL}$, almost all $x\in \Lambda_L(\gamma)$, and all $k\in\{1,\ldots, d\}$, as long as $x+2(\gamma_k - x_k)e_k \in \Lambda_{RL}$ and $x \pm L e_k \in \Lambda_{RL}$.
The extended coefficient functions are defined in an analogous way by
\[
  \hat a^{ij}_L (x) = a_L^{ij} (x), \quad \hat b^i_L (x) = b_{L}^i (x), \quad \hat c_L (x) = c_L (x), \quad \hat V_L (x) = V_L (x) , \quad \hat \zeta_L (x) = \zeta_L (x) ,
\]
on $\overline{\Lambda_L}$, and extended by requiring
\begin{align*}
&\hat a^{ij}_L\big(x\pm Le_k\big)=\hat a^{ij}_L\big(x+2(\gamma_k - x_k)e_k\big)
& \text{if}\ i\not=k\ \text{and} \ j\not=k, \\
&\hat a^{kk}_L\big(x\pm Le_k\big)=\hat a^{kk}_L\big(x+2(\gamma_k - x_k)e_k\big),
& \\
&\hat a^{kj}_L\big(x\pm Le_k\big)=\hat a^{jk}_L\big(x\pm Le_k\big)= -\hat a^{kj}_L\big(x+2(\gamma_k - x_k)e_k\big)
& \text{if}\ k\not=j, \\
& \hat b^i_L\big(x\pm L{e}_k\big)=\hat b^i_L\big(x+2(\gamma_k - x_k)e_k\big)
& \text{if}\  i\not=k, \\
&\hat b^k_L\big(x\pm L{e}_k\big)=-\hat b^k_L\big(x+2(\gamma_k - x_k){e}_k\big),
& \\
& \hat c_L\big(x\pm Le_k\big)=\hat c_L\big(x+2(\gamma_k - x_k)e_k\big),
& \\
&\hat V_L\big(x\pm L{e}_k\big)=\hat V_L\big(x+2(\gamma_k - x_k){e}_k\big) ,
& \\
&\hat \zeta_L\big(x\pm L{e}_k\big)=\hat \zeta_L\big(x+2(\gamma_k - x_k){e}_k\big) ,
&
\end{align*}
for all $\gamma\in (L\ZZ)^d$, $x\in \Lambda_L(\gamma)$ and $i,j,k\in\{1,\ldots, d\}$, as long as $x+2(\gamma_k - x_k)e_k \in \Lambda_{RL}$ and $x \pm L e_k \in \Lambda_{RL}$. On the boundaries of $\Lambda_L (\gamma)$ the coefficients $\hat a_L^{ij}$ are continuously extended, while all the other coefficients are set to zero.
Recall from the proof of Theorem~\ref{thm:equidistribution} that we denote by $\hat a_{L} : \HNullEins (\Lambda_{RL}) \times \HNullEins (\Lambda_{RL}) \to \CC$ the densely defined, closed, and sectorial form
\[
 \hat a_L (u,v)  = \int_{\Lambda_{RL}} \left( \sum_{i,j=1}^d \hat a_L^{ij} \partial_i u \overline{\partial_j v} + \sum_{i=1}^d \hat b_L^i \partial_i u \overline{v} + \hat c_L u \overline{v} \right) \drm x ,
\]
and by $\hat H_L$ the m-sectorial operator associated with the form $\hat a_L$ with domain $\Dom (\hat H_L)$. By an iterative application of Lemma~\ref{lemma:extensions} we immediately obtain
\begin{lemma}
\begin{enumerate}[(i)]
\item Let Assumption (Dir) be satisfied. Then for all $x,y \in \Lambda_{RL}$ and all $\xi \in \RR^d$ we have
\begin{equation*} 
\Elliptic^{-1} \lvert \xi \rvert^2 \leq \xi^\T \hat A_L (x) \xi \leq \Elliptic \lvert \xi \rvert^2
\quad\text{and}\quad
\lVert \hat A_L (x) - \hat A_L (y) \rVert_\infty \leq \Lipschitz \lvert x-y \rvert .
\end{equation*}
\item We have $\hat \psi_L \in \Dom (\hat H_L)$ and
  $\lvert \hat H_L  \hat \psi_L \rvert \leq \lvert \hat V_L \hat \psi_L \rvert + \lvert \hat \zeta_L \rvert$ almost everywhere on $\Lambda_{RL}$.
 \end{enumerate}
 \label{lemma:extensions}
\end{lemma}

In a completely analogous way the coefficient functions of the operator $\Op_L$ can be extended to functions on the whole
of $\RR^d$ satisfying in particular
\begin{equation*}
\forall x,y,\xi \in \RR^d : \quad
\Elliptic^{-1} \lvert \xi \rvert^2 \leq \xi^\T \hat A(x) \xi \leq \Elliptic \lvert \xi \rvert^2
\quad\text{and}\quad
\lVert \hat A(x) - \hat A(y) \rVert_\infty \leq \Lipschitz \lvert x-y \rvert .
\end{equation*}
This gives rise to an elliptic operator $\hat\Op\colon C_{\mathrm{c}}^\infty (\RR^d) \to L^2 (\RR^d)$
whose Friedrichs extension is denoted by $\hat H$ and used in the approximation argument of the proof
of Theorem~\ref{thm:equidistribution}.
\paragraph*{Acknowledgments}
The authors gratefully acknowledge stimulating discussions with Alexander Dicke, Michela Egidi, Albrecht Seelmann, and Christian Seifert which helped to improve the manuscript.
This research was partially supported by the Deutsche Forschungsgemeinschaft e.V.
through the grant Ve 253/6 devoted to the topic 'Unique continuation principles and equidistribution properties of eigenfunctions'.

%
%
%
%
%
%
%
%
\newcommand{\etalchar}[1]{$^{#1}$}

\end{document}